\documentclass[11pt, oneside]{amsart}   	% use "amsart" instead of "article" for AMSLaTeX format
\usepackage{geometry}                		% See geometry.pdf to learn the layout options. There are lots.
\usepackage{graphicx}				% Use pdf, png, jpg, or eps§ with pdflatex; use eps in DVI mode
\usepackage{amssymb}
\usepackage{color}
\usepackage{amsmath}
\usepackage{amsthm}
\usepackage{appendix}
\usepackage{hyperref}

\allowdisplaybreaks
\numberwithin{equation}{section}

\newcommand{\R}{\mathbb{R}}

\newcommand{\N}{\mathbb{N}}
\renewcommand{\S}{\mathbb{S}}
\renewcommand{\div}{\operatorname{div}}

\newcommand{\Ric}{\operatorname{Ric}}
\newcommand{\Rm}{\operatorname{Rm}}

\newcommand{\tr}{\operatorname{tr}}
\newcommand{\supp}{\operatorname{Supp}}

\newcommand{\vol}{\operatorname{vol}}
\newcommand{\id}{\operatorname{id}}
\newcommand{\loc}{\text{loc}}

\renewcommand{\L}{\mathcal{L}}

\newcommand{\dist}{\operatorname{dist}}

\theoremstyle{plain}
\newtheorem{theorem*}{Theorem}
\newtheorem{corollary*}{Corollary}
\newtheorem{question*}{Question}
\newtheorem{definition*}{Definition}
\newtheorem{claim*}{Claim}
\newtheorem{theorem}{Theorem}[section]
\newtheorem{lemma}[theorem]{Lemma}
\newtheorem{corollary}[theorem]{Corollary}

\newtheorem{proposition}[theorem]{Proposition}
\theoremstyle{definition}
\newtheorem{definition}[theorem]{Definition}
\theoremstyle{remark}
\newtheorem{remark}[theorem]{Remark}

\title{ADM mass for $C^0$ metrics and distortion under Ricci-DeTurck flow}
\author{Paula Burkhardt-Guim}
\address{Courant Institute of Mathematical Sciences \\ New York University}
\email{pbguim@cims.nyu.edu}
\begin{document}
\begin{abstract}
We show that there exists a quantity, depending only on $C^0$ data of a Riemannian metric, that agrees with the usual ADM mass at infinity whenever the ADM mass exists, but has a well-defined limit at infinity for any continuous Riemannian metric that is asymptotically flat in the $C^0$ sense and has nonnegative scalar curvature in the sense of Ricci flow. Moreover, the $C^0$ mass at infinity is independent of choice of $C^0$-asymptotically flat coordinate chart, and the $C^0$ local mass has controlled distortion under Ricci-DeTurck flow when coupled with a suitably evolving test function.
\end{abstract}
\maketitle

\tableofcontents

\section{Introduction}\label{sec:introduction}

In recent years considerable attention has been devoted to the study of Riemannian metrics with lower scalar curvature bounds in various nonsmooth settings (see, for instance, \cite{Gro14}, \cite{LeeLeFloch14}, \cite{LeFlochSormani14}, \cite{Bam16}, \cite{ParkTianWang18}, \cite{Jauregui20A} among others). In the $C^0$ setting, Gromov \cite{Gro14} showed that pointwise lower scalar curvature bounds are preserved under uniform convergence. Somewhat more recently, the Ricci and Ricci-DeTurck flows have emerged as useful tools in this setting, since they provide a smoothing of the metric under which the scalar curvature has a well-behaved evolution equation (a Ricci-DeTurck flow is a parabolic flow that is related to a Ricci flow via pullback by a family of diffeomorphisms; the precise definition is given in Section \ref{sec:preliminaries}). For instance, in \cite{Bam16}, Bamler provided a Ricci flow proof of Gromov's \cite{Gro14} result (see also \cite{JiangShengZhang21}, \cite{LammSimon21}, \cite{HuangLee21}, \cite{ChuLee22}). In \cite{PBG19} (see also \cite{PBG20}, \cite{PBGthesis}) the author introduced a synthetic notion of pointwise lower scalar curvature bounds for $C^0$ Riemannian metrics using Ricci flow.

In light of this context, it is natural to ask whether other metric quantities associated with the scalar curvature may be formulated using only $C^0$ data of the metric, and whether anything can be learned about these quantities by letting the metric evolve by Ricci or Ricci-DeTurck flow. One such quantity is the ADM mass. Recall that if $(M^n, g)$ is a smooth Riemannian manifold and $\Phi: M\setminus K\to \R^n\setminus \overline{B(0,1)}$ is a smooth coordinate chart for $M$, where $K$ is a compact subset of $M$, then the ADM mass (introduced in \cite{ArnowittDeserMisner61}) is given by (\cite[p. 143]{SchoenMontecatini}):
\begin{equation}\label{eq:classicalADMmassdef}
m_{ADM}(g) := \lim_{r\to \infty}\frac{1}{4\pi (n-1)\omega_{n-1}}\int_{\S(r)} \sum_{i=1}^{n}(\partial_i g_{ij} - \partial_j g_{ii})\nu^jdS,
\end{equation}
where the coordinate expression in the integrand corresponds to the coordinates $\Phi$, 
\begin{equation*}\S(r) = \{x\in \R^n : (x^1)^2 + \cdots + (x^n)^2 = r^2\},
\end{equation*}
 $\nu$ denotes the outward unit normal to $\S(r)$ with respect to the Euclidean metric, $\omega_{n-1}$ denotes the Euclidean volume of the $(n-1)$-dimensional unit sphere, and $dS$ denotes the Euclidean surface measure on $\S(r)$. Henceforth, if we wish to emphasize the coordinate chart $\Phi$, then we write $m_{ADM}(g, \Phi)$. 

A priori it is not clear whether the limit (\ref{eq:classicalADMmassdef}) should always exist, or whether the limit depends on the choice of $\Phi$, but Bartnik \cite[Theorems $4.2$ and $4.3$]{Bartnik86} (see also \cite{Chrusciel88} for the asymptotically Minkowski case) showed that under certain conditions the ADM mass does indeed exist, is finite, and is independent of choice of coordinate chart:
\begin{theorem}[cf. \cite{Bartnik86}]\label{thm:classicalADMmassexistence}
Let $(M^n,g)$ be a smooth Riemannian manifold. Suppose that for some compact set $K\subset M$ there exists a coordinate chart $\Phi: M\setminus K\to \R^n\setminus \overline{B(0,1)}$ for $M$ such that, for some $\tau > \tfrac{1}{2}(n-2)$,
we have
\begin{equation}\label{eq:C0AFdecay}
|(\Phi_*g)_{ij} - \delta_{ij}|\big|_x = O(|x|_{\delta}^{-\tau}),
\end{equation}
\begin{equation}\label{eq:CkAFdecay}
|\partial_k (\Phi_*g)_{ij}|\big|_x = O(|x|_{\delta}^{-\tau - |k|}) \text{ for } |k|= 1,2,
\end{equation}
where $\delta$ denotes the Euclidean metric and $k$ is a multiindex, and
\begin{equation}\label{eq:RL1}
\int_M |R(g)| < \infty,
\end{equation}
where $R(g)$ denotes the scalar curvature of $g$. Then the limit from (\ref{eq:classicalADMmassdef}) with respect to $\Phi$ exists, is finite, and is independent of choice of $\Phi$ satisfying (\ref{eq:C0AFdecay}) and (\ref{eq:CkAFdecay}).
\end{theorem}
Henceforth we will say that a continuous Riemannian metric is \emph{$C^0$-asymptotically flat} if, for some smooth coordinate chart, it satisfies (\ref{eq:C0AFdecay}) but not necessarily (\ref{eq:CkAFdecay}). As noted above, the ADM mass is a metric quantity that is associated with lower scalar curvature bounds; this is because of the Riemannian Positive Mass Theorem:
\begin{theorem}[cf. \cite{SchoenYau79}, \cite{SchoenYau81}, \cite{Witten81}, \cite{SchoenYau19}]\label{thm:PMT}
For $n\geq 3$ let $(M^n, g)$ be a smooth Riemannian manifold and suppose that there exists a compact set $K\subset M$ and a coordinate chart $\Phi: M\setminus K\to \R^n\setminus \overline{B(0,1)}$ for $M$ such that, for some $\tau > \tfrac{1}{2}(n-2)$, (\ref{eq:C0AFdecay}), (\ref{eq:CkAFdecay}), and (\ref{eq:RL1}) hold. If $R(g)\geq 0$ then $m_{ADM}(g) \geq 0$. Moreover, $m_{ADM}(g) = 0$ if and only if $g$ is flat.
\end{theorem}
Here we have stated the result for a manifold with a single asymptotically flat end, but the result also holds for multiple ends; see \cite{SchoenYau81}, \cite{Witten81}. There have been a number of proofs of the Positive Mass Theorem using different techniques, the earliest of which were due to Schoen -- Yau \cite{SchoenYau79}, \cite{SchoenYau81} and Witten \cite{Witten81}. We refer the reader to \cite[Section $3.1$]{Bray11} for a more thorough discussion of the various proofs and their techniques. Moreover, we remark that the Ricci and Ricci-DeTurck flows have already emerged as useful tools in the context of the Riemannian Positive Mass Theorem; see, for instance, \cite{YLi16}, \cite{McFeronSzekelyhidi12}, and \cite{ChuLee22}. In view of Theorem \ref{thm:PMT} we often impose the condition that the Riemannian metric have nonnegative scalar curvature (in a generalized sense) throughout the rest of this paper. 

Several compelling notions of $C^0$ masses, involving volumes and capacities, have been introduced; see \cite{JaureguiLee16} and \cite{Jauregui20}. In this paper we will take a different approach from these works, with the intention of evolving the metric by Ricci-DeTurck flow in order to show the existence of a limit at infinity even for metrics that have only $C^0$ control. Towards the $C^0$ setting, observe that $m_{ADM}(g, \Phi)$ is computed by integrating over a single coordinate sphere, but when the limit $m_{ADM}(g, \Phi)$ exists, one may alternatively compute $m_{ADM}(g, \Phi)$ by integrating over a family of spheres weighted by some test function, since, if $\varphi: \R \to \R$ is any smooth function with $\int_{.9}^{1.1}\varphi(\ell)d\ell \neq 0$, then
\begin{equation}\label{eq:computingmADMviavarphi}
\begin{split}
&\frac{\int_{.9r}^{1.1r}\varphi(\tfrac{\ell}{r})\int_{\S(\ell)}\sum_{i=1}^n (\partial_i g_{ij} - \partial_j g_{ii})\nu^jdS  d\ell}{4\pi(n-1)\omega_{n-1} r\int_{.9}^{1.1}\varphi(\ell)d\ell}
%%%
\\&= \frac{\int_{.9}^{1.1}\varphi(\ell)\int_{\S(\ell r)}\sum_{i=1}^n (\partial_i g_{ij} - \partial_j g_{ii})\nu^jdS  d\ell}{4\pi(n-1)\omega_{n-1}\int_{.9}^{1.1}\varphi(\ell)d\ell} 
\xrightarrow[r\to\infty]{} m_{ADM}(g, \Phi).
\end{split}
\end{equation}

In fact, the observation (\ref{eq:computingmADMviavarphi}) is an entry point to proving a $C^0$ version of the Positive Mass Theorem: the significance of the left-hand side of (\ref{eq:computingmADMviavarphi}) is that it may be expressed solely in terms of the $C^0$ data of $g$; see Definition \ref{def:C0mass}. In order to get a well-defined limit at infinity we replace $\varphi$ with a family of functions $\varphi^r$ that vary with $r$. We will explain how $\varphi^r$ relates to $\varphi$ in Section \ref{sec:technicalintro}. We summarize this fact in the following theorem; a more precise statement is given in the next section (see Theorem \ref{thm:fullC0existence}).
\begin{theorem}\label{thm:vagueC0existence}
Let $M$ be a smooth manifold, and $g$ a continuous Riemannian metric on $M$. Suppose there is a smooth coordinate chart $\Phi: M\setminus K\to \R^n\setminus\overline{B(0,1)}$ for $M$, where $K$ is some compact set. For any smooth cutoff function $\varphi: \R \to \R^{\geq 0}$ with $\supp(\varphi)\subset\subset(.9, 1.1)$ and for any $r>0$, there exists a smooth family of functions $(\varphi^r)_{r>0}: \R\to \R$ such that $\varphi^r\xrightarrow[r\to\infty]{C^\infty} \varphi$, and there exists a quantity $M_{C^0}(g, \Phi, \varphi^r, r)$, depending on only the $C^0$ data of $\Phi_*g$, for which the following is true:
\begin{enumerate}
\item If $g$ is $C^2$ and $m_{ADM}(g,\Phi)$ exists, then $m_{ADM}(g, \Phi) = \lim_{r\to\infty}M_{C^0}(g, \Phi, \varphi^r, r)$.
\item\label{item:generalC0limitexistence} If $g$ has nonnegative scalar curvature in the sense of Ricci flow on $M\setminus K$ and (\ref{eq:C0AFdecay}) holds for $\Phi$ for some $\tau > (n-2)/2$, then $\lim_{r\to \infty}M_{C^0}(g, \Phi, \varphi^r, r)$ exists, is either finite or $+\infty$, and is independent of choice of such $\Phi$ and $\varphi$. Moreover, this limit is finite if and only if a particular condition involving the scalar curvature of time slices of Ricci-DeTurck flows associated to $g$ is satisfied.
\end{enumerate}
Furthermore, in the case that $M$ has multiple ends, the result holds if $M\setminus K$ is replaced by a neighborhood of an end of $M$ that is diffeomorphic to $\R^n\setminus \overline{B(0,1)}$.
\end{theorem}

\begin{corollary}\label{cor:derivdecaynotneeded}
Let $M^n$ be a smooth manifold and suppose $g$ is any $C^2$ Riemannian metric on $M$. Suppose that there is a smooth coordinate chart $\Phi: M\setminus K \to \R^n\setminus \overline{B(0,1)}$, where $K$ is some compact set, such that
\begin{enumerate}
\item\label{item:C0AF} there is some $\tau> (n-2)/2$ for which $\Phi_*g$ satisfies (\ref{eq:C0AFdecay}) for $\tau$, and
\item\label{item:mADMexists} the classical ADM mass $m_{ADM}(g, \Phi)$ exists.
\end{enumerate}
If $g$ has nonnegative scalar curvature, then $m_{ADM}(g, \Phi)$ is independent of choice of $\Phi$ satisfying \ref{item:C0AF} and \ref{item:mADMexists}, regardless of whether $\Phi_*g$ satisfies (\ref{eq:CkAFdecay}) or not. In particular, in the coordinate invariance statement in Theorem \ref{thm:classicalADMmassexistence}, the hypothesis (\ref{eq:CkAFdecay}) may be replaced with the condition that $m_{ADM}(g, \Phi)$ exists, when the scalar curvature of the metric is assumed to be nonnegative.
\end{corollary}
\begin{proof}[Proof of Corollary \ref{cor:derivdecaynotneeded}]
The first statement of Theorem \ref{thm:vagueC0existence} applied to $\Phi$ implies that, for any such $\Phi$, 
\begin{equation*}
m_{ADM}(g, \Phi) = \lim_{r\to\infty}M_{C^0}(g, \Phi, \varphi^r, r). 
\end{equation*}
On the other hand, the second statement of Theorem \ref{thm:vagueC0existence} implies that $\lim_{r\to\infty}M_{C^0}(g, \Phi, \varphi^r, r)$ is independent of choice of $\Phi$ satisfying (\ref{eq:C0AFdecay}) for $\tau > (n-2)/2$. 
\end{proof}

As we will explain, Theorem \ref{thm:vagueC0existence} follows from a monotonicity result for the $C^0$ local mass, the precise statement of which is given in the next section (see Theorem \ref{thm:fullC0monotonicity}):
\begin{theorem}\label{thm:vagueC0monotonicity}
Let $M$ be a smooth manifold and $g$ a $C^0$ metric on $M$. Suppose that $U_1$ and $U_2$ are open subsets of $M$ for which, for $m= 1,2$, there exist coordinate charts $\Phi^m: U_m\to \R^n\setminus \overline{B(0,1)}$ that determine the same end of $M$ (see Definition \ref{def:C0AFend}) and such that, for some $\tau_m> \tfrac{1}{2}(n-2)$, (\ref{eq:C0AFdecay}) holds for $\Phi^m$ and $\tau_m$. If $g$ has nonnegative scalar curvature in the sense of Ricci flow and $\varphi_m \geq 0$ are smooth cutoff functions with $\supp(\varphi_m)\subset \subset (.9, 1.1)$, then, for all sufficiently large $r$,
\begin{equation*}
M_{C^0}(g, \Phi^1, \varphi_1^{200r} ,200r) \geq M_{C^0}(g, \Phi^2, \varphi_2^r, r) - cr^{-\omega}
\end{equation*}
for some $\omega>0$, where $c$ and $\omega$ do not depend on $r$.
\end{theorem}

It is natural to ask where the $C^0$ quantity $\lim_{r\to\infty}M_{C^0}(g, \Phi, \varphi^r, r)$ agrees with other $C^0$ notions of mass that have already been introduced:
\begin{question*}
How does $\lim_{r\to\infty}M_{C^0}(g, \Phi, \varphi^r, r)$ relate to Huisken's isoperimetric mass \cite[Definition $11$]{JaureguiLee16} and Jauregui's isocapacitary mass \cite[Definition $4$]{Jauregui20}?
\end{question*}
Towards proving a $C^0$ version of Theorem \ref{thm:PMT} we ask:
\begin{question*}
If $g$ satisfies the hypotheses of \ref{item:generalC0limitexistence} in Theorem \ref{thm:vagueC0existence} and moreover $g$ has nonnegative scalar curvature in the sense of Ricci flow everywhere on $M$, do we have
\begin{equation*}
\lim_{r\to\infty}M_{C^0}(g, \Phi, \varphi^r, r)\geq 0?
\end{equation*}
\end{question*}

\subsection*{Acknowledgements} I would like to thank Richard Bamler and Bruce Kleiner for their guidance and helpful discussions regarding both the mathematics and the writing of this paper. I would also like to thank Dan Lee for suggesting that I look into Positive Mass Theorems, for sending me useful references, and for making several helpful comments on an earlier draft of this paper. I would like to thank Yuqiao Li for sending me many helpful comments on an earlier draft of this paper, and for reading my work in such detail. I am grateful to Otis Chodosh, Christos Mantoulidis, and Sung-Jin Oh for sending helpful references and explaining important background material to me. Finally, I would like to thank Charles Cifarelli for many useful discussions.

This material is based upon work supported by the National Science Foundation under Award No. DMS $2103145$ and by the National Science Foundation Graduate Research Fellowship Program under Grant No. DGE $1752814$.

\section{Technical Introduction}\label{sec:technicalintro}
We now make precise some of the statements from the previous section. We first define the $C^0$ mass.
\begin{definition}\label{def:C0mass}
Let $r>0$, $g$ be a $C^0$ metric on a smooth manifold $M$, $\Phi: U\to \R^n$ a smooth coordinate chart for $M$ such that $A(0, .9r, 1.1r) \subset \Phi(U)$, where $A(0, .9r, 1.1r)$ denotes the annulus in $\R^n$ measured with respect to the Euclidean metric, and $\varphi:\R\to \R$ some smooth function such that $\int_{.9}^{1.1}\varphi(s)ds \neq 0$. Writing $g_{ij}$ for $(\Phi_*g)_{ij}$, we define the \emph{$C^0$ local mass of $g$ with respect to $\varphi$ and $\Phi$ at $r$} by
\begin{equation}\label{eq:defC0mass}
\begin{split}
M_{C^0}&(g, \Phi,  \varphi, r)
\\& := 
\frac{(4\pi (n-1)\omega_{n-1})^{-1}}{r\int_{.9}^{1.1}\varphi(\ell)d\ell}\sum_{i,j = 1}^{n}\bigg[ \int_{A(o,.9r, 1.1r)} \left( \frac{n-2}{|x|}\varphi(\tfrac{|x|}{r}) + \frac{1}{r}\varphi'(\tfrac{|x|}{r})\right)\delta^{ij}(g_{ij} - \delta_{ij})
\\& \qquad \qquad + \left(\frac{1}{|x|}\varphi(\tfrac{|x|}{r}) - \frac{1}{r}\varphi'(\tfrac{|x|}{r})\right)(g_{ij} - \delta_{ij})\frac{x^ix^j}{|x|^2}dx
\\& \qquad \qquad+ \int_{\partial A(o, .9r, 1.1r)} (g_{ij} - \delta_{ij})\frac{x^i}{|x|}\varphi(\tfrac{|x|}{r})\nu^j - (g_{jj}- \delta_{jj})\frac{x^i}{|x|}\varphi(\tfrac{|x|}{r})\nu^i dS\bigg],
\end{split}
\end{equation}
where $\nu$ denotes the outward unit normal vector with respect to the Euclidean metric along $\partial A(0,.9r, 1.1r)$ and $dS$ is the Euclidean surface measure on $\partial A(0, .9r, 1.1r)$. If we are working with a given coordinate chart $\Phi$, or with a Riemannian metric on $\R^n$, then we write $M_{C^0}(g, \varphi, r)$ rather than $M_{C^0}(g, \Phi, \varphi, r)$. Henceforth we will suppress the summation notation in (\ref{eq:defC0mass}) and the normalization constant $(4\pi (n-1)\omega_{n-1})^{-1}$ for the sake of brevity.
\end{definition}
\begin{remark}\label{rmk:C0agreeswithC1average}
Observe that for any smooth $\varphi: \R \to \R$ with $\int_{.9}^{1.1}\varphi(s)ds \neq 0$, if $h_{ij}:= g_{ij} - \delta_{ij}$, we have:
\begin{align*}
\int_{.9r}^{1.1r}&\int_{\S(u)} (\partial_j g_{ij} - \partial_i g_{jj})\nu^i\varphi(\tfrac{u}{r})dSdu  = \int_{.9r}^{1.1r}\int_{\S(u)} (\partial_j h_{ij} - \partial_i h_{jj})\nu^i\varphi(\tfrac{u}{r})dSdu 
\\& \qquad \qquad + \int_{.9r}^{1.1r}\int_{\S(u)} (\partial_j \delta_{ij} - \partial_i \delta_{jj})\nu^i\varphi(\tfrac{u}{r})dSdu
\\& = \int_{.9r}^{1.1r}\int_{\S(u)} (\partial_j h_{ij} - \partial_i h_{jj})\nu^i\varphi(\tfrac{u}{r})dSdu  + 0
\\& = \int_{A(o,.9r, 1.1r)}(\partial_j h_{ij} - \partial_i h_{jj})\frac{x^i}{|x|}\varphi(\tfrac{|x|}{r})dx
\\&= \int_{A(o,.9r, 1.1r)} -h_{ij}\partial_j\left(\frac{x^i}{|x|}\varphi(\tfrac{|x|}{r})\right) + h_{jj}\partial_i\left(\frac{x^i}{|x|}\varphi(\tfrac{|x|}{r})\right)dx 
\\& \qquad \qquad + \int_{\partial A(o, .9r, 1.1r)} h_{ij}\frac{x^i}{|x|}\varphi(\tfrac{|x|}{r})\nu^j - h_{jj}\frac{x^i}{|x|}\varphi(\tfrac{|x|}{r})\nu^i dS
\\&= \int_{A(o,.9r, 1.1r)} -h_{ij}\frac{\delta_i^j}{|x|}\varphi(\tfrac{|x|}{r}) + h_{ij}\frac{x^ix^j}{|x|^3}\varphi(\tfrac{|x|}{r}) - h_{ij}\frac{x^ix^j}{r|x|^2}\varphi'(\tfrac{|x|}{r})
\\& \qquad \qquad + h_{jj}\frac{\delta_i^i}{|x|}\varphi(\tfrac{|x|}{r}) - h_{jj}\frac{(x^i)^2}{|x|^3}\varphi(\tfrac{|x|}{r}) + h_{jj}\frac{(x^i)^2}{r|x|^2}\varphi'(\tfrac{|x|}{r})dx 
\\& \qquad \qquad +  \int_{\partial A(o, .9r, 1.1r)} h_{ij}\frac{x^i}{|x|}\varphi(\tfrac{|x|}{r})\nu^j - h_{jj}\frac{x^i}{|x|}\varphi(\tfrac{|x|}{r})\nu^i dS
\\&= \int_{A(o,.9r, 1.1r)}  \tr_\delta h \frac{(n-2)}{|x|}\varphi(\tfrac{|x|}{r}) + \tr_\delta h \frac{1}{r}\varphi'(\tfrac{|x|}{r}) + h_{ij}\frac{x^ix^j}{|x|^3}\varphi(\tfrac{|x|}{r}) - h_{ij}\frac{x^ix^j}{r|x|^2}\varphi'(\tfrac{|x|}{r})dx 
\\& \qquad \qquad + \int_{\partial A(o, .9r, 1.1r)} h_{ij}\frac{x^i}{|x|}\varphi(\tfrac{|x|}{r})\nu^j - h_{jj}\frac{x^i}{|x|}\varphi(\tfrac{|x|}{r})\nu^i dS.
\end{align*} 
Therefore,
\begin{equation*}
M_{C^0}(g, \Phi, \varphi, r) = (4\pi (n-1)\omega_{n-1})^{-1}\frac{\int_{.9r}^{1.1r}\int_{\S(u)} (\partial_j g_{ij} - \partial_i g_{jj})\nu^i\varphi(\tfrac{u}{r})dSdu}{r\int_{.9}^{1.1}\varphi(\ell)d\ell}.
\end{equation*}
In particular, if $m_{ADM}(g, \Phi)$ exists, then (\ref{eq:computingmADMviavarphi}) implies that 
\begin{equation*}
m_{ADM}(g, \Phi) = \lim_{r\to \infty} M_{C^0}(g, \Phi, \varphi, r).
\end{equation*}
\end{remark}

We now discuss the meaning of ``nonnegative scalar curvature in the sense of Ricci flow'', which takes the place of the classical nonnegative scalar curvature condition in the $C^0$ setting (cf. \cite{PBG19}, \cite{PBG20}, \cite{PBGthesis}).
\begin{definition}\label{def:nonnegscalarRFdef}
Let $M^n$ be a smooth manifold and $g$ be a continuous Riemannian metric on $M$. For $\beta \in (0,1/2)$ we say that $g$ has scalar curvature bounded below by $\kappa_0\in \R$ in the $\beta$-weak sense at $x\in M$ if there exists a coordinate chart $\Phi: U_x\to \Phi(U_x)$ for $M$, where $U_x$ is a neighborhood of $x$, and there exists a continuous metric $g_0$ on $\R^n$ and a Ricci-DeTurck flow $(g_t)_{t\in (0,T]}$ for $g_0$, with respect to the Euclidean background metric, satisfying (\ref{eq:RDTFinitialcondition}), (\ref{eq:RDTFXest}), and (\ref{eq:RDTFderivests}) (cf. \cite{KL1}), such that
\begin{equation*}
g_0 \big|_{\Phi(U_x)} = \Phi_*g,
\end{equation*}
and such that
\begin{equation}\label{eq:betaweakcondition}
\inf_{C>0}\left(\liminf_{t\searrow 0}\left(\inf_{B_{\delta}(\Phi(x), Ct^{\beta})} R(g_t) \right)\right) \geq \kappa_0.
\end{equation}
We say that $g$ has scalar curvature bounded below by $\kappa_0\in \R$ in the sense of Ricci flow on an open region $U\subset M$ if there exists some $\beta\in (0, 1/2)$ such that, for all $x\in U$, $g$ has scalar curvature bounded below by $\kappa_0$ in the $\beta$-weak sense at $x$.
\end{definition}

\begin{remark}\label{rmk:globalbetaweak}
As explained in \cite[Section $6.2$]{PBGthesis}, Definition \ref{def:nonnegscalarRFdef} is an extension of the definition introduced in \cite{PBG19}, \cite{PBG20} to noncompact manifolds. Also, by the discussion in \cite[Section $6.2$]{PBGthesis} and \cite[Remark $3.6$ and Theorem $3.7$]{PBG20}, if $|| \Phi_*g - \delta||_{C^0(\Phi(U))}$ is sufficiently small so that $\Phi_*g$ may be extended to a continuous metric $g_0$ on $\R^n$ for which there is a Ricci-DeTurck flow $(g_t)_{t\in (0, T]}$ for $g_0$, with respect to the Euclidean background metric, satisfying (\ref{eq:RDTFinitialcondition}), (\ref{eq:RDTFXest}), and (\ref{eq:RDTFderivests}), then it is equivalent to require that there exists some $\beta \in (0, 1/2)$ such that, for all $y\in \Phi(U)$,
\begin{equation}\label{eq:betaweakconditionRn}
\inf_{C>0}\left(\liminf_{t\searrow 0}\left(\inf_{B_{\delta}(y, Ct^{\beta})} R(g_t) \right)\right) \geq \kappa_0,
\end{equation}
i.e. that (\ref{eq:betaweakcondition}) holds for a fixed Ricci-DeTurck flow that is independent of choice of $x\in U$.
\end{remark}
\begin{remark}\label{rmk:betaweakscalinginvariant}
By invariance of the Ricci-DeTurck flow under parabolic rescaling (see Remark \ref{rmk:parabolicscaling}), the $\beta$-weak condition (\ref{eq:betaweakcondition}) is invariant under rescaling of $g$ when $\kappa_0 = 0$.
\end{remark}

We next provide the precise statements of Theorems \ref{thm:vagueC0existence} and \ref{thm:vagueC0monotonicity}, and explain how the functions $\varphi^r$ are related to $\varphi$. Much of Theorem \ref{thm:vagueC0monotonicity} is proved by using estimates for the Ricci-DeTurck flow on perturbations of Euclidean space (see \cite{KL1}) to control the distortion of the $C^0$ local mass, when coupled with a suitably evolving smooth function:

\begin{lemma}\label{lemma:massdistortionestimate}
Let $\varphi: \R \to \R^{\geq 0}$ be a smooth cutoff function with $\supp(\varphi)\subset \subset (.9, 1.1)$. For all $r>0, \eta>0$ there exists a smooth function $\varphi_{r^{-\eta}}(\ell, t): \R \times [0, r^{-\eta}]\to \R$ such that
\begin{equation}\label{eq:cutofffunctionevsinglevar}
\begin{cases}
\partial_t \varphi_{r^{-\eta}}(|x|, t) &= -\Delta \varphi_{r^{-\eta}}(|x|, t) + \frac{n-1}{|x|^2}\varphi_{r^{-\eta}}(|x|, t) \text{ for } (x,t)\in \R^n\times (0, r^{-\eta})\\
\varphi_{r^{-\eta}}(\ell, r^{-\eta}) &= \varphi(\ell) \text{ for all } \ell\in \R.
\end{cases}
\end{equation}
Moreover, there exists $\bar r = \bar r(n, \eta)$ such that for all $r> \bar r$, the following is true:

Suppose that $g_0$ is a continuous metric on $\R^n$ such that $|| g_0 - \delta||_{C^0(\R^n)} \leq \varepsilon$ for some $\varepsilon < 1$, and for which there exists a smooth Ricci-DeTurck flow, $(g_t)_{t>0}$, with respect to the Euclidean background metric, satisfying (\ref{eq:RDTFinitialcondition}), (\ref{eq:RDTFXest}), and (\ref{eq:RDTFderivests}). Then
\begin{equation*}
\begin{split}
\int_0^{r^{2-\eta}}\left|\frac{d}{dt} M_{C^0}(g_t, \varphi_{r^{-\eta}}(\cdot, \tfrac{t}{r^2}), r)\right|dt &\leq c(n,\varphi)\varepsilon^2r^{n-2} 
\\& + c(n, \varphi)r^{n-2 +n\eta/2-\eta}\exp\left(-D(\supp(\varphi))r^{\eta}\right),
\end{split}
\end{equation*}
where $D= D(n)$.  

In particular, there exists $\bar r = \bar r(n, \eta, \supp(\varphi), c_0, \tau)$ such that for all $r>\bar r$, the following is true:

Suppose that $g_0$ is a continuous metric on $\R^n$ such that $|| g_0 - \delta||_{C^0(\R^n)} \leq c_0r^{-\tau}$ for some $\tau > (n-2)/2$. Suppose that $(g_t)_{t>0}$ is a smooth Ricci-DeTurck flow for $g_0$, with respect to the Euclidean background metric, satisfying (\ref{eq:RDTFinitialcondition}), (\ref{eq:RDTFXest}), and (\ref{eq:RDTFderivests}). Then
\begin{equation*}
\int_0^{r^{2-\eta}}\left|  \frac{d}{dt}M_{C^0}(g_t, \varphi_{r^{-\eta}}(\cdot, \tfrac{t}{r^2}), r)\right|dt \leq c(n,\varphi, c_0)r^{n - 2 - 2\tau}.
\end{equation*} 
\end{lemma}
Observe that, because of the specified lower bound for $\tau$, $n-2 -2\tau <0$, so the right hand side decays to $0$ as $r\to \infty$. Typically we will take $g_0$ to be a continuous extension of a $C^0$ metric $g|_{A(0, .9r, 1.1r)}$ where $g$ is defined on $A(0, .8r, 1.2r)$ and satisfies $|| g - \delta||_{C^0(A(0,.8r, 1.2r))} \leq c_0r^{-\tau}$.
\begin{remark}
If $u$ is any spherically symmetric solution to the backwards heat equation $\partial_t u = -\Delta u$ on Euclidean space, then the evolution equation from (\ref{eq:cutofffunctionevsinglevar}) is precisely the evolution equation for the radial derivative of $u$.
\end{remark}

In order to address the case when $(M, g)$ has multiple ends, we now discuss what we mean by a $C^0$-asymptotically flat end of $(M,g)$. We use the convention of \cite[p. $201$]{DK03} to describe the ends. Note that because $M$ is a topological manifold, it admits a compact exhaustion. Let $K_1 \subseteq K_2 \subseteq K_3 \subseteq \dots$ be an exhaustion of $M$ by compact sets. Consider a sequence $(X_i)_{i=1}^{\infty}$ such that each $X_i$ is a connected component of $M\setminus K_i$, and such that $X_1 \supset X_2 \supset X_3 \supset \ldots$. Suppose $(K_i')_{i=1}^{\infty}$ is another compact exhaustion of $M$ and $(X_i')_{i=1}^{\infty}$ is another such sequence, i.e. each $X_i'$ is a connected component of $M\setminus K_i'$ and $X_1'\supset X_2'\supset X_3'\supset\ldots$. We say that $(X_i)_{i=1}^{\infty}$ and $(X_i')_{i=1}^{\infty}$ are equivalent if, for all $i\in \N$ there exist $j$ and $k$ such that $X_i \supset X_j'$ and $X_i' \supset X_k$. The ends of $M$ are the equivalence classes of such sequences.

Suppose that $U\subset M$ is an open subset of $M$ for which there exists a diffeomorphism $\Phi: U\to \R^n\setminus \overline{B(0, 1)}$ such that $\Phi$ satisfies (\ref{eq:C0AFdecay}) for some $\tau >0$. Then $\Phi$ determines an end of $M$ as follows: Let $V = \Phi^{-1}(\R^n\setminus B(0,2))$. Consider the sequence of annuli $A^i = \overline{A(0, 2, 10^i)}$. Observe that $\Phi^{-1}(A_i)$ is compact because $A_i$ is compact in $\R^n\setminus \overline{B(0,1)}$ and $\Phi$ is a diffeomorphism onto this region. We extend $\Phi^{-1}(A_i)$ to a a compact exhaustion of $M$ as follows: let $K_i'$ be any compact exhaustion of $M$. Let $K_i = \overline{(K_i' \setminus V)} \cup \Phi^{-1}(A_i)$. Then the $K_i$ are compact because they are all unions of two compact sets, and they exhaust $M$ since $\overline{K_i'\setminus V}$ exhausts $M\setminus V$ and $\Phi^{-1}(A_i)$ exhausts $V$ (because $A_i$ exhausts $\R^n\setminus B(0,2)$).

 For all $i\in \N$ let $X^i$ denote the connected component of $M\setminus K_i$ that contains $\Phi^{-1}(\R^n\setminus\overline{B(0, 10^i)})$; this is possible since $\Phi$ is a diffeomorphism onto its image and hence $\Phi^{-1}(\R^n\setminus\overline{B(0, 10^i)})$ is connected. We actually have that for all sufficiently large $i\in \N$, $X^i = \Phi^{-1}(\R^n\setminus\overline{B(0, 10^i)})$: to see this, it is sufficient to show that $\Phi^{-1}(\R^n\setminus\overline{B(0, 10^i)})$ is clopen in $M\setminus K_i$. Certainly $\Phi^{-1}(\R^n\setminus\overline{B(0, 10^i)})$ is open in $M\setminus K_i$ by continuity, since $\R^n\setminus\overline{B(0, 10^i)}$ is open in $\R^n$ and $M\setminus K_i$ is open in $M$. To show closedness, let $y\in M\setminus K_i$ be a limit point of $\Phi^{-1}(\R^n\setminus\overline{B(0, 10^i)})$, so that there exists a sequence of points $y_k \in \Phi^{-1}(\R^n\setminus\overline{B(0, 10^i)})$ such that $y_k \to y$. Then $(y_k)_{k=1}^{\infty}$ is a Cauchy sequence with respect to $g$. Let $x_k = \Phi(y_k)\in \R^n\setminus\overline{B(0, 10^i)}$, so that $(x_k)_{k=1}^{\infty}$ is Cauchy with respect to $\Phi_* g$ and hence is also Cauchy with respect to $\delta$, since $\Phi_*g$ is uniformly bilipschitz to $\delta$ outside of a large compact set. By completeness $x_k \to x\in \R^n\setminus B(0, 10^i)$, so either $x\in \R^n\setminus \overline{B(0, 10^i)}$ or $x\in \partial B(0, 10^i)$. By continuity we have $\Phi^{-1}(x) = \lim_{k\to \infty} \Phi^{-1}(x_k) = \lim_{k\to \infty} y_k = y$. If $x\in \partial B(0, 10^i)$ then $x\in A_i$ and hence $y \in \Phi^{-1}(A_i) \subseteq K_i$, a contradiction. Therefore $x\in \R^n\setminus \overline{B(0,10^i)}$ and $y \in \Phi^{-1}(\R^n\setminus \overline{B(0,10^i)})$, so $\Phi^{-1}(\R^n\setminus \overline{B(0, 10^i)})$ is closed in $M\setminus K_i$. In particular, $ \Phi^{-1}(\R^n\setminus \overline{B(0, 10^i)})$ is a connected component of $M\setminus K_i$. Then $X^1 \supset X^2 \supset \ldots$ so $(X^i)_{i=1}^{\infty}$ determines an end of $M$. 
  
\begin{definition}\label{def:C0AFend}
We say that an end $E$ of $M$ is $C^0$-asymptotically flat if it is determined by some such $\Phi$, and we say that $\Phi$ is a $C^0$-asymptotically flat coordinate chart for $E$. We now establish some terminology. If $\Phi$ has the property that for all $|x| \geq r_0$,
\begin{equation*}
|(\Phi_* g)_{ij}|_x - \delta_{ij}| \leq c_0 |x|^{-\tau} 
\end{equation*}
then we say that $r_0$ is the \emph{decay threshold} of $\Phi$, $c_0$ is the \emph{decay coefficient}, and $\tau$ is the \emph{decay rate}.
\end{definition}

\begin{theorem}\label{thm:fullC0existence}
Let $M$ be a smooth manifold and $g$ a continuous Riemannian metric on $M$. Suppose $E$ is an end of $M$ and that $\Phi: U\to \R^n\setminus \overline{B(0,1)}$ is a $C^0$-asymptotically flat coordinate chart for $E$. Let $\varphi: \R \to \R^{\geq 0}$ be a smooth cutoff function with $\supp(\varphi)\subset \subset (.9, 1.1)$. For all $r>0$ and $\eta>0$, let $\varphi_{r^{-\eta}}(\ell, t)$ denote the smooth time-dependent function corresponding to $\varphi$ given by Lemma \ref{lemma:massdistortionestimate}. Then the following is true:
\begin{enumerate}
\item If $g$ is $C^2$ and $m_{ADM}(g, \Phi)$ exists, then, for all $\eta>0$, we have 
\begin{equation*}
m_{ADM}(g, \Phi) = \lim_{r\to \infty}M_{C^0}(g, \Phi, \varphi_{r^{-\eta}}(\cdot, 0), r).
\end{equation*}
\item If $g$ has nonnegative scalar curvature in the sense of Ricci flow on $U$ and (\ref{eq:C0AFdecay}) holds for $\Phi$ for some $\tau > (n-2)/2$, then, for all $0 < \eta < 2\tau - n + 2$, the limit $\lim_{r\to \infty}M_{C^0}(g, \Phi, \varphi_{r^{-\eta}}(\cdot, 0), r)$ exists, is either finite or $+\infty$, and is independent of choice of $C^0$-asymptotically flat coordinate chart for $E$ with decay rate $\tau > (n-2)/2$ and choice of $\varphi$ with $\supp(\varphi)\subset \subset (.9, 1.1)$. Moreover, $\lim_{r\to \infty}M_{C^0}(g, \Phi, \varphi_{r^{-\eta}}(\cdot, 0) ,r)$ is finite if and only if the following condition holds:

There exists a sequence of numbers $r_k \to \infty$ such that $r_{k+1} > 1.1/.9 r_k >0$ for all $k$, and such that for all $k$ there exists an extension $g_0^k$ of $\Phi_* g|_{\R^n\setminus \overline{B(0, .7r_k)}}$ to all of $\R^n$ for which there is a Ricci-DeTurck flow $(g_t^k)_{t>0}$ satisfying (\ref{eq:RDTFinitialcondition}), (\ref{eq:RDTFXest}), and (\ref{eq:RDTFderivests}), such that
\begin{equation}\label{eq:L1scalarcondition}
\lim_{k\to \infty}\sup_{r' > 1.1/.9 r_k}\int_{A(0, .9r_k, 1.1 r')} R(g^k_{(.9/1.1r_k)^{2-\eta}})dx= 0.
\end{equation}
\end{enumerate}
\end{theorem}

\begin{remark}\label{rmk:etaindependence}
The limit $\lim_{r\to \infty}M_{C^0}(g, \varphi_{r^{-\eta}}(\cdot, 0), r)$ is also independent of choice of $\eta \in (0, 2\tau - n +2)$, as we will explain in Section \ref{sec:masslimit}.
\end{remark}

\begin{remark}
A keen reader will note that while Theorem \ref{thm:fullC0existence} is an analogous statement to Theorem \ref{thm:classicalADMmassexistence} for $C^0$ metrics, it deviates from Theorem \ref{thm:classicalADMmassexistence} in the following way: Theorem \ref{thm:classicalADMmassexistence} assumes that the scalar curvature is in $L^1$ (\ref{eq:RL1}) in order to conclude that the ADM mass is a well-defined, finite limit. On the other hand, Theorem \ref{thm:fullC0existence} does not assume an integrability condition for the scalar curvature, instead substituting this condition for the condition that the scalar curvature of the metric is nonnegative (in a weak sense, in a neighborhood of infinity), and concluding that the corresponding limit is well-defined, but possibly infinite. 

The condition that the scalar curvature of the metric be nonnegative in a weak sense is natural for the purposes of pursuing a $C^0$ version of the Riemannian Positive Mass Theorem, and there is precedent in the literature for imposing this condition (see, for instance, \cite[Theorem $1.1$ and Proposition $2.4(2)$]{LeeLeFloch14}). Nonetheless, it is natural to ask whether this condition could instead be replaced by a $C^0$ version of the condition that the scalar curvature be in $L^1$:
\begin{question*}\label{question:L1}
Suppose $g$ is a continuous Riemannian metric on a smooth manifold $M$. Suppose that $E$ is an end of $M$ and that $\Phi: U\to \R^n\setminus\overline{B(0,1)}$ is a $C^0$-asymptotically flat coordinate chart for $E$. Let $\varphi$ and $\varphi_{r^{-\eta}}$ be as in Theorem \ref{thm:fullC0existence}. Does there exist a condition, $(*)$, such that:
\begin{enumerate}
\item if $g$ satisfies $(*)$, then $\lim_{r\to\infty} M_{C^0}(g, \Phi, \varphi_{r^{-\eta}}(\cdot, 0), r)$ is well-defined and finite, and
\item if $g$ is $C^2$ and satisfies (\ref{eq:C0AFdecay}), (\ref{eq:CkAFdecay}), and $R(g)\in L^1(E)$, then $(*)$ holds for $g$?
\end{enumerate}
\end{question*}
A $C^0$ condition $(*)$ should satisfy both of these items in order to be a $C^0$-weak version of the condition (\ref{eq:RL1}). We expect that a suitable modification of (\ref{eq:L1scalarcondition}) would satisfy both of these items, though we do not show such a thing in this paper. The proof of Theorem \ref{thm:finitemass} implies that when condition (\ref{eq:L1scalarcondition}) holds, then there exists a sequence of radii $\tilde r_k\to \infty$ such that $\lim_{k\to\infty}M_{C^0}(g, \varphi_{(\tilde r_k)^{-\eta}}(\cdot, 0), \tilde r_k)$ exists and is finite; the non-negativity of the scalar curvature is only used to guarantee that the limit does not depend on the sequence of radii. Whether the condition (\ref{eq:L1scalarcondition}) can be modified to satisfy the second item in Question \ref{question:L1} is more subtle: in the classical setting, McFeron -- Sz\'ekelyhidi showed \cite[Lemma $10$]{McFeronSzekelyhidi12} uniform-in-time decay of the integral of the scalar curvature outside of balls of large radii, for time slices of the Ricci flow. Therefore, we would expect some version of condition (\ref{eq:L1scalarcondition}) to hold in the classical setting when $R(g)\in L^1$. However, this is not immediate from \cite[Lemma $10$]{McFeronSzekelyhidi12}, since in this paper we use Ricci-DeTurck, rather than Ricci flow.

Finally, we remark that Lee -- LeFloch \cite[Proposition $2.4$]{LeeLeFloch14} have shown that for metrics $g\in C^0\cap W^{1,n}$ that are asymptotically flat in a suitable sense, the condition (\ref{eq:RL1}) can be replaced by the condition that the scalar curvature of $g$ be a finite, signed measure outside of a compact set. This condition is not available in our setting, since the scalar curvature does not have a distributional interpretation for general $C^0$ metrics.
\end{remark}

The proof of first statement in Theorem \ref{thm:fullC0existence} is relatively straightforward, as we will explain in Section \ref{sec:masslimit}. The second statement follows from the monotonicity result Theorem \ref{thm:fullC0monotonicity}, as we will also explain in Section \ref{sec:masslimit}. Therefore, we devote the bulk of the paper to proving Theorem \ref{thm:fullC0monotonicity}:

\begin{theorem}\label{thm:fullC0monotonicity}
Let $M$ be a smooth manifold and $g$ a continuous metric on $M$. Suppose $\Phi^1: U_1 \to \R^n\setminus \overline{B(0,1)}$ and $\Phi^2: U_2\to \R^n\setminus \overline{B(0,1)}$ are two $C^0$-asymptotically flat coordinate charts for the same end, so that for $m=1,2$, there exist $c_m>0, r_m > 1$, and $\tau_m > (n-2)/2$ such that
\begin{equation*}
|(\Phi^m_*g)_{ij} - \delta_{ij}| \leq c_m|x|^{-\tau_m} \text{ for all } |x| > r_m.
\end{equation*}
Then there exist $\bar r = \bar r(\Phi_1, \Phi_2, \varphi^1, \varphi^2, \eta, \beta ,n)$ and $c= c(\Phi_1, \Phi_2, \varphi^1, \varphi^2, \eta, \beta, n)$ such that if $g$ has nonnegative scalar curvature in the $\beta$-weak sense on $U_1 \cup U_2$, $\varphi^1$ and $\varphi^2$ are smooth cutoff functions with $\supp(\varphi^m)\subset \subset (.9, 1.1)$ for $m=1,2$, and $0 < \eta < 2\min\{\tau_1, \tau_2\}- n + 2$ then, for all $r> \bar r$,
\begin{equation*}
M_{C^0}(g, \Phi^1, \varphi_{(200 r)^{-\eta}}^1(\cdot, 0), 200 r) \geq M_{C^0}(g, \Phi^2, \varphi_{r^{-\eta}}^2(\cdot, 0), r) -cr^{n-2 - 2\min\{\tau_1, \tau_2\} + \eta},
\end{equation*}
where $\varphi^1_{r^{-\eta}}(\ell, t)$ and $\varphi^2_{r^{-\eta}}(\ell, t)$ are the smooth functions corresponding to $\varphi^1$ and $\varphi^2$ respectively, given by Lemma \ref{lemma:massdistortionestimate}. In particular,
\begin{equation*}
\lim_{r\to\infty} M_{C^0}(g, \Phi^2, \varphi_{r^{-\eta}}^2(\cdot, 0), r) = \lim_{r\to\infty} M_{C^0}(g, \Phi^1, \varphi_{r^{-\eta}}^1(\cdot, 0), r).
\end{equation*}
\end{theorem}

We now will explain the structure of the rest of the paper: In Section \ref{sec:preliminaries} we introduce the Ricci and Ricci-DeTurck flows and provide estimates for the heat kernel and scalar curvature under the flows. We also record some elementary facts, and discuss properties of glued locally bilipschitz maps and transition maps between $C^0$-asymptotically flat coordinate charts. In Section \ref{sec:distortion} we prove Lemma \ref{lemma:massdistortionestimate}. In Section \ref{sec:monotonicity} we prove a preliminary version of the monotonicity statement in Theorem \ref{thm:fullC0monotonicity}, for a single $C^0$-asymptotically flat coordinate chart. In Section \ref{sec:coordinates} we prove the full version of the monotonicity statement in Theorem \ref{thm:fullC0monotonicity}, using mollification and gluing statements from Appendix \ref{appendix:mollification}, which rely heavily on the results concerning bilipschitz maps and almost-isometries recorded in Appendix \ref{appendix:almostisometries}. In Section \ref{sec:masslimit} we prove the full statements of Theorems \ref{thm:fullC0existence} and \ref{thm:fullC0monotonicity}.

\section{Preliminaries}\label{sec:preliminaries}
\subsection{Ricci and Ricci-DeTurck flow preliminaries}
If $M$ is a smooth manifold and $(\tilde g_t)_{t\in (0, T)}$ is a smooth family of Riemannian metrics on $M$, recall that $\tilde g_t$ evolves by Ricci flow if
\begin{equation}\label{eq:RF}
\partial_t \tilde g_t = -2\Ric(\tilde g_t).
\end{equation}
We use the notation $\tilde g_t$ to distinguish this flow from the Ricci-DeTurck flow, which we use more often in this paper, and which we will denote by $g_t$. The Ricci-DeTurck flow, introduced by DeTurck in \cite{De}, is a strongly parabolic flow that is related to the Ricci flow by pullback via a family of diffeomorphisms. More specifically, we define the following operator, which maps symmetric $2$-forms on $M$ to vector fields:
\begin{equation}\label{eq:Xoperator}
X_{\bar g}(g):= \sum_{i=1}^n(\nabla^{\bar g}_{e_i}e_i - \nabla^{g}_{e_i}e_i),
\end{equation}
where $\{e_i\}_{i=1}^n$ is any local orthonormal frame with respect to $g$. Then the Ricci-DeTurck equation is
\begin{equation}\label{eq:RDTF}
\partial_t g(t) = -2\Ric(g(t)) - \L_{X_{\bar g(t)}(g(t))}g(t),
\end{equation}
where $\bar g(t)$ is a background Ricci flow. 
\begin{remark}\label{rmk:parabolicscaling}
If $g_t$ solves (\ref{eq:RDTF}) with respect to some background Ricci flow $\bar g_t$ for $t\in (a,b)$, then, for all $\lambda >0$, the parabolically rescaled flow $g'_t:= \lambda g_{t/\lambda}$ solves (\ref{eq:RDTF}) with respect to the background Ricci flow $\bar g'_t := \lambda \bar g_{t/\lambda}$ for $t\in (\lambda a, \lambda b)$.
\end{remark}

In this paper we work with Ricci-DeTurck flows on $\R^n$ with respect to a Euclidean background, so we take $\bar g(t)\equiv \delta$ and (\ref{eq:RDTF}) becomes
\begin{equation}\label{eq:RDTFeucl}
\partial_t g(t) = -2\Ric(g(t)) - \L_{X_{\delta}(g(t))}g(t).
\end{equation}
As mentioned, if $g(t)$ solves (\ref{eq:RDTF}) then it is related to a Ricci flow via pullback by diffeomorphisms. More precisely, if $g(t)$ solves (\ref{eq:RDTF}) and $(\chi_t)_{t\in (0, T)}: M\to M$ is the family of diffemorphisms satisfying
\begin{equation}\label{eq:diffeoseq}
\begin{cases}
X_{\bar g(t)}(g(t))f &= \frac{\partial}{\partial t}(f\circ\chi_t) \text{ for all } f\in C^\infty(M)\\
\chi_{\bar t} &= \id,
\end{cases}
\end{equation}
 then $\tilde g(t):= \chi_t^*g(t)$ solves (\ref{eq:RF}) with the condition $\tilde g(\bar t) = g(\bar t)$.

It is known (see \cite[Appendix A]{BK}) that if $g_t$ solves (\ref{eq:RDTF}) with respect to the background Ricci flow $\bar g_t$, and if $h_t  = g_t - \bar g_t$, then the evolution equation for $h_t$ is given by
\begin{equation}\label{eq:hevolution}
\partial_t h_t = -L h_t + Q[h_t],
\end{equation}
where 
\begin{equation}\label{eq:Lis}
\begin{split}
L h_t &= -\Delta^{\bar g_t}h_t - 2\Rm^{\bar g_t}[h_t]
%%%
\\&:= -\Delta^{\bar g_t}h_t - 2{\bar g}^{pq}R_{pij}^{m}h_{q m}dx^i\otimes dx^j + \bar g^{pq}(h_{pj}R_{qi} + h_{ip}R_{qj})
\end{split}
\end{equation}
(note that our notation convention for $\Rm^{\bar g_t}[h_t]$ differs slightly from that of \cite{BK} as we do not use the Uhlenbeck trick in this paper) and $Q$ denotes the quadratic term 
\begin{equation}
\begin{split}
\left(Q_{\bar g_t}[h_t]\right)_{ij}&:= \left((\bar g+h)^{pq} - \bar g^{pq}\right)\left(\nabla^2_{pq}h_{ij} +R_{pij}^mh_{mq} + R_{pji}^mh_{mq}\right)
\\& \qquad + \left(\bar g^{pq} - (\bar g+h)^{pq}\right)\left(R_{ipq}^mh_{mj} + R_{jpq}^mh_{im}\right)
\\& -\frac{1}{2}(\bar g+h)^{pq}(\bar g+h)^{m\ell}\big(-\nabla_i h_{pm}\nabla_jh_{q\ell} - 2\nabla_mh_{ip}\nabla_qh_{j\ell}
\\& \qquad + 2\nabla_mh_{ip}\nabla_{\ell}h_{jq} + 2\nabla_ph_{i\ell}\nabla_jh_{qm} + 2\nabla_ih_{pm}\nabla_qh_{j\ell}\big)
\\&= \nabla_p (\left((\bar g+h)^{pq} - \bar g^{pq}\right)\nabla_qh_{ij}) 
\\& \qquad - \left(\nabla_p\left((\bar g+h)^{pq} - \bar g^{pq}\right)\right)\nabla_qh_{ij} +  \left((\bar g+h)^{pq} - \bar g^{pq}\right)\left(R_{pij}^mh_{mq} + R_{pji}^mh_{mq}\right)
\\& \qquad + \left(\bar g^{pq} - (\bar g+h)^{pq}\right)\left(R_{ipq}^mh_{mj} + R_{jpq}^mh_{im}\right)
\\& -\frac{1}{2}(\bar g+h)^{pq}(\bar g+h)^{m\ell}\big(-\nabla_i h_{pm}\nabla_jh_{q\ell} - 2\nabla_mh_{ip}\nabla_qh_{j\ell}
\\& \qquad + 2\nabla_mh_{ip}\nabla_{\ell}h_{jq} + 2\nabla_ph_{i\ell}\nabla_jh_{qm} + 2\nabla_ih_{pm}\nabla_qh_{j\ell}\big),
\end{split}
\end{equation}
where here $\nabla$ denotes the covariant derivative with respect to $\bar g_t$. The second equality follows from the Leibniz rule. 

We often write
\begin{equation}
Q[h_t] = Q^0_t + \nabla^* Q^1_t,
\end{equation}
where
\begin{equation}\label{eq:Q0is}
\begin{split}
Q^0_t&:=-\frac{1}{2}(\bar g+h)^{pq}(\bar g+h)^{m\ell}\big(-\nabla_i h_{pm}\nabla_jh_{q\ell} - 2\nabla_mh_{ip}\nabla_qh_{jm}
\\& \qquad + 2\nabla_mh_{ip}\nabla_{\ell}h_{jq} + 2\nabla_ph_{i\ell}\nabla_jh_{qm} + 2\nabla_ih_{pm}\nabla_qh_{j\ell}\big)
\\& \qquad - \left(\nabla_p((\bar g+h)^{pq} - \bar g^{pq})\right)\nabla_qh_{ij} +  \left((\bar g+h)^{pq} - \bar g^{pq}\right)\left(R_{pij}^mh_{mq} + R_{pji}^mh_{mq}\right)
\\& = (\bar g + h)^{-1}\star (\bar g + h)^{-1} \star \nabla h \star \nabla h + ((\bar g + h)^{-1} - \bar g^{-1})\star \Rm^{\bar g_t}\star h
\end{split}
\end{equation}
and
\begin{equation}\label{eq:Q1is}
\nabla^*Q^1_t:= \nabla_p (\left((\bar g+h)^{pq} - \bar g^{pq}\right)\nabla_qh_{ij}) = \nabla(((\bar g + h)^{-1} - \bar g^{-1})\star \nabla h),
\end{equation}
where we use the notation $A\star B$ for two tensor fields $A$ and $B$ to mean a linear combination of products of the coefficients of $A$ and $B$, and $(\bar g + h)^{-1}$ and $\bar g^{-1}$ denote tensor fields with coefficients $(\bar g + h)^{ij}$ and $\bar g^{ij}$ respectively. Henceforth we will use the notation $A*_gB$ to denote any linear combination of tensor fields obtained from $A\otimes B$ by using $g$ to raise or lower any number of indices or by contracting any indices of such tensor fields using $g$, any number of times. Since in this paper we often work on Euclidean space, we will use $A*B$ to denote any term such that $|A*B|_\delta \leq c(n)|A|_\delta |B|_\delta$.

Henceforth we take all covariant derivatives and measure all balls and all norms with respect to the Euclidean metric $\delta$ on $\R^n$, unless otherwise stated. When $\bar g(t)\equiv \delta$, (\ref{eq:hevolution}) becomes (see \cite[(4.4)]{KL1}): 
\begin{equation}\label{eq:heveuclbackground}
\begin{split}
\partial_t h_{ij} &= \Delta h_{ij} + \frac{1}{2}(\delta + h)^{pq}(\delta + h)^{m\ell}\big( \nabla_i h_{pm}\nabla_j h_{q\ell} + 2\nabla_m h_{ip}\nabla_q h_{jm} -2\nabla_m h_{ip}\nabla_{\ell}h_{jp} 
\\& - 2\nabla_p h_{i\ell}\nabla_j h_{qm} -2\nabla_i h_{pm}\nabla_q h_{j\ell} \big) - \nabla_p((\delta + h)^{pq})\nabla_q h_{ij}
\\& + \nabla_p \big( ((\delta + h)^{pq} - \delta^{pq})\nabla_q h_{ij} \big)
%%%
\\&=: \Delta h_{ij} + Q^0[h] + \nabla^*Q^1[h]
%%%
\\& \text{ where}
%%%
\\& Q^0[h] = \frac{1}{2}(\delta + h)^{pq}(\delta + h)^{m\ell}\big( \nabla_i h_{pm}\nabla_j h_{q\ell} + 2\nabla_m h_{ip}\nabla_q h_{jm} -2\nabla_m h_{ip}\nabla_{\ell}h_{jp} 
\\& - 2\nabla_p h_{i\ell}\nabla_j h_{qm} -2\nabla_i h_{pm}\nabla_q h_{j\ell} \big) - \nabla_p((\delta + h)^{pq})\nabla_q h_{ij} = \nabla h * \nabla h,
%%%
\\& \nabla^*Q^1[h] = \nabla_p \big( ((\delta + h)^{pq} - \delta^{pq})\nabla_q h_{ij} \big) = \nabla (h * \nabla h),
\end{split}
\end{equation}
where $\Delta$ denotes the usual Euclidean Laplacian.

We now record the following result concerning Ricci-DeTurck flows starting from small $C^0$ perturbations of Euclidean space (cf. \cite[Theorem $4.3$]{KL1}, \cite[Lemma $3.3$ and Corollary $3.4$]{PBG19}):
\begin{lemma}\label{lemma:RDTFexistenceandests}
There exists $\bar \varepsilon = \bar \varepsilon(n) < 1$ and $c = c(n)$ such that the following is true:

If $g_0$ is any continuous Riemannian metric on $\R^n$ such that $|| g_0 - \delta||_{C^0(\R^n)} < \bar \varepsilon$ then there exists a smooth solution $(g_t)_{t>0}$ to (\ref{eq:RDTFeucl}) such that 
\begin{equation}\label{eq:RDTFinitialcondition}
g_t \xrightarrow[t\searrow 0]{C^0_{\loc}} g_0,
\end{equation}
\begin{equation}\label{eq:RDTFXest}
|| g_t - \delta||_{X} \leq c|| g_0 - \delta||_{C^0(\R^n)},
\end{equation}
and, for all $k\in \N$ there exists $c_k(n)>0$ such that for all $t>0$,
\begin{equation}\label{eq:RDTFderivests}
|| \nabla^k(g_t - \delta)||_{C^0(\R^n)} \leq c_k(n)\frac{|| g_0 - \delta||_{C^0(\R^n)}}{t^{k/2}},
\end{equation}
where $||\cdot||_X$ is given by
\begin{equation*}
\begin{split}
|| h ||_X &= \sup_{0 < t < \infty}|| h(t) ||_{L^\infty(\R^n)} 
\\& + \sup_{x\in \R^n}\sup_{0 < r}\left( r^{-n/2}||\nabla h ||_{L^2(B(x,r)\times (0, r^2))} + r^{\tfrac{2}{n+4}}|| \nabla h||_{L^{n+4}(B(x, r)\times(\tfrac{r^2}{2}, r^2))} \right).
\end{split}
\end{equation*}

Moreover, if $U, V \subset \R^n$ are any two open sets with $U\subset \subset V$ and $g$ is a continuous Riemannian metric on $V$ such that $||g - \delta||_{C^0(V)} < \bar \varepsilon$, then there exists a continuous Riemannian metric $g_0$ defined on all of $\R^n$ such that
\begin{equation*}
g_0\bigg|_{U} = g\bigg|_{U} \text{ and } || g_0 - \delta||_{C^0(\R^n)} \leq || g - \delta||_{C^0(V)},
\end{equation*}
so there exists a smooth solution $(g_t)_{t>0}$ to (\ref{eq:RDTFeucl}) satisfying (\ref{eq:RDTFinitialcondition}), (\ref{eq:RDTFXest}), and (\ref{eq:RDTFderivests}) for $g_0$.
\end{lemma}
\begin{proof}
The existence of a smooth solution $(g_t)_{t>0}$ and the estimates (\ref{eq:RDTFXest}) and (\ref{eq:RDTFderivests}) are due to \cite[Theorem $4.3$]{KL1}. That the solution converges to the initial data as $t\searrow 0$ follows in the same way as in the proof of \cite[Corollary $3.7$]{PBG19}.

Towards the second statement, let $\chi: \R^n \to [0,1]$ be a smooth cutoff function with $\chi\equiv 1$ on $U$ and $\supp(\chi)\subset V$. Then let $g_0 = \chi g + (1-\chi)\delta$. The $(0,2)$-tensor $g_0$ is continuous, symmetric, and positive definite because $g$ and $\delta$ are, so it is a $C^0$ Riemannian metric on $\R^n$. We also have
\begin{equation*}
|g_0 - \delta|\bigg|_x = \chi(x)|g - \delta| \bigg|_x \leq || g-\delta||_{C^0(V)} < \bar \varepsilon.
\end{equation*}
The existence of a solution $g_t$ satisfying (\ref{eq:RDTFinitialcondition}), (\ref{eq:RDTFXest}), and (\ref{eq:RDTFderivests}) then follows from the first statement.
\end{proof}
\begin{remark}
The result \cite[Theorem $4.3$]{KL1} also guarantees a solution $g_t$ when $g_0$ is only in $L^\infty(\R^n)$ with $|| g_0 - \delta||_{L^\infty(\R^n)} < \bar \varepsilon(n)$, but we will not address this situation in this paper.
\end{remark}

Observe that if $g_t$ is any solution to (\ref{eq:RDTFeucl}) satisfying (\ref{eq:RDTFderivests}), then we have
\begin{equation}\label{eq:RDTFscalargenerallowerbound}
R(g_t) \geq -\frac{c(n)}{t}.
\end{equation}

\subsection{Heat kernel estimates and evolution of scalar curvature under Ricci-DeTurck flow}
If $g_t$ evolves by Ricci-DeTurck flow (\ref{eq:RDTF}) then the scalar curvature of $g_t$ satisfies (\cite[(2.24)]{PBG19})
\begin{equation}\label{eq:RDTFscalarev}
\partial_t R(g_t) \geq \Delta^{g_t} R(g_t) - \langle X, \nabla R(g_t)\rangle + \frac{2}{n}R(g_t)^2,
\end{equation}
where $X$ is as in (\ref{eq:Xoperator}). 

For a given solution $(g_t)_{t>0}$ to (\ref{eq:RDTF}) with respect to a smooth background Ricci flow $(\bar g_t)_{t\geq 0}$ on a smooth manifold $M$, let $\Phi^{RD}(x,t; y,s)$ denote the scalar heat kernel for the operator $\partial_t - \Delta^{g_t} + \nabla^{g_t}_X$, where $X$ is as in (\ref{eq:Xoperator}), i.e. for fixed $y\in M$ and $0 < s < t$,
\begin{equation*}
\partial_t \Phi^{RD}(x,t; y,s) = \Delta_{g_t, x} \Phi^{RD}(x,t; y,s) - \nabla^{g_t}_{X, x}\Phi^{RD}(x,t;y,s).
\end{equation*}

We refer the reader to \cite{CCG+} for a thorough discussion of the heat kernel. We now record an estimate for $\Phi^{RD}(x,t;y,s)$ (cf. \cite[Lemma $3.8$]{PBGthesis}, \cite[Lemma $3$]{Bam16}, \cite[Lemma $2.9$]{PBG19}):
\begin{lemma}\label{lemma:RDTFscalarHKests}
Suppose $(g_t)_{t>0}$ is a solution to (\ref{eq:RDTFeucl}) satisfying $|| g_t - \delta||_{C^0(\R^n)} < b$ and (\ref{eq:RDTFderivests}) and let $\Phi^{RD}(x, t; y,s)$ denote the scalar heat kernel for $g_t$ as described above. Then there exist constants $C= C(n, b), D = D(n) >0$ such that for all $t>0$ and all $\tfrac{t}{2} \leq s \leq t$,
\begin{equation*}
\Phi^{RD}(x,t; y, s) \leq \frac{C}{(t-s)^{n/2}}\exp\left(-\frac{|x - y|^2_{\delta}}{D(t-s)}\right).
\end{equation*}
\end{lemma}
\begin{proof}
Since $g_t$ is uniformly $b$-bilipschitz to $\delta$, there exists a constant $c(n,b) >1$ such that for all $r>0$ and $x\in \R^n$, 
\begin{equation}\label{eq:volumecomparison}
c(n,b)^{-1}r^n \leq \vol_{g_t}(B_{g_t}(x,r)) \leq c(n,b)r^n.
\end{equation}
Let $\tilde g_t = \chi_t^* g_t$, where $(\chi_t)_{t>0}$ solve the differential equation from (\ref{eq:diffeoseq}) subject to the condition $\chi_{\bar t} = \id$, for some $\bar t>0$. Then $\tilde g_t$ is a Ricci flow, defined for $t>0$, with $\tilde g_{\bar t} = g_{\bar t}$. Let $\Phi^{RF}$ be the heat kernel for the operator $\partial_t - \Delta^{\tilde g_t}$ on the $\tilde g_t$-background. Let $t_0>0$. By (\ref{eq:RDTFderivests}), $\tilde  g_t$ satisfies the hypotheses of \cite[Lemma $2.4$]{PBG19}. Then, applying \cite[Lemma $2.4$]{PBG19} to the time interval $[\tfrac{t_0}{2}, t_0]$, there exist $C(n), D(n)$ such that, for $s, t \in [\tfrac{t_0}{2}, t_0]$
\begin{equation*}
\Phi^{RF}(x,t ;y,s ) \leq \frac{C\exp\left(-\frac{d^2_{\tilde g(t_0/2)}(x,y)}{D(t-s)}\right)}{\vol^{1/2}_{\tilde g(t_0/2)}B_{\tilde g(t_0/2)}\left(x, \sqrt{\tfrac{t-s}{2}}\right)\vol^{1/2}_{\tilde g(t_0/2)}B_{\tilde g(t_0/2)}\left(y, \sqrt{\tfrac{t-s}{2}}\right)}.
\end{equation*}

Pushing forward by the $\chi_t$ and arguing as in the proof of \cite[Lemma $2.9$]{PBG19}, we find, by (\ref{eq:volumecomparison})
\begin{align*}
\Phi^{RD}(x, t; y,s) &= \Phi^{RF}(\chi_t^{-1}(x), t; \chi_s^{-1}(y), s)
%%%
\\& \leq \frac{C\exp\left(-\frac{d^2_{\tilde g(t_0/2)}(\chi_t^{-1}(x),\chi_s^{-1}(y))}{D(t-s)}\right)}{\vol^{1/2}_{\tilde g(t_0/2)}B_{\tilde g(t_0/2)}\left(\chi_t^{-1}(x), \sqrt{\tfrac{t-s}{2}}\right)\vol^{1/2}_{\tilde g(t_0/2)}B_{\tilde g(t_0/2)}\left(\chi_s^{-1}(y), \sqrt{\tfrac{t-s}{2}}\right)}
%%%
\\& \leq \frac{C\exp\left(-\frac{d^2_{g(t_0/2)}(x,y)}{D(t-s)}\right)}{\vol^{1/2}_{g(t_0/2)}B_{g(t_0/2)}\left(x, \sqrt{\tfrac{t-s}{2}}\right)\vol^{1/2}_{g(t_0/2)}B_{g(t_0/2)}\left(y, \sqrt{\tfrac{t-s}{2}}\right)}
%%%
\\& \leq \frac{C}{(t-s)^{n/2}}\exp\left(-\frac{|x - y|^2_{\delta}}{D(t-s)}\right),
\end{align*}
with $C$ and $D$ adjusted. Since $C$ and $D$ and (\ref{eq:RDTFderivests}) do not depend on choice of $t_0$, the estimate holds with $t_0$ replaced by $t$, whence follows the result.
\end{proof}

\subsection{Preliminary results for the local masses}

\begin{lemma}\label{lemma:rotationinvarianceofmass}
Let $O: \R^n\to \R^n$ be a rotation. If $g$ is a $C^0$ Riemannian metric on $\overline{A(0, .9r, 1.1r)}$ for some $r>0$ and $\varphi: \R \to \R^{\geq 0}$ is some smooth function not identically $0$ on $(.9, 1.1)$, then
\begin{equation*}
M_{C^0}(O^*g, \varphi, r) = M_{C^0}(g, \varphi, r).
\end{equation*}
\end{lemma}
\begin{proof}
This is a calculation. Let $O_i^j$ denote the $ij$-entry of the matrix for $O$, so that $\sum_{\ell = 1}^{n}O^i_{\ell}O^j_{\ell} = \delta^{ij}$. Then we have
\begin{align*}
&\int_{A(o,.9r, 1.1r)} \left( \frac{n-2}{|x|}\varphi(\tfrac{|x|}{r}) + \frac{1}{r}\varphi'(\tfrac{|x|}{r})\right)\delta^{ij}(O^*g_{ij} - \delta_{ij}) 
\\& \qquad \qquad + \left(\frac{1}{|x|}\varphi(\tfrac{|x|}{r}) - \frac{1}{r}\varphi'(\tfrac{|x|}{r})\right)(O^*g_{ij} - \delta_{ij})\frac{x^ix^j}{|x|^2}dx
%%%
\\&= \int_{A(o,.9r, 1.1r)} \left( \frac{n-2}{|x|}\varphi(\tfrac{|x|}{r}) + \frac{1}{r}\varphi'(\tfrac{|x|}{r})\right)\delta^{ij}(O^a_i g_{ab} O_j^b|_{Ox} - \delta_{ij}) 
\\& \qquad \qquad + \left(\frac{1}{|x|}\varphi(\tfrac{|x|}{r}) - \frac{1}{r}\varphi'(\tfrac{|x|}{r})\right)(O^a_i g_{ab} O_j^b|_{Ox} - \delta_{ij})\frac{x^ix^j}{|x|^2}dx
%%%
\\&= \int_{A(o,.9r, 1.1r)} \left( \frac{n-2}{|y|}\varphi(\tfrac{|y|}{r}) + \frac{1}{r}\varphi'(\tfrac{|y|}{r})\right)(\delta^{ab}g_{ab}|_{y} - \delta^{ij}\delta_{ij}) 
\\& \qquad \qquad + \left(\frac{1}{|y|}\varphi(\tfrac{|y|}{r}) - \frac{1}{r}\varphi'(\tfrac{|y|}{r})\right)(O^a_i g_{ab} O_j^b|_{y} - \delta_{ij})\frac{O^\ell_{i}y^{\ell}O^m_{j}y^{m}}{|y|^2}dy
%%%
\\&= \int_{A(o,.9r, 1.1r)} \left( \frac{n-2}{|y|}\varphi(\tfrac{|y|}{r}) + \frac{1}{r}\varphi'(\tfrac{|y|}{r})\right)(\delta^{ab}g_{ab}|_{y} - \delta^{ij}\delta_{ij}) 
\\& \qquad \qquad + \left(\frac{1}{|y|}\varphi(\tfrac{|y|}{r}) - \frac{1}{r}\varphi'(\tfrac{|y|}{r})\right) \left(g_{\ell m} - \delta_{\ell m} \right) \frac{y^{\ell}y^{m}}{|y|^2}dy.
\end{align*}
Similarly,
\begin{align*}
&\int_{\partial A(o, .9r, 1.1r)} (O^*g_{ij} - \delta_{ij})\frac{x^i}{|x|}\varphi(\tfrac{|x|}{r})\nu^j - (O^*g_{jj}- \delta_{jj})\frac{x^i}{|x|}\varphi(\tfrac{|x|}{r})\nu^i dS(x)
%%%
\\&= \int_{\partial A(o, .9r, 1.1r)} (O^*g_{ij}|_x - \delta_{ij})\varphi(\tfrac{|x|}{r})\frac{x^ix^j}{|x|^2} - (O^*g_{jj}|_x- \delta_{jj})\varphi(\tfrac{|x|}{r})\frac{(x^i)^2}{|x|^2} dS(x)
%%%
\\&= \int_{\partial A(o, .9r, 1.1r)} (O_i^ag_{ab}O_j^b|_y - \delta_{ij})\varphi(\tfrac{|y|}{r})\frac{O^i_{\ell}y^{\ell}O^j_my^m}{|y|^2} - (O_j^ag_{ab}O_j^b|_y- \delta_{jj})\varphi(\tfrac{|y|}{r})\frac{(O^i_{\ell}y^{\ell})^2}{|y|^2} dS(y)
%%%
\\&= \int_{\partial A(o, .9r, 1.1r)} (g_{\ell m}|_y - \delta_{\ell m})\varphi(\tfrac{|y|}{r})\frac{y^{\ell}y^m}{|y|^2} - (\delta^{ab}g_{ab}|_y- \delta_{jj})\varphi(\tfrac{|y|}{r})\frac{(y^{i})^2}{|y|^2} dS(y).
\end{align*}
\end{proof}

We now introduce the following notation: if $g$ is a $C^2$ Riemannian metric defined on a region in $\R^n$ containing $\S(r)$ for some $r>0$, then we write the following:
\begin{equation}\label{eq:mC2notation}
m_{C^2}(g, r) := \int_{\S(r)}(\partial_i g_{ij} - \partial_j g_{ii})\nu^jdS.
\end{equation}
If $g$ is a $C^2$ Riemannian defined on a subset of a smooth manifold $M$ and $\Phi$ is a coordinate chart for $M$ such that $\Phi_*g$ is defined on a region in $\R^n$ containing $\S(r)$ for some $r>0$, then we let
\begin{equation}
m_{C^2}(g, \Phi, r) := m_{C^2}(\Phi_*g, r).
\end{equation}
Here we record a calculation concerning the classical local mass, cf. \cite[Proposition $4.1$]{Bartnik86}. We will make use of this calculation in Section \ref{sec:monotonicity}.
\begin{lemma}\label{lemma:Bartnikscalculation}
Suppose $g$ is a $C^2$ Riemannian metric on $\R^n\setminus \overline{B(0,r_0)}$ such that 
\begin{equation*}
|| g - \delta||_{C^0(\R^n\setminus \overline{B(0, r_0)})} \leq b.
\end{equation*}
 For $r_0 < r_1 < r_2$ we have
\begin{equation*}
m_{C^2}(g, r_2) - m_{C^2}(g, r_1)
 \geq \int_{A(0, r_1, r_2)} R(g)  - c(n,b) |g - \delta| |\nabla^2 g | - c(n,b)|\nabla g|^2 dV_\delta.
\end{equation*}
\end{lemma}
\begin{proof}
First observe that 
\begin{align*}
R(g) &= g^{k\ell}(\partial_m \Gamma_{k\ell}^m - \partial_{\ell}\Gamma_{km}^m) +g^{k\ell}( \Gamma_{k\ell}^q\Gamma_{qm}^m - \Gamma_{mk}^q\Gamma_{\ell q}^m)
%%%
\\& = \frac{g^{k\ell}g^{mp}}{2}[(\partial_m\partial_k g_{p\ell} + \partial_m\partial_{\ell}g_{pk} - \partial_m\partial_{p}g_{k\ell}) 
\\& \qquad \qquad - (\partial_\ell\partial_k g_{pm} + \partial_\ell\partial_{m}g_{pk} - \partial_\ell\partial_p g_{km})]  
\\& + \frac{g^{k\ell}}{2}\left[ (\partial_m g^{mp})(\partial_k g_{\ell p} + \partial_{\ell}g_{kp} - \partial_p g_{k\ell}) - (\partial_{\ell}g^{mp})(\partial_k g_{mp} + \partial_m g_{kp} - \partial_p g_{km}) \right]
\\& + g^{k\ell}(\Gamma_{k\ell}^q\Gamma_{qm}^m - \Gamma_{mk}^q\Gamma_{\ell q}^m)
%%%
\\&= \sum_{m, k= 1}^{n}\frac{1}{2}[\partial_m \partial_k g_{mk} + \partial_m\partial_k g_{mk} - \partial_m^2g_{kk} - (\partial_k^2 g_{mm} + \partial_k\partial_m g_{mk} - \partial_k\partial_m g_{km})]
\\& + \frac{(g^{k\ell}g^{mp} - \delta^{k\ell}\delta^{mp})}{2}[(\partial_m\partial_k g_{p\ell} + \partial_m\partial_{\ell}g_{pk} - \partial_m\partial_{p}g_{k\ell}) 
\\& \qquad \qquad - (\partial_\ell\partial_k g_{pm} + \partial_\ell\partial_{m}g_{pk} - \partial_\ell\partial_p g_{km})]
\\& + \frac{g^{k\ell}}{2}\left[ (\partial_m g^{mp})(\partial_k g_{\ell p} + \partial_{\ell}g_{kp} - \partial_p g_{k\ell}) - (\partial_{\ell}g^{mp})(\partial_k g_{mp} + \partial_m g_{kp} - \partial_p g_{km}) \right]
\\& + g^{k\ell}(\Gamma_{k\ell}^q\Gamma_{qm}^m - \Gamma_{mk}^q\Gamma_{\ell q}^m)
%%%
\\&= \sum_{i,j = 1}^n \partial_j\partial_i g_{ij} - \partial_j^2g_{ii} + \frac{(g^{k\ell}g^{mp} - \delta^{k\ell}\delta^{mp})}{2}[(\partial_m\partial_k g_{p\ell} + \partial_m\partial_{\ell}g_{pk} - \partial_m\partial_{p}g_{k\ell}) 
\\& \qquad \qquad - (\partial_\ell\partial_k g_{pm} + \partial_\ell\partial_{m}g_{pk} - \partial_\ell\partial_p g_{km})]
\\& + \frac{g^{k\ell}}{2}\left[ (\partial_m g^{mp})(\partial_k g_{\ell p} + \partial_{\ell}g_{kp} - \partial_p g_{k\ell}) - (\partial_{\ell}g^{mp})(\partial_k g_{mp} + \partial_m g_{kp} - \partial_p g_{km}) \right]
\\& + g^{k\ell}(\Gamma_{k\ell}^q\Gamma_{qm}^m - \Gamma_{mk}^q\Gamma_{\ell q}^m).
\end{align*}
To summarize, we have
\begin{equation}\label{eq:scalarcurvexpansion}
R(g) = \sum_{i,j = 1}^n \partial_j\partial_i g_{ij} - \partial_j^2g_{ii} + Q^{R}[g]
\end{equation}
where
\begin{equation*}
\begin{split}
Q^{R}[g] &:=\frac{(g^{k\ell}g^{mp} - \delta^{k\ell}\delta^{mp})}{2}[(\partial_m\partial_k g_{p\ell} + \partial_m\partial_{\ell}g_{pk} - \partial_m\partial_{p}g_{k\ell}) 
\\& \qquad \qquad - (\partial_\ell\partial_k g_{pm} + \partial_\ell\partial_{m}g_{pk} - \partial_\ell\partial_p g_{km})]
\\& + \frac{g^{k\ell}}{2}\left[ (\partial_m g^{mp})(\partial_k g_{\ell p} + \partial_{\ell}g_{kp} - \partial_p g_{k\ell}) - (\partial_{\ell}g^{mp})(\partial_k g_{mp} + \partial_m g_{kp} - \partial_p g_{km}) \right]
\\& + g^{k\ell}(\Gamma_{k\ell}^q\Gamma_{qm}^m - \Gamma_{mk}^q\Gamma_{\ell q}^m).
\end{split}
\end{equation*}

Therefore, if $Y$ is the vector field on $\R^n$ given by $\sum_{j=1}^{n}Y^j\partial_j = \sum_{i,j=1}^{n}(\partial_i g_{ij} - \partial_j g_{ii})\partial_j$ and $\vec{n}$ denotes the outward unit normal to $A(0, r_1, r_2)$, we have
\begin{align*}
\int_{\S(r_2)}\sum_{i=1}^{n}&(\partial_i g_{ij} - \partial_j g_{ii})\nu^jdS - \int_{\S(r_1)}\sum_{i=1}^{n}(\partial_i g_{ij} - \partial_j g_{ii})\nu^jdS 
%%%
\\& = \int_{\partial A(0, r_1, r_2)} \langle Y, \vec{n} \rangle_{\delta} 
%%%
\\& =  \int_{A(0, r_1, r_2)} \div_{\delta}(Y)dV_\delta
%%%
\\&= \int_{A(0, r_1, r_2)} R(g) - Q^R[g]
%%%
\\& \geq \int_{A(0, r_1, r_2)} R(g) - c(n, b)|g - \delta | | \nabla^2 g |- c(n,b)|\nabla g|^2 dV_\delta,
\end{align*} 
where $\nabla$ denotes the covariant derivative with respect to $\delta$, as usual.
\end{proof}

\subsection{Some elementary analytic facts}
We first record the following variant of the Intermediate Value Theorem for integrals:
\begin{lemma}\label{lemma:integralIVT}
Let $\varphi: [a,b]\to \R^{\geq 0}$ be a nonnegative continuous function such that $\int_{a}^{b}\varphi \neq 0$, and let $f:[a,b]\to \R$ be any continuous function. Then there exists $c\in [a,b]$ such that 
\begin{equation*}
f(c) = \frac{\int_a^b f(\ell)\varphi(\ell)d\ell}{\int_a^b \varphi(\ell)d\ell}.
\end{equation*}
\end{lemma}
\begin{proof}
This follows from the usual Intermediate Value Theorem. We have
\begin{align*}
\min_{s\in [a,b]}f(s) &= \frac{\int_a^b \min_{s\in [a,b]}f(s)\varphi(\ell)d\ell}{\int_a^b \varphi(\ell)d\ell} \leq \frac{\int_a^b f(\ell)\varphi(\ell)d\ell}{\int_a^b \varphi(\ell)d\ell}
%%%
\\& \leq \frac{\int_a^b \max_{s\in [a,b]}f(s)\varphi(\ell)d\ell}{\int_a^b \varphi(\ell)d\ell} = \max_{s\in [a,b]}f(s)
\end{align*}
so the Intermediate Value Theorem applied to $f$ implies that there exists $c$ with
\begin{equation*}
f(c) = \frac{\int_a^b f(\ell)\varphi(\ell)d\ell}{\int_a^b \varphi(\ell)d\ell}.
\end{equation*}
\end{proof}

We now record the following elementary result, concerning convergence of almost-monotone sequences:
\begin{lemma}\label{lemma:monotonicityconverges}
Suppose that $(a_k)_{k=0}^{\infty}$ is a sequence of numbers satisfying
\begin{equation}\label{eq:monotonesequenceconvergence}
a_k\geq a_{k-1}- b_{k-1},
\end{equation}
for some sequence of nonnegative numbers $(b_k)_{k=0}^{\infty}$ satisfying $\sum_{i=0}^{\infty}b_i < \infty$. Then either $a_k$ converges or $a_k\to +\infty$.
\end{lemma}
\begin{proof}
Define the sequence $(A_k)_{k=1}^{\infty}$ by
\begin{equation*}
A_k := a_k + \sum_{i=0}^{k-1}b_i.
\end{equation*}
Then $(A_k)_{k=1}^{\infty}$ is (weakly) increasing. In particular, if $(A_k)$ is bounded, then it converges to some finite limit $A_{\infty}$, and if $(A_k)$ is unbounded, then it tends to $+\infty$. In the case that $(A_k)$ is bounded and hence tends to $A_\infty < \infty$, we have
\begin{equation*}
\lim_{k\to\infty}a_k = \lim_{k\to\infty} A_k - \lim_{k\to\infty}\sum_{i=0}^{k-1}b_i = A_{\infty} - \sum_{i=0}^{\infty}b_i < \infty.
\end{equation*}
In the case that $(A_k)$ is unbounded, and hence tends to $+\infty$, we have
\begin{equation*}
a_k \geq A_k - \sum_{i=0}^{\infty}b_i,
\end{equation*}
so $a_k\to +\infty$ as well.
\end{proof}

\subsection{Bilipschitz maps and $C^0$ asymptotically flat transition maps}\label{subsec:bilipschitzmaps}

For $0\leq \delta < 1$ we say that a map $\phi: D\to C$ between open subsets of normed spaces is $(1+\delta)$-bilipschitz if, for all $x,y \in D$ such that $x\neq y$, we have
\begin{equation}\label{eq:bilipschitzdefinition}
(1+\delta)^{-1} \leq \frac{|\phi(x) - \phi(y)|_C}{|x-y|_D} \leq 1+\delta.
\end{equation}
We will use this condition interchangeably with the condition that
\begin{equation*}
1-\delta \leq \frac{|\phi(x) - \phi(y)|_C}{|x-y|_D} \leq 1+\delta,
\end{equation*}
since we typically work with $\delta$ that are very small, and in this case the two conditions agree, up to multiplying $\delta$ by a constant. In this paper we will usually take $D$ and $C$ to be subsets of $\R^n$, and when we say ``bilipschitz'' we mean bilipschitz with respect to the Euclidean metric.

If $C$ and $D$ are open subsets of $\R^n$ and $\phi: D\to C$ is $C^1$, then we say that $\phi$ is \emph{locally} $(1+\delta)$-bilipschitz for some $\delta<1$ if, for all $x\in D$,
\begin{equation*}
(1 + \delta)^{-1} \leq |d\phi|\big|_x \leq 1+\delta.
\end{equation*}
As above, we use this condition interchangeably with the condition that
\begin{equation*}
1 - \delta \leq |d\phi|\big|_x \leq 1+\delta.
\end{equation*}

If $D$ and $C$ are open subsets of $\R^n$ and $\phi: D\to C$ is a $C^1$ diffeomorphism onto its image such that, for some $\delta < 1/2$,
\begin{equation*}
1 - \delta  \leq || d\phi||_{C^0(D)} \leq  1 + \delta,
\end{equation*}
then
\begin{equation}\label{eq:bilipschitzinverse}
1- 2\delta \leq || d\phi^{-1}||_{C^0(\phi(D))} \leq 1+ 2\delta.
\end{equation}
This is due to the fact that, for all $x\in D$, we have (after rotation by some orthogonal matrix $O$ depending on $x$)
\begin{equation*} 
d\phi^{-1} |_{\phi(x)} = (d\phi |_x)^{-1} = I + (I - d\phi |_x) + (I - d\phi |_x)^2 + \cdots 
\end{equation*}
so
\begin{equation*}
| d\phi^{-1} |_{\phi(x)} - I| \leq 2|I - d\phi |_x| \leq 2\delta.
\end{equation*}

In particular, suppose $C, D\subset \R^n$ are open sets and $x,y\in D$ such that $D$ contains the line from $x$ to $y$ and $C$ contains the line from $\phi(x)$ to $\phi(y)$. If $\phi:D\to C$ is a locally $(1+\delta)$-bilipschitz map for some $\delta< 1/2$ and is also a diffeomorphism, then
\begin{equation}\label{eq:localglobalbilipdiffeo}
\begin{split}
(1 + 2\delta)^{-1}|x-y| &= (1 + 2\delta)^{-1}|\phi^{-1}(\phi(x))-\phi^{-1}(\phi(y))|
%%%
\\& \leq (1 + 2\delta)^{-1} ||d\phi^{-1} ||_{C^0(C)} |\phi(x) - \phi(y)| \leq |\phi(x) - \phi(y)|
%%%
\\& \leq (1+2\delta)|x-y|.
\end{split}
\end{equation}

If $f: D\to C$ is any map such that, for all $x,y\in D$, $||f(x) - f(y)|_C- |x-y|_D| \leq \delta$, then we say $f$ is a $\delta$-isometry. We will compare $\delta$-isometries to $(1+\delta)$-bilipschitz maps in Appendix \ref{appendix:almostisometries}.

\begin{lemma}\label{lemma:generaldiffeogluing}
Let $r_0 > 0$. Suppose that $\tilde \phi: \R^n\setminus \overline{B(0, r_0)} \to \R^n$ is a local diffeomorphism such that for some $r> r_0$,
\begin{equation*}
\tilde \phi(x) \big|_{\R^n\setminus \overline{B(0, r)}} = L \big|_{\R^n\setminus\overline{B(0,r)}},
\end{equation*}
where $L:\R^n \to \R^n$ is a Euclidean isometry. Then $\tilde \phi|_{\R^n\setminus \overline{B(0, r_0)}}$ is a diffeomorphism onto its image.
\end{lemma}
\begin{proof}
Let $r'' \in (r_0, r)$. We will show that $\tilde \phi|_{\R^n\setminus B(0, r'')}$ is a degree one covering map of its image. Fix some $y\in \tilde \phi(\R^n\setminus B(0, r''))$. We first show that there is a neighborhood of $y$ that is evenly covered. Write $L(x) = Ox + v$, where $O$ is an orthogonal matrix and $v\in \R^n$. Now choose $r'> r > r''$ as follows: write $y = \tilde \phi(x)$ for some $x\in \R^n\setminus B(0, r'')$. Take $r' > \max\{|x|, r\}$. Then $y\in \tilde \phi(\overline{A(0, r'', r')})$ since $x\in \overline{A(0, r'', r')}$.

Since $\tilde \phi$ is a local diffeomorphism and $\overline{A(0, r'', r')}$ is compact, $\tilde\phi^{-1}(y)\cap \overline{A(0, r'', r')}$ is finite (otherwise, by compactness, there would exist some point $x^\infty\in \overline{A(0, r'', r')}$ such that on any neighborhood of $x_{\infty}$, $\tilde \phi$ is not injective). There is at most one point in $\R^n\setminus \overline{B(0, r')}$ that maps to $y$ under $\tilde \phi$, since $\tilde \phi$ agrees with $L$ and hence is injective in this region. Therefore we write $\tilde \phi^{-1}(y) = \{x^1, x^2, \ldots, x^k, x^{k+1}\}$, where $x^1, x^1,\ldots, x^k \in \overline{A(0, r'', r')}$ and $x^{k+1} \in \R^n\setminus \overline{B(0,r')}$. Since $\tilde \phi$ is a local diffeomorphism there exist disjoint open neighborhoods $U^i$ of each $x^i$ on which $\tilde \phi$ is a diffeomorphism. Let $V = \cap_{i=1}^{k+1}\tilde \phi(U^i)$. Then $V$ is a neighborhood of $y$ that is evenly covered by $\tilde \phi|_{\R^n\setminus B(0, r'')}$. 

Now note that since $\R^n\setminus B(0, r'')$ and hence $\tilde \phi(\R^n\setminus B(0, r''))$ is connected (because $\tilde \phi$ is continuous), to show $\tilde \phi$ is a degree one covering map it is sufficient to show that there is a point in the image of $\tilde \phi$ that is evenly covered by a single sheet. Note that, by compactness, $\tilde \phi(\overline{A(0, r'', r')})$ is bounded, but $\tilde \phi(\R^n\setminus B(0, r''))$ is unbounded because $\tilde \phi$ agrees with $L$ on $\R^n\setminus \overline{B(0,r)}$, which is unbounded. Therefore, there exists $y\in \tilde\phi(\R^n\setminus \overline{B(0,r)})$ such that $y\notin \tilde \phi(\overline{A(0, r'', r')})$. In particular, $\tilde \phi^{-1}(y)$ consists of a single point in $\R^n\setminus \overline{B(0, r)}$, since $\tilde \phi$ agrees with $L$ and hence is injective in this region. Therefore, $\tilde \phi$ is degree-one.

We have shown that $\tilde \phi|_{\R^n\setminus B(0, r'')}$ is injective for all $r'' > r_0$. In particular, if $x,y\in \R^n\setminus \overline{B(0, r_0)}$ such that $x\neq y$ then there exists $r'' > r_0$ such that $x,y\in \R^n\setminus B(0, r'')$, so $\tilde \phi(x)\neq \tilde \phi(y)$. Therefore, $\tilde \phi$ is injective on $\R^n\setminus \overline{B(0, r_0)}$ and hence is a diffeomorphism onto its image.
%To show that $\tilde\phi$ is a diffeomorphism onto its image, it is sufficient to show that it is injective (\cite[Exercise $1.3.5$]{GuilleminPallock}). 
\end{proof}

\begin{lemma}\label{lemma:transmapdomains}
Suppose that $M^n$ is a smooth manifold and $g$ is a $C^0$ metric on $M$, and that $E$ is an end of $M$. Suppose that $\Phi_1$ and $\Phi_2$ are $C^0$-asymptotically flat coordinate charts for $E$, so that there exist $c_k>0, r_k >1$ such that for all $|x| \geq r_k$, we have
\begin{equation}\label{eq:AFcondition2charts}
|((\Phi_k)_*g)_{ij}|_x - \delta_{ij}| \leq c_k |x|^{-\tau_k}.
\end{equation}
Let $\phi$ denote the transition map $\phi:= \Phi_2\circ\Phi_1^{-1}$. There exist $r_0 = r_0(\Phi_1, \Phi_2)$ and $r_0' = r_0'(\Phi_1, \Phi_2)\geq r_0$ such that: 
\begin{enumerate}
\item\label{item:transmapdefndoutsideofball}  $\phi$ is defined on $\R^n\setminus \overline{B(0, r_0/10)}$ and is a locally $(1+\tfrac{1}{2})$-bilipschitz map in this region,
\item\label{item:bilipoutsideofball} for all $r\geq r_0$, $\phi$ is locally $(1+cr^{-\min\{\tau_1, \tau_2\}})$-bilipschitz on $\R^n\setminus \overline{B(0, r/10)}$, where $c = c(\Phi_1, \Phi_2)$,
\item\label{item:imagecontainsmostofRn} $\phi(\R^n\setminus \overline{B(0, r_0/10)}) \supset \R^n\setminus \overline{B(0, \tilde r_0)}$ for some $\tilde r_0 = \tilde r_0(\Phi_1, \Phi_2)$, and
\item\label{item:imagecontainsballs} for all $r> r_0'$ and all $x\in \R^n\setminus \overline{B(0, r)}$, $B(\phi(x), r/4)\subset \phi(\R^n\setminus \overline{B(0, r_0/10)})$.
\end{enumerate}
\end{lemma}
\begin{proof}
As in the discussion preceding Definition \ref{def:C0AFend}, extend the $\Phi_1^{-1}(\overline{A(0, 2, 10^i)})$ and $\Phi_2^{-1}(\overline{A(0, 2, 10^i)})$ to compact exhaustions $(K_i)_{i=1}^{\infty}$ and $(K_i')_{i=1}^{\infty}$ respectively, of $M$. Let $(X^i)_{i=1}^{\infty}$ denote the sequence of connected components of $M\setminus K_i$ that contain $\Phi_1^{-1}(\R^n\setminus \overline{B(0, 10^i)})$, and let $(Y^i)_{i=1}^{\infty}$ denote the sequence of connected components of $M\setminus K_i'$ that contain $\Phi_2^{-1}(\R^n\setminus \overline{B(0, 10^i)})$. By the discussion preceding Definition \ref{def:C0AFend} we actually have that, for all sufficiently large $i$, $X^i = \Phi_1^{-1}(\R^n\setminus \overline{B(0, 10^i)})$ and $Y^i = \Phi_2^{-1}(\R^n\setminus \overline{B(0, 10^i)})$. Choose some $i$ sufficiently large so that this holds. Since $\Phi_1$ and $\Phi_2$ are both $C^0$-asymptotically flat coordinate charts for $E$, $(X^i)_{i=1}^{\infty}$ and $(Y^i)_{i=1}^{\infty}$ determine the same end, and hence there exists $j > i$ such that $X^i \supset Y^j$ and also there exists $k> j$ such that $Y^j \supset X^k$ (see the discussion before Definition \ref{def:C0AFend}). In particular, $\phi = \Phi_2 \circ \Phi_1^{-1}$ is defined on $\Phi_1(X^k) = \R^n\setminus \overline{B(0, 10^k)}$. Now set $r_0 > 10^{k+1}$. This proves the first part of \ref{item:transmapdefndoutsideofball}; the second part we will address after the proof of \ref{item:bilipoutsideofball}. We will increase $r_0$ throughout the proof as needed.

We now show \ref{item:bilipoutsideofball}. First, observe that, because of (\ref{eq:AFcondition2charts}) there exist $r_1'>r_1$ and $r_2' > r_2$ depending on $c_1, r_1, \tau_1$ and $c_2, r_2, \tau_2$ respectively such that, for $k=1,2$, we have that
\begin{equation*}
|| (\Phi_k)_*g- \delta||_{C^0(\R^n\setminus \overline{B(0, r_k')})} \leq \frac{1}{10}.
\end{equation*}
Increase $r_1'$ if necessary so that $r_1' > r_0$ and $\phi$ is defined on $\R^n\setminus \overline{B(0, r_1')}$. Note that by increasing $r_1'$ further we may assume that for all $x\in \R^n\setminus \overline{B(0, r_1')}$, $\phi(x)\in \R^n\setminus \overline{B(0, r_2')}$, otherwise we could find a sequence of radii $\ell_i \to \infty$ and a sequence of points points $(x_i)_{i=1}^{\infty}$ with $|x_i| > \ell_i$ but $\phi(x_i) \in \overline{B(0, r_2')} \cap \phi(\overline{A(0, r_0/10, \ell_i)})$, which is impossible because otherwise, after passing to a subsequence, we would find $\phi(x_i)\to y\in \overline{B(0, r_2')} \cap \phi(\overline{A(0, r_0/10, \ell_i)})$ and hence $x_i \to \phi^{-1}(y)$. Then, for any unit vector $v\in \R^n$ and any $x\in \R^n\setminus \overline{B(0, r_1')}$, we have
\begin{align*}
\frac{.9}{1.1} &= \frac{.9}{1.1}\delta (v,v) \leq \frac{1}{1.1}(\Phi_1)_*g(v,v)\big|_x
%%%
\\& =  \frac{1}{1.1}(\Phi_2)_*g(d_x\phi v, d_x\phi v)\big|_{\phi(x)}
%%%
\\& \leq \delta(d_x\phi v, d_x\phi v)
%%%
\\& \leq 1.1 (\Phi_2)_*g(d_x\phi v, d_x\phi v)\big|_{\phi(x)}
%%%%
\\&= 1.1 (\Phi_1)_*g(v, v)\big|_x \leq \frac{1.1}{.9}\delta (v,v) = \frac{1.1}{.9}.
\end{align*}
Therefore,
\begin{equation*}
\sqrt{.9/1.1} \leq |d_x\phi| \leq \sqrt{1.1/.9}.
\end{equation*}
In particular, if we increase $r_0$ so that $r_0 > 10 r_1'$, then for all $r>r_0$ and all $x\in \R^n\setminus \overline{B(0, r/10})$ we have
\begin{equation}\label{eq:phi(x)comparabletox}
\begin{split}
|\phi(x)| &\leq ||d\phi||_{C^0(\R^n\setminus \overline{B(0, r_1')})}\dist_{\delta}(x, \overline{B(0, r_1') }) + ||\phi||_{C^0(\partial B(0, r_1'))}
%%%
\\& \leq \frac{1.1}{.9}\dist_{\delta}(x, \overline{B(0, r_1')}) + ||\phi||_{C^0(\partial B(0, r_1'))}
%%%
\\& \leq C|x|,
\end{split}
\end{equation}
where $C$ depends on $\phi$ and $r_1'$ but not $x$. Also, using (\ref{eq:bilipschitzinverse}), we have, for all $r>r_0$ and all $x\in \R^n\setminus \overline{B(0, r/10)}$,
\begin{align*}
|x| &= |\phi^{-1}(\phi(x))| \leq ||d\phi^{-1}||\dist_{\delta}(\phi(x), \overline{B(0, r_2')}) + ||\phi^{-1}||_{C^0(\partial B(0, r_2')))}
%%%
\\& \leq C|\phi(x)|,
\end{align*}
where again $C$ depends on $r_2'$ and $\phi$ but not on $x$.

Returning to the proof of \ref{item:bilipoutsideofball} let $c = \max\{c_1, c_2\}$ and $\tau = \min\{\tau_1, \tau_2\}$. Now let $r>r_0$ and $x\in \R^n\setminus \overline{B(0, r/10)}$. The condition (\ref{eq:AFcondition2charts}) implies 
\begin{equation*}
|(\Phi_1)_*g|_x - \delta_{ij} | \leq c|x|^{-\tau}
\end{equation*}
and
\begin{equation*}
|(\Phi_2)_*g|_{\phi(x)} - \delta_{ij}| \leq c|\phi(x)|^{-\tau}.
\end{equation*}
Then, for any unit vector $v\in \R^n$ we argue as above to find
\begin{align*}
(1 + c|x|^{-\tau})^{-2} &= (1+c|x|^{-\tau})^{-2}\delta(v,v) \leq (1+ c|x|^{-\tau})^{-1}(\Phi_1)_*g(v,v)\big|_x 
%%%
\\& = (1+ c|x|^{-\tau})^{-1}(\Phi_2)_*g(d\phi v, d\phi v)\big|_{\phi(x)}
%%%
 \\& \leq \delta(d\phi_x v, d\phi_x v) 
 %%%
 \\& \leq (1 + c|\phi(x)|^{-\tau})((\Phi_2)_*g)_{\phi(x)}(d\phi_x v, d\phi_x v) 
%%%
\\& = (1 + c|\phi(x)|^{-\tau})(\Phi_1)_*g(v,v)\big|_x 
 %%%
 \\& \leq (1+c|\phi(x)|^{-\tau})^2 \delta(v, v)
 %%%
 \\&= (1+c|\phi(x)|^{-\tau})^2.
\end{align*}
In particular, for any unit vector $v\in \R^n$,
\begin{equation*}
1 - 2c|x|^{-\tau} \leq || d\phi v||_{C^0(\R^n\setminus \overline{B(0, r/10)})} \leq 1 + 2c|\phi(x)|^{-\tau}.
\end{equation*}
By (\ref{eq:phi(x)comparabletox}), this becomes
\begin{equation*}
1 - c|x|^{-\tau} \leq || d\phi v||_{C^0(\R^n\setminus \overline{B(0, r/10)})} \leq 1 + c|x|^{-\tau},
\end{equation*}
where the constant $c$ is adjusted. This establishes \ref{item:bilipoutsideofball}. The second part of \ref{item:transmapdefndoutsideofball} follows from \ref{item:bilipoutsideofball} after possibly increasing $r_0$ depending on $\max\{c_1, c_2\}$ and $\min \{\tau_1, \tau_2\}$ so that $\max\{c_1, c_2\}r^{-\min\{\tau_1, \tau_2\}} \leq \tfrac{1}{2}$.

To see why \ref{item:imagecontainsmostofRn} is true, choose $\ell \in \N$ such that $r_0/10 < 10^{\ell}$. Then $\R^n\setminus \overline{B(0, r_0/10)}\supset \R^n\setminus \overline{B(0, 10^{\ell})}$ and, by equivalence of $(X^i)$ and $(Y^i)$ there exists $m > \ell$ such that 
\begin{equation*}
\Phi_1^{-1}(\R^n\setminus \overline{B(0, 10^{\ell})}) = X^\ell \supset Y^m = \Phi_2^{-1}(\R^n\setminus \overline{B(0, 10^{m})}.
\end{equation*} 
Then we have
\begin{align*}
\phi(\R^n\setminus \overline{B(0, r_0/10)}) & \supset \phi(\R^n\setminus \overline{B(0, 10^{\ell})})
%%%
\\& = \Phi_2\circ \Phi_1^{-1}(\R^n\setminus \overline{B(0, 10^{\ell})})
%%%
\\& = \Phi_2(X^{\ell})
%%%
\\& \supset \Phi_2(Y^m)
%%%
\\&= \R^n\setminus \overline{B(0, 10^m)}
\end{align*}
so the result follows from setting $\tilde r_0 = 10^m$.

We now show \ref{item:imagecontainsballs}. Choose $r_0' \geq r_0$ depending on $c_1, r_1, \tau_1, c_2, r_2, \tau_2, \phi$ so that, for all $r> r_0'$, $cr^{-\min\{\tau_1, \tau_2\}}< 1/100$, where $c$ is as in \ref{item:bilipoutsideofball}. Then, by \ref{item:bilipoutsideofball}, $\phi$ is locally $(1+1/100)$-bilipschitz on $\R^n\setminus \overline{B(0, r_0'/10)}$. We will increase $r_0'$ further in the course of the proof.

Suppose $|x| \geq r_0' \geq r_0$. By \ref{item:imagecontainsmostofRn}, $|\phi(x)| \geq \tilde r_0$, so we have
\begin{align*}
|\phi(x)| & \geq \dist_{\delta}(\phi(x), \partial B(0, \tilde r_0))
%%%
\\& \geq (1-\tfrac{1}{100})\dist_{\delta}(x, \phi^{-1}(B(0, \tilde r_0)))
%%%
\\& \geq (1-\tfrac{1}{100})||x| - ||\phi^{-1}||_{C^0(B(0, \tilde r_0))}  |.
\end{align*}
Increase $r_0'$ depending on $||\phi^{-1}||_{C^0(B(0, \tilde r_0))}$ so that $r_0' \geq 10||\phi^{-1}||_{C^0(B(0, \tilde r_0))}$. Then, if $|x| \geq r_0'$,
\begin{equation*}
|\phi(x)| \geq (1 - \tfrac{1}{100})(9|x|/10) \geq |x|/2.
\end{equation*}
If in addition we require that $r_0' \geq 4\tilde r_0$, then it follows from the previous estimate that for $r> r_0'$ and $|x| \geq r$,
\begin{equation*}
\inf_{y\in B(\phi(x), r/4)} |y| > |\phi(x)| - r/4 \geq r/4 \geq \tilde r_0,
\end{equation*}
so
\begin{equation*}
B(\phi(x), r/4) \subset \R^n\setminus \overline{B(0, \tilde r_0)} \subset \phi(\R^n\setminus \overline{B(0, r_0/10)}).
\end{equation*}
This establishes \ref{item:imagecontainsballs}.
\end{proof}

\section{Distortion of the $C^0$ local mass along Ricci-DeTurck flow}\label{sec:distortion}
In this section we prove Lemma \ref{lemma:massdistortionestimate}. We first record some results involving the requisite evolving test function.
\begin{lemma}\label{lemma:cutofffunctiondecay}
For all $n \in \N, a > .9, b < 1.1$ there exists $\bar \theta = \bar \theta(n, a, b) >0$ such that the following is true:

Suppose $\varphi: \R\to \R^{\geq 0}$ is a smooth cutoff function with $\supp(\varphi)\subset (a,b) \subset\subset (.9, 1.1)$. For all $0 < \theta < \bar \theta$ there exists a function $\varphi_{\theta}: \R \times [0, \theta]\to \R$ such that
\begin{equation}\label{eq:radialsolvesbackwardsheat}
\begin{cases}
\partial_t \varphi_\theta(|x|,t) &= -\Delta \varphi_\theta(|x|,t) + \frac{n-1}{|x|^2}\varphi_\theta(|x|,t) \text{ for } (x,t) \in \R^n\times(0,\theta) \\
\varphi_{\theta}(\ell, \theta) &= \varphi(\ell).
\end{cases}
\end{equation}
Moreover, $\varphi_\theta(\ell, t)\geq 0$ for all $\ell\in \R$ and $t\in [0, \theta]$, and
\begin{equation}\label{eq:boundarydecay}
\begin{split}
\sup_{(x,t)\in \partial A(o, .9, 1.1) \times [0, \theta]} |\varphi_\theta(|x|,t)| &\leq \frac{c(n, ||\varphi||_{C^0(\R)})}{\theta^{n/2}}\exp\left(-\frac{d_{a,b}^2}{4\theta}\right)\\
\sup_{(x,t)\in \partial A(o, .9, 1.1) \times [0, \theta]} |\varphi'_\theta(|x|,t)| &\leq \frac{c(n, ||\varphi||_{C^1(\R)})}{\theta^{n/2}}\exp\left(-\frac{d_{a,b}^2}{4\theta}\right),
\end{split}
\end{equation}
where $d_{a,b} = \min\{a - .9, 1.1 - b\}$ and $\varphi_\theta'(\ell, t) := \frac{\partial}{\partial \ell}\varphi_\theta(\ell, t)$.
\end{lemma}

\begin{proof}
Let $F[\varphi](\ell) = \int_0^{\ell} \varphi(s)ds$, and let $u$ solve
\begin{equation*}
\begin{cases}
\partial_t u(x,t) &= \Delta u(x,t)\\
u(x,0) &= F[\varphi](|x|).
\end{cases}
\end{equation*}
Since $x\mapsto F[\varphi](|x|)$ is a spherically symmetric function, $u(\cdot, t)$ is spherically symmetric for all $t>0$, and we may write $u(x,t) = \tilde u(|x|, t)$ for all $x\in \R^n$, where $\tilde u: \R\times \R^{\geq 0}\to \R$ is smooth. Let $\varphi_\theta(\ell, t) = \partial_{\ell}\tilde u(\ell, \theta-t)$. Now observe that, because $u$ is spherically symmetric,
\begin{align*}
\partial_t\partial_{|x|} u(x,t) &= \partial_{|x|}\partial_t u(x,t) = \partial_{|x|}\left(\partial^2_{|x|} + \frac{n-1}{|x|}\partial_{|x|}\right) u(x,t)
\\&= \left(\partial_{|x|}^3 + \frac{n-1}{|x|}\partial^2_{|x|} - \frac{n-1}{|x|^2}\partial_{|x|}\right)u(x,t)
\\&= \Delta \partial_{|x|}u(x,t) - \frac{n-1}{|x|^2}\partial_{|x|}u(x,t),
\end{align*}
where we use $\partial_{|x|}$ to denote the radial derivative. Therefore,
\begin{equation*}
\partial_t \varphi_\theta(|x|,t) = (-\partial_t\partial_{|x|}u)(x, \theta - t) = -\Delta \partial_{|x|}u(x,\theta - t) + \frac{n-1}{|x|^2}\partial_{|x|}u(x,\theta - t)
\end{equation*}
so
\begin{equation*}
\begin{cases}
\partial_t \varphi_\theta(|x|,t) &= -\Delta \varphi_\theta(|x|, t) + \frac{n-1}{|x|^2}\varphi_\theta(|x|,t)\\
\varphi_\theta(\ell,\theta) &= \partial_{\ell}\tilde u(\ell, 0) = F[\varphi]'(\ell) =\varphi(\ell).
\end{cases}
\end{equation*}
This proves (\ref{eq:radialsolvesbackwardsheat}). Let $\Phi_t(x,y)$ denote the usual scalar heat kernel on Euclidean space. To show that $\varphi_\theta(\ell, t) \geq 0$, observe that for all $x\in \R^n$,
 \begin{align*}
 \varphi_\theta(|x|,t) &= \partial_{|x|}u(x,\theta - t) = \partial_{|x|} \Phi_{\theta - t}* F[\varphi](|\cdot|)\big|_x
 %%%
 \\& = \int_{\R^n}\frac{1}{(4\pi (\theta - t))^{n/2}}\exp\left(-\frac{|x-y|^2}{4(\theta - t)}\right)\varphi(|y|)\frac{x\cdot y}{|x||y|}dy
 %%%
 \\& = \int_{\{y: x\cdot y \geq 0\}}\frac{1}{(4\pi (\theta - t))^{n/2}}\exp\left(-\frac{|x-y|^2}{4(\theta - t)}\right)\varphi(|y|)\frac{x\cdot y}{|x||y|}dy
 \\& + \int_{\{y: x\cdot y \leq 0\}}\frac{1}{(4\pi (\theta - t))^{n/2}}\exp\left(-\frac{|x-y|^2}{4(\theta - t)}\right)\varphi(|y|)\frac{x\cdot y}{|x||y|}dy
 %%%
 \\&= \int_{\{y: x\cdot y \geq 0\}}\frac{1}{(4\pi (\theta - t))^{n/2}}\exp\left(-\frac{|x-y|^2}{4(\theta - t)}\right)\varphi(|y|)\frac{x\cdot y}{|x||y|}dy
 \\& +(-1)^n \int_{\{y: x\cdot y \geq 0\}}\frac{1}{(4\pi (\theta - t))^{n/2}}\exp\left(-\frac{|x+y|^2}{4(\theta - t)}\right)\varphi(|y|)\frac{x\cdot (-y)}{|x||y|}dy
 %%%
\\& \geq \int_{\{y: x\cdot y \geq 0\}}\frac{1}{(4\pi (\theta - t))^{n/2}}\left[\exp\left(-\frac{|x-y|^2}{4(\theta - t)}\right) - \exp\left(-\frac{|x+y|^2}{4(\theta - t)}\right) \right]\varphi(|y|)\frac{x\cdot y}{|x||y|}dy
%%%
\\& \geq 0,
\end{align*}
where the last step is due to the fact that if $x\cdot y \geq 0$, we have
\begin{equation*}
|x+y|^2 = |x|^2 + |y|^2 + 2x\cdot y \geq |x|^2 + |y|^2 - 2x\cdot y = |x-y|^2
\end{equation*}
and hence
\begin{equation*}
\exp\left(-\frac{|x+y|^2}{4(\theta - t)}\right) \leq  \exp\left(-\frac{|x-y|^2}{4(\theta - t)}\right).
\end{equation*}

 Towards (\ref{eq:boundarydecay}), observe that for all $x\in \partial A(0, .9, 1.1)$ we have
\begin{align*}
|\varphi_\theta(|x|,t)| &= |\partial_{|x|}u(x, \theta - t)| = \left|\frac{x^i}{|x|}\partial_i \left(\Phi_{\theta - t} * F[\varphi](|\cdot|)\right) \right|= \left|\frac{x^i}{|x|}\int_{\R^n}\Phi_{\theta - t}(x,y)\partial_i F[\varphi](|y|)dy\right|
\\&= \left|\sum_{i=1}^{n} \frac{x^i}{|x|}\int_{A(o,a, b)} \frac{1}{(4\pi(\theta -t))^{n/2}}\exp\left(-\frac{|x-y|^2}{4(\theta - t)}\right)\varphi(|y|)\frac{y^i}{|y|}dy\right|
\\& \leq c(n, ||\varphi||_{C^0(\R)})\int_{A(o,a, b)}\frac{1}{(\theta - t)^{n-2}}\exp\left(-\frac{d_{a,b}^2}{4(\theta - t)}\right)dy 
%%%
\\& \leq \frac{c(n, ||\varphi||_{C^0(\R)}(1.1)^n)}{(\theta - t)^{n/2}}\exp\left(-\frac{d_{a,b}^2}{4(\theta - t)}\right).
\end{align*} 
Now observe that $\frac{d}{ds}\left[ \frac{1}{s^{n/2}}\exp\left(-\frac{d_{a,b}^2}{4s}\right)\right] \geq 0$ on $(0, d^2_{a,b}/(2n))$. In particular, if $0< \theta < \bar \theta < d_{a,b}^2/(2n)$ then, for $t\in [0,\theta]$, we have
 \begin{equation*}
 \frac{c(n, ||\varphi||_{C^0(\R)})(1.1)^n}{(\theta - t)^{n/2}}\exp\left(-\frac{d_{a,b}^2}{4(\theta - t)}\right) \leq \frac{c(n, ||\varphi||_{C^0(\R)})}{\theta^{n/2}}\exp\left(-\frac{d_{a,b}^2}{4\theta}\right).
 \end{equation*}
 The argument to show the second estimate is similar, since
 \begin{align*}
&| \varphi_\theta'(|x|, t)| = |\partial_{|x|}^2 u(x, \theta - t)| 
%%%
\\& = \left|\frac{x^j}{|x|}\partial_j \left[\sum_{i=1}^{n} \frac{x^i}{|x|}\int_{A(o,a, b)} \frac{1}{(4\pi(\theta -t))^{n/2}}\exp\left(-\frac{|x-y|^2}{4(\theta - t)}\right)\varphi(|y|)\frac{y^i}{|y|}dy\right]\right|
%%%
\\&\leq  \bigg| \sum_{j=1}^{n}\frac{x^j}{|x|}\sum_{i=1}^{n}\left(\frac{\delta_j^i}{|x|} - \frac{x^ix^j}{|x|^3}\right)\int_{A(o,a, b)} \frac{1}{(4\pi(\theta -t))^{n/2}}\exp\left(-\frac{|x-y|^2}{4(\theta - t)}\right)\varphi(|y|)\frac{y^i}{|y|}dy \bigg|
\\& + \bigg| \sum_{j=1}^{n}\frac{x^j}{|x|}\sum_{i=1}^{n}\frac{x^i}{|x|}\int_{A(o,a, b)} \frac{1}{(4\pi(\theta -t))^{n/2}}\exp\left(-\frac{|x-y|^2}{4(\theta - t)}\right)\left( \varphi'(|y|)\frac{y^jy^i}{|y|^2} + \varphi(|y|)\frac{\delta_j^i}{|y|} - \varphi(|y|)\frac{y^iy^j}{|y|^3} \right)dy \bigg|
%%%
\\& \leq \frac{c(n)}{|x|}\left| \int_{A(o,a, b)} \frac{1}{(4\pi(\theta -t))^{n/2}}\exp\left(-\frac{|x-y|^2}{4(\theta - t)}\right)\varphi(|y|)\frac{y^i}{|y|}dy \right|
\\& + c(n)|| \varphi'||_{C^0(\R)} \int_{A(0, a, b)}\frac{1}{(4\pi(\theta -t))^{n/2}}\exp\left(-\frac{|x-y|^2}{4(\theta - t)}\right)dy 
\\& + \sup_{y\in A(0, a, b)}\frac{c(n, ||\varphi||_{C^0(\R)})}{|y|}\int_{A(0, a, b)}\frac{1}{(4\pi(\theta -t))^{n/2}}\exp\left(-\frac{|x-y|^2}{4(\theta - t)}\right)dy
%%%
\\& \leq \frac{ c(n, ||\varphi||_{C^1(\R)})}{\theta^{n/2}}\exp\left( -\frac{d^2_{a,b}}{4\theta} \right),
 \end{align*}
 arguing as before. 
 \end{proof}

\begin{lemma}\label{lemma:massdistortionscale1}
There exists $\bar \theta = \bar \theta(n)$ such that for all $\theta < \bar \theta$ the following is true: 

Suppose $g_0$ is a $C^0$ metric on $\R^n$ such that $|| g_0 - \delta||_{C^0(\R^n)} < \varepsilon < 1$ and for which there is a smooth Ricci-DeTurck flow $(g_t)_{t>0}$ satisfying (\ref{eq:RDTFinitialcondition}), (\ref{eq:RDTFXest}), and (\ref{eq:RDTFderivests}). If $\varphi: \R\to \R^{\geq 0}$ is a smooth cutoff function with $\supp(\varphi) \subset (a,b)\subset\subset (.9, 1.1)$ and $\varphi_{\theta}(\ell,t)$ is the function corresponding to $\varphi$ and $\theta$ given by Lemma \ref{lemma:cutofffunctiondecay}, then 
\begin{equation}
\begin{split}
&  \int_{\theta_0}^{\theta}\bigg|\frac{d}{dt}\left[M_{C^0}(g_t ,\varphi_\theta(t), 1)\left(\int_{.9}^{1.1}\varphi_\theta(\ell, t)d\ell\right)\right]\bigg|dt
%%%%
\\& \leq \frac{c(n)}{\theta^{n/2}}\exp\left(-\frac{d_{a,b}^2}{4\theta}\right) + c(n, \varphi)\int_0^{\theta}\int_{A(o,.9, 1.1)}|\nabla h|^2
%%%
\\& \qquad \qquad + c(n, \varphi)\int_0^{\theta}\int_{A(o,.9, 1.1)}|h||\nabla h|,
\end{split}
\end{equation}
for all $\theta_0 \in [0, \theta)$, where $d_{a,b} = \min\{a - .9, 1.1- b\}$.
\end{lemma}

\begin{proof}
Let $\bar \theta$ be as in Lemma \ref{lemma:cutofffunctiondecay}. Write $\partial_t \varphi_\theta (|x|, t) = -\Delta \varphi_\theta(|x|, t) + f(|x|)$, where $f(\ell) = \frac{n-1}{\ell^2}\varphi_\theta(\ell, t)$, and let  $h_{ij} = g_{ij} - \delta_{ij}$. We write $\varphi'_{\theta}(\ell, t)$ for $\partial_{\ell}\varphi_{\theta}(\ell, t)$ and $f'(\ell)$ for $\partial_{\ell}f(\ell)$. We use (\ref{eq:heveuclbackground}) to obtain:
\begin{align*}
 &\int_{\theta_0}^{\theta}\frac{d}{dt}\left[M_{C^0}(g_t ,1, \varphi_\theta(t))\left(\int_{.9}^{1.1}\varphi_\theta(\ell, t)d\ell\right)\right]dt
 %%%
 \\&= \int_{\theta_0}^{\theta}\frac{d}{dt}\int_{A(o,.9, 1.1)} \left( \frac{n-2}{|x|}\varphi_\theta(|x|, t) +\varphi_\theta'(|x|, t)\right)\tr_\delta h + \left(\frac{1}{|x|}\varphi_\theta(|x|,t) - \varphi_\theta'(|x|, t)\right)h_{ij}\frac{x^ix^j}{|x|^2}dxdt
\\& + \int_{\partial A(o,.9, 1.1)} h_{ij}\frac{x^i}{|x|}\varphi_\theta(|x|, t)\nu^j - h_{jj}\frac{x^i}{|x|}\varphi_\theta(|x|,t)\nu^i dS \bigg|_{t= \theta_0}^{\theta}
%%%
\\&=  \int_{\theta_0}^{\theta}\int_{A(o, .9, 1.1)} \left( \frac{n-2}{|x|} \left[-\Delta \varphi_\theta(|x|, t) + f(|x|)\right] + \left[-\Delta\varphi_\theta'(|x|, t) + \frac{n-1}{|x|^2}\varphi_\theta'(|x|, t) + f'(|x|)\right]\right)\tr_\delta h 
\\& + \int_{\theta_0}^{\theta}\int_{A(o,.9, 1.1)} \left( \frac{n-2}{|x|}\varphi_\theta(|x|, t) +\varphi_\theta'(|x|, t)\right)\tr_\delta [\Delta h + \nabla h * \nabla h + \nabla(h*\nabla h)] 
\\& + \int_{\theta_0}^{\theta}\int_{A(o,.9, 1.1)}\left(\frac{1}{|x|}\left[-\Delta \varphi_\theta(|x|, t) + f(|x|)\right] - \left[-\Delta\varphi_\theta'(|x|, t) + \frac{n-1}{|x|^2}\varphi_\theta'(|x|, t) + f'(|x|)\right]\right)h_{ij}\frac{x^ix^j}{|x|^2}dxdt
\\& +  \int_{\theta_0}^{\theta}\int_{A(o,.9, 1.1)}\left(\frac{1}{|x|}\varphi_\theta(|x|,t) - \varphi_\theta'(|x|, t)\right)[\Delta h + \nabla h * \nabla h + \nabla(h*\nabla h)]_{ij}\frac{x^ix^j}{|x|^2}dxdt
\\& + \int_{\partial A(o,.9, 1.1)} h_{ij}\frac{x^i}{|x|}\varphi_\theta(|x|, t)\nu^j - h_{jj}\frac{x^i}{|x|}\varphi_\theta(|x|,t)\nu^i dS \bigg|_{t= \theta_0}^{\theta}
%%%
\\&= \int_{\theta_0}^{\theta}\int_{A(o, .9, 1.1)} \left( \frac{n-2}{|x|} \left[-\Delta \varphi_\theta(|x|, t) + f(|x|)\right] + \left[-\Delta\varphi_\theta'(|x|, t) + \frac{n-1}{|x|^2}\varphi_\theta'(|x|, t) + f'(|x|)\right]\right)\tr_\delta h 
\\& + \Delta\left( \frac{n-2}{|x|}\varphi_\theta(|x|, t) + \varphi_\theta'(|x|, t)\right)\tr_\delta h
\\& + \left(\frac{1}{|x|}\left[-\Delta\varphi_\theta(|x|, t) + f(|x|)\right]- \left[-\Delta\varphi_\theta'(|x|, t) + \frac{n-1}{|x|^2}\varphi_\theta'(|x|, t) + f'(|x|)\right]\right)h_{ij}\frac{x^ix^j}{|x|^2} 
\\& + \Delta\left(\left(\frac{1}{|x|}\varphi_\theta(|x|, t) - \varphi_\theta'(|x|, t)\right)\frac{x^ix^j}{|x|^2}\right)h_{ij}
 dxdt
\\& + \int_{\partial A(o,.9, 1.1)} h_{ij}\frac{x^i}{|x|}\varphi_\theta(|x|, t)\nu^j - h_{jj}\frac{x^i}{|x|}\varphi_\theta(|x|,t)\nu^i dS \bigg|_{t= \theta_0}^{\theta}
\\& + \int_{\theta_0}^{\theta}\int_{A(o,.9,1.1)} \left( \frac{n-2}{|x|}\varphi_\theta(|x|, t) +\varphi_\theta'(|x|, t)\right) \nabla h *_\delta \nabla h 
\\& + \int_{\theta_0}^{\theta} \int_{A(o,.9,1.1)} \nabla \left( \frac{n-2}{|x|}\varphi_\theta(|x|, t) +\varphi_\theta'(|x|, t)\right) h *_\delta \nabla h
\\& + \int_{\theta_0}^{\theta}\int_{A(o,.9,1.1)}\left(\frac{1}{|x|}\varphi_\theta(|x|,t) - \varphi'_\theta(|x|, t)\right)\frac{x^ix^j}{|x|^2} \nabla h *_\delta \nabla h 
\\& + \int_{\theta_0}^{\theta}\int_{A(o,.9,1.1)} \nabla \left(\left(\frac{1}{|x|}\varphi_\theta(|x|,t) - \varphi_\theta'(|x|, t)\right)\frac{x^ix^j}{|x|^2} \right) h *_\delta \nabla h
%%%
\\&= \int_{\theta_0}^{\theta}\int_{A(o, .9, 1.1)} \left( \frac{n-2}{|x|} \left[-\Delta \varphi_\theta(|x|, t) + f(|x|)\right] + \left[-\Delta\varphi_\theta'(|x|, t) + \frac{n-1}{|x|^2}\varphi_\theta'(|x|, t) + f'(|x|)\right]\right)\tr_\delta h 
\\& + \Delta\left( \frac{n-2}{|x|}\varphi_\theta(|x|, t) + \varphi_\theta'(|x|, t)\right)\tr_\delta h
\\& + \left(\frac{1}{|x|}\left[-\Delta\varphi_\theta(|x|, t) +f(|x|)\right]- \left[-\Delta\varphi_\theta'(|x|, t) + \frac{n-1}{|x|^2}\varphi_\theta'(|x|, t) + f'(|x|)\right]\right)h_{ij}\frac{x^ix^j}{|x|^2} 
\\& + \Delta\left(\left(\frac{1}{|x|}\varphi_\theta(|x|, t) - \varphi_\theta'(|x|, t)\right)\frac{x^ix^j}{|x|^2}\right)h_{ij}
 dxdt + C
%%%
\\& =: \int_{\theta_0}^{\theta} \int_{A(0, .9, 1.1)} A\tr_\delta hdxdt + \int_{\theta_0}^{\theta}\int_{A(0, .9, 1.1)} Bx^ix^jh_{ij}dxdt  + C,
\end{align*}
where
\begin{equation*}
\begin{split}
|C| &\leq \frac{c(n)}{\theta^{n/2}}\exp\left(-\frac{d_{a,b}^2}{4\theta}\right) + c(n, ||\varphi||_{C^1})\int_0^{\theta}\int_{A(o,.9, 1.1)}|\nabla h|^2 
\\& \qquad \qquad + c(n, ||\varphi||_{C^2})\int_0^{\theta}\int_{A(o,.9, 1.1)}|h||\nabla h|, 
\end{split}
\end{equation*}
due to the fact that, by (\ref{eq:boundarydecay}) and (\ref{eq:RDTFXest}), $||h||_{C^0(\R^n\times[0,\theta])} \leq c(n)$. Therefore it is sufficient to show that $A = B = 0$.

Towards this objective, observe that 
\begin{align*}
\Delta\bigg(\bigg(\frac{1}{|x|}\varphi_\theta(|x|, t) - \varphi_\theta'(|x|, t)\bigg)&\frac{x^ix^j}{|x|^2}\bigg)h_{ij} = \Delta\left(\frac{1}{|x|^3}\varphi_\theta(|x|, t) - \frac{1}{|x|^2}\varphi_\theta'(|x|, t)\right)x^ix^jh_{ij} 
\\& + 2\nabla\left(\frac{1}{|x|^3}\varphi_\theta(|x|, t) - \frac{1}{|x|^2}\varphi_\theta'(|x|, t)\right)\nabla(x^ix^j)h_{ij} 
\\& + \left(\frac{1}{|x|^3}\varphi_\theta(|x|, t) - \frac{1}{|x|^2}\varphi_\theta'(|x|, t)\right)\Delta(x^ix^j)h_{ij}
%%%
\\&= \Delta\left(\frac{1}{|x|^3}\varphi_\theta(|x|, t) - \frac{1}{|x|^2}\varphi_\theta'(|x|, t)\right)x^ix^jh_{ij} 
\\& + 2\frac{d}{d\ell}\bigg|_{\ell = |x|}\left(\frac{1}{\ell^3}\varphi_\theta(\ell, t) - \frac{1}{\ell^2}\varphi_\theta'(\ell, t)\right)(\frac{x^k(\delta_k^ix^j + x^i\delta_k^j)}{|x|})h_{ij}
\\& + 2\left(\frac{1}{|x|^3}\varphi_\theta(|x|, t) - \frac{1}{|x|^2}\varphi_\theta'(|x|, t)\right)\tr_{\delta}h.
\end{align*}
Therefore,
\begin{align*}
A&= \left( \frac{n-2}{|x|} \left[-\Delta \varphi_\theta(|x|, t) + f(|x|)\right] + \left[-\Delta\varphi_\theta'(|x|, t) + \frac{n-1}{|x|^2}\varphi_\theta'(|x|, t) + f'(|x|)\right]\right)
\\ & + \Delta\left( \frac{n-2}{|x|}\varphi_\theta(|x|, t) + \varphi_\theta'(|x|, t)\right) + 2\left(\frac{1}{|x|^3}\varphi_\theta(|x|, t) - \frac{1}{|x|^2}\varphi_\theta'(|x|, t)\right),\\
%%%%
B&= \left(\frac{1}{|x|}\left[-\Delta\varphi_\theta(|x|, t) + f(|x|)\right]- \left[-\Delta\varphi_\theta'(|x|, t) + \frac{n-1}{|x|^2}\varphi_\theta'(|x|, t) + f'(|x|)\right]\right)\frac{1}{|x|^2}
\\&+ \Delta\left(\frac{1}{|x|^3}\varphi_\theta(|x|, t) - \frac{1}{|x|^2}\varphi_\theta'(|x|, t)\right) + 4\frac{d}{d\ell}\bigg|_{\ell = |x|}\left(\frac{1}{\ell^3}\varphi_\theta(\ell, t) - \frac{1}{\ell^2}\varphi_\theta'(\ell, t)\right) \frac{1}{|x|} 
\end{align*}

\textbf{Term $A$.} We have
\begin{align*}
 A &= \frac{n-2}{|x|}\left(-\Delta\varphi_\theta(|x|) + f(|x|)\right) - \Delta \varphi_\theta'(|x|) + \frac{n-1}{|x|^2}\varphi_\theta'(|x|) + f'(|x|) 
 \\& \qquad \qquad + \Delta\left( \frac{n-2}{|x|}\varphi_\theta(|x|) + \varphi_\theta'(|x|)\right) + \frac{2}{|x|^3}\varphi_\theta(|x|) - \frac{2}{|x|^2}\varphi_\theta'(|x|)
 %%%
 \\&= \frac{n-2}{|x|}f(|x|) + \frac{n-1}{|x|^2}\varphi'_{\theta}(|x|) + f'(|x|) + \Delta\left(\frac{n-2}{|x|}\right)\varphi_{\theta}(|x|)
 + 2\nabla\left(\frac{n-2}{|x|}\right)\cdot \nabla \varphi_{\theta}(|x|)
 \\& \qquad \qquad+ \frac{2}{|x|^3}\varphi_\theta(|x|) - \frac{2}{|x|^2}\varphi_\theta'(|x|)
 %%%
 \\&= \frac{n-2}{|x|}f(|x|) + \frac{(n-1)}{|x|^2}\varphi_\theta'(|x|) + f'(|x|)  -\frac{(n-2)(n-3)}{|x|^3}\varphi_\theta(|x|) - \frac{2(n-2)}{|x|^2}\varphi_\theta'(|x|) 
 \\& \qquad \qquad + \frac{2}{|x|^3}\varphi_\theta(|x|) - \frac{2}{|x|^2}\varphi_\theta'(|x|)
 %%%
 \\&= \frac{n-2}{|x|}f(|x|) + f'(|x|) - \frac{n-1}{|x|^2}\varphi_\theta'(|x|) - \frac{(n-4)(n-1)}{|x|^3}\varphi_\theta(|x|).
\end{align*}
If we take $f(\ell) = \frac{(n-1)}{\ell^2}\varphi_{\theta}(\ell)$ so that $f'(\ell) = -\frac{2(n-1)}{\ell^3}\varphi_\theta(\ell) + \frac{(n-1)}{\ell^2}\varphi_\theta'(\ell)$, then we find
\begin{align*}
\frac{n-2}{|x|}f(|x|) + f'(|x|) &= \frac{(n-2)(n-1)}{|x|^3}\varphi_\theta(|x|) - \frac{2(n-1)}{|x|^3}\varphi_\theta(|x|) + \frac{(n-1)}{|x|^2}\varphi_\theta'(|x|)
\\&= \frac{(n-4)(n-1)}{|x|^3}\varphi_\theta(|x|) + \frac{(n-1)}{|x|^2}\varphi_\theta'(|x|),
\end{align*}
so $A= 0$.

\textbf{Term $B$.} We have
\begin{align*}
B &=  \left(\frac{1}{|x|}\left[-\Delta\varphi_\theta(|x|, t) + f(|x|)\right]- \left[-\Delta\varphi_\theta'(|x|, t) + \frac{n-1}{|x|^2}\varphi_\theta'(|x|, t) + f'(|x|)\right]\right)\frac{1}{|x|^2}
\\&+ \Delta\left(\frac{1}{|x|^3}\varphi_\theta(|x|, t) - \frac{1}{|x|^2}\varphi_\theta'(|x|, t)\right) + 4\frac{d}{d\ell}\bigg|_{\ell = |x|}\left(\frac{1}{\ell^3}\varphi_\theta(\ell, t) - \frac{1}{\ell^2}\varphi_\theta'(\ell, t)\right) \frac{1}{|x|} 
%%%
\\& = \frac{1}{|x|^3}f(|x|) - \frac{n-1}{|x|^4}\varphi'_\theta(|x|, t) - \frac{1}{|x|^2}f'(|x|)
\\& + \Delta\left(\frac{1}{|x|^3}\right)\varphi_\theta(|x|, t) + 2\nabla \left(\frac{1}{|x|^3}\right)\cdot\nabla\varphi_\theta(|x|,t) - \Delta\left(\frac{1}{|x|^2}\right)\varphi'_\theta(|x|, t) 
\\& - 2\nabla\left(\frac{1}{|x|^2}\right)\cdot\nabla \varphi'_\theta(|x|, t)  
\\& + 4\left(-3\frac{1}{\ell^4}\varphi_\theta(\ell, t) + \frac{1}{\ell^3}\varphi'_\theta(\ell, t) + 2\frac{1}{\ell^3}\varphi'_\theta(\ell, t) - \frac{1}{\ell^2}\varphi''_\theta(\ell, t)\right)\bigg|_{\ell = |x|}\frac{1}{|x|}
%%%
\\& =  \frac{1}{|x|^3}f(|x|) - \frac{n-1}{|x|^4}\varphi'_\theta(|x|, t) - \frac{1}{|x|^2}f'(|x|)
\\& + \frac{15 - 3n}{|x|^5}\varphi_\theta(|x|, t) -\frac{6}{|x|^4}\varphi'_\theta(|x|, t) - \frac{8 - 2n}{|x|^4}\varphi'_\theta(|x|, t)+ \frac{4}{|x|^3}\varphi''_\theta(|x|, t)
\\& -\frac{12}{|x|^5}\varphi_\theta(|x|, t) + \frac{12}{|x|^4}\varphi'_\theta(|x|, t) -\frac{4}{|x|^3}\varphi''_\theta(|x|, t)
%%%
\\& = \frac{1}{|x|^3}f(|x|) - \frac{1}{|x|^2}f'(|x|) + \frac{3(1 - n)}{|x|^5}\varphi_\theta(|x|, t) + \frac{n-1}{|x|^4}\varphi'_\theta(|x|, t)
\end{align*}
Taking $f(\ell) = \frac{(n-1)}{\ell^2}\varphi_{\theta}(\ell)$ as above, we find
\begin{align*}
\frac{1}{|x|^3}f(|x|) &- \frac{1}{|x|^2}f'(|x|) 
\\& = \frac{1}{|x|^3}\left(\frac{(n-1)}{|x|^2}\varphi_{\theta}(|x|)\right) - \frac{1}{|x|^2}\left(-\frac{2(n-1)}{|x|^3}\varphi_{\theta}(|x|) + \frac{(n-1)}{|x|^2}\varphi_{\theta}'(|x|)\right)
\\&= \frac{(n-1)}{|x|^5}\varphi_{\theta}(|x|) + \frac{2(n-1)}{|x|^5}\varphi_{\theta}(|x|) -\frac{(n-1)}{|x|^4}\varphi_{\theta}'(|x|)
\\&= \frac{3(n-1)}{|x|^5}\varphi_{\theta}(|x|) - \frac{(n-1)}{|x|^4}\varphi_{\theta}'(|x|),
\end{align*}
so $B= 0$ as well.
\end{proof}

\begin{corollary}\label{cor:numeratordistortion}
There exists $\bar r = \bar r(n, \eta)$ such that for all $r > \bar r$ the following is true:

Suppose $g_0$ is a $C^0$ metric on $\R^n$ such that $|| g_0 - \delta||_{C^0(\R^n)} < \varepsilon < 1$ for some $\varepsilon$, and for which there is a smooth Ricci-DeTurck flow $(g_t)_{t>0}$ satisfying (\ref{eq:RDTFinitialcondition}), (\ref{eq:RDTFXest}), and (\ref{eq:RDTFderivests}). If $\varphi: \R\to \R^{\geq 0}$ is a smooth cutoff function with $\supp(\varphi) \subset (a,b)\subset\subset (.9, 1.1)$ and $\varphi_{r^{-\eta}}(\ell,t)$ is the function corresponding to $\varphi$ and $\theta = r^{-\eta}$ given by Lemma \ref{lemma:cutofffunctiondecay}, then
\begin{equation}
\begin{split}
&  \bigg|\int_{\theta_0}^{r^{2-\eta}}\frac{d}{dt}\left[M_{C^0}(g_t , \varphi_{r^{-\eta}}(\tfrac{t}{r^2}), r)\left(r\int_{.9}^{1.1}\varphi_{r^{-\eta}}(\ell, \tfrac{t}{r^2})d\ell\right)\right]dt \bigg|
\\&  \leq c\varepsilon^2r^{n-1} + c'(n)r^{n\eta/2 + n -1}\exp\left(-\frac{d^2_{ab}}{4}r^{\eta}\right)
\end{split}
\end{equation}
for all $\theta_0 \in [0, r^{2-\eta})$, where $c = c(n, \varphi)$ and $d_{a,b} = \min\{a- .9, 1.1-b\}$.
\end{corollary}

\begin{proof}
Let $\bar \theta$ be as in Lemma \ref{lemma:massdistortionscale1}. Let $\hat h(x,t) := h(rx, r^2t) = g(rx, r^2 t) - \delta =: \hat g(x,t)- \delta$. First note that
\begin{align*}
& \int_{\theta_0}^{r^{2-\eta}} \frac{d}{dt}\left[M_{C^0}(g_t, \varphi_{r^{-\eta}}(\cdot, \tfrac{t}{r^2}), r)\left(r\int_{.9}^{1.1}\varphi_{r^{-\eta}}(\ell, \tfrac{t}{r^2})d\ell\right)\right] dt
%%%
\\& = \int_{\theta_0}^{r^{2-\eta}}\frac{d}{dt}\int_{A(o,.9r, 1.1r)} \left(\frac{n-2}{|x|}\varphi_{r^{-\eta}}(\tfrac{|x|}{r}, \tfrac{t}{r^2}) + \frac{1}{r}\varphi_{r^{-\eta}}'(\tfrac{|x|}{r}, \tfrac{t}{r^2})\right)\tr h(x,t)
\\& \qquad \qquad + \left( \frac{1}{|x|}\varphi_{r^{-\eta}}(\tfrac{|x|}{r}, \tfrac{t}{r^2}) - \frac{1}{r}\varphi_{r^{-\eta}}'(\tfrac{|x|}{r}, \tfrac{t}{r^2})\right)\frac{h_{ij}(x,t)x^ix^j}{|x|^2}dxdt
\\& \qquad \qquad+ \int_{\theta_0}^{r^{2-\eta}}\frac{d}{dt}\int_{\partial A(o, .9r, 1.1r)} h_{ij}\frac{x^i}{|x|}\varphi_{r^{-\eta}}(\tfrac{|x|}{r}, \tfrac{t}{r^2})\nu^j - h_{jj}\frac{x^i}{|x|}\varphi_{r^{-\eta}}(\tfrac{|x|}{r}, \tfrac{t}{r^2})\nu^i dS dt
%%%
\\&= r^{n-1}\int_{\theta_0/r^2}^{r^{-\eta}} \frac{d}{dt} \int_{A(o,.9, 1.1)} \left(\frac{n-2}{|x|}\varphi_{r^{-\eta}}(|x|, t) + \varphi_{r^{-\eta}}'(|x|,t)\right)\tr \hat h(x,t) 
\\& \qquad \qquad + \left(\frac{1}{|x|}\varphi_{r^{-\eta}}(|x|, t) - \varphi_{r^{-\eta}}'(|x|, t)\right)\frac{\hat h_{ij}x^ix^j}{|x|^2}dxdt
\\& \qquad \qquad + r^{n-1}\int_{\theta_0/r^2}^{r^{-\eta}} \frac{d}{dt}\int_{\partial A(o, .9, 1.1)} \hat h_{ij}\frac{x^i}{|x|}\varphi_{r^{-\eta}}(|x|, t)\nu^j - \hat h_{jj}\frac{x^i}{|x|}\varphi_{r^{-\eta}}(|x|, t)\nu^i dS dt
%%%
\\&= r^{n-1}\int_{\theta_0/r^2}^{r^{-\eta}} \frac{d}{dt}\left[M_{C^0}(\hat g_t, \varphi_{r^{-\eta}}(t), 1)\left(\int_{.9}^{1.1}\varphi_{r^{-\eta}}(\ell, t)d\ell\right)\right]dt
\end{align*}

Therefore, if $r$ is sufficiently large depending on $\eta$ and $n$ so that we may apply Lemma \ref{lemma:massdistortionscale1} with $\theta = r^{-\eta}$ and so that $r^{-\eta} < .25$, we have
\begin{align*}
& \bigg| \int_{\theta_0}^{r^{2-\eta}}  \frac{d}{dt}\left[M_{C^0}(g_t, \varphi_{r^{-\eta}}(\cdot, \tfrac{t}{r^2}), r)\left(r\int_{.9}^{1.1}\varphi_{r^{-\eta}}(\ell, \tfrac{t}{r^2})d\ell\right)\right] dt\bigg|
%%%
\\&\leq r^{n-1}\int_{\theta_0/r^2}^{r^{-\eta}} \bigg|\frac{d}{dt}\left[M_{C^0}(\hat g_t, \varphi_{r^{-\eta}}(t), 1)\left(\int_{.9}^{1.1}\varphi_{r^{-\eta}}(\ell, t)d\ell\right)\right]\bigg|dt
%%%
\\& \leq r^{n-1}\bigg[ c(n)r^{n\eta/2}\exp\left(-\frac{d_{a,b}^2}{4}r^{\eta}\right) + c(n, \varphi)\int_{0}^{r^{-\eta}}\int_{A(o,.9, 1.1)} |\nabla \hat h|^2(x,t)dxdt 
\\& \qquad \qquad + c(n, \varphi)\int_{0}^{r^{-\eta}}\int_{A(o,.9, 1.1)}|\hat h|(x,t) |\nabla \hat h|(x,t)dxdt\bigg]
%%%
\\& = r^{n-1}\bigg[ c(n)r^{n\eta/2}\exp\left(-\frac{d_{a,b}^2}{4}r^{\eta}\right) + c(n, \varphi)\int_{0}^{r^{-\eta}}\int_{A(o,.9, 1.1)} r^2|\nabla h|^2(rx, r^2t)dxdt 
\\& \qquad \qquad + c(n, \varphi)\int_{0}^{r^{-\eta}}\int_{A(o,.9, 1.1)}r|h|(rx,r^2t) |\nabla h|(rx,r^2t)dxdt\bigg]
%%%
\\&= r^{n-1}\bigg[ c(n)r^{n\eta/2}\exp\left(-\frac{d_{a,b}^2}{4}r^{\eta}\right) + c(n, \varphi)\int_{0}^{r^{2-\eta}}\int_{A(o,.9r, 1.1r)} r^{2-n-2}|\nabla h|^2(x, t)dxdt 
\\& \qquad \qquad + c(n, \varphi)\int_{0}^{r^{2-\eta}}\int_{A(o,.9r, 1.1r)}r^{1-n-2}|h|(x,t) |\nabla h|(x,t)dxdt\bigg]
%%%
\\& = c(n)r^{n\eta/2 + n - 1}\exp\left(-\frac{d_{a,b}^2}{4}r^{\eta}\right) + \frac{c(n, \varphi)}{r}\int_{0}^{r^{2-\eta}}\int_{A(o,.9r, 1.1r)}|\nabla h|^2(x,t)dxdt 
\\& \qquad \qquad + \frac{c(n, \varphi)}{r^2}\int_{0}^{r^{2-\eta}}\int_{A(o,.9r, 1.1r)}|h|(x,t)|\nabla h|(x,t)dxdt
%%%
\\& \leq c(n)r^{n\eta/2 + n - 1}\exp\left(-\frac{d_{a,b}^2}{4}r^{\eta}\right) + \frac{c(n, \varphi)}{r}\int_{0}^{.25r^2}\int_{B(z_1, .5r)}|\nabla h|^2(x,t)dxdt 
\\& \qquad \qquad + \frac{c(n, \varphi)}{r^2}\int_{0}^{.25r^2}\int_{B(z_2, .5r)}|h|(x,t)|\nabla h|(x,t)dxdt
%%%
\\& =: c(n)r^{n\eta/2 + n - 1}\exp\left(-\frac{d_{a,b}^2}{4}r^{\eta}\right) + A + B,
\end{align*}
for some $z_1, z_2\in A(0,.9r, 1.1r)$, where $z_1$ and $z_2$ are chosen as follows: let $\{z_i\}_{i=1}^{k}$ be a maximal collection of points in $A(0, .9r, 1.1r)$ such that the $B(z_i, .25 r)$ are pairwise disjoint. Observe that $\{B(z_i, .5r)\}_{i=1}^{k}$ is a cover of $A(0, .9r, 1.1r)$. Choose $z_1$ and $z_2$ from $\{z_i\}_{i=1}^{k}$ so that
\begin{align*}
\int_{0}^{.25r^2}\int_{B(z_1, .5r)}|\nabla h|^2(x,t)dxdt &=  \max_{i = 1,\ldots, k} \int_{0}^{.25r^2}\int_{B(z_i, .5r)}|\nabla h|^2(x,t)dxdt\\
\int_{0}^{.25r^2}\int_{B(z_2, .5r)}|h|(x,t)|\nabla h|(x,t)dxdt &= \max_{i=1,\ldots, k} \int_{0}^{.25r^2}\int_{B(z_i, .5r)}|h|(x,t)|\nabla h|(x,t)dxdt.
\end{align*}
We have $k\leq c(n)$ since
\begin{equation*}
k\omega_n(.25r)^n = \sum_{i=1}^k|B(z_i, .25r)| \leq |A(0, .65r, 1.35r)| \leq \omega_n(1.35r)^n,
\end{equation*} 
so
\begin{equation*}
\begin{split}
&\frac{1}{r}\int_{0}^{r^{2-\eta}}\int_{A(o,.9r, 1.1r)}|\nabla h|^2(x,t)dxdt + \frac{1}{r^2}\int_{0}^{r^{2-\eta}}\int_{A(o,.9r, 1.1r)}|h|(x,t)|\nabla h|(x,t)dxdt 
%%%
\\& \leq \frac{k}{r}\int_{0}^{.25r^2}\int_{B(z_1, .5r)}|\nabla h|^2(x,t)dxdt + \frac{k}{r^2}\int_{0}^{.25r^2}\int_{B(z_2, .5r)}|h|(x,t)|\nabla h|(x,t)dxdt.
\end{split}
\end{equation*}

Note that, by H\"older's inequality,
\begin{align*}
\frac{1}{r^2}\int_{0}^{.25r^2}\int_{B(z_2, .5r)}&|h|(x,t)|\nabla h|(x,t)dxdt 
\\& \leq \frac{c(n)}{r^2}||h||_{L^\infty(B(z_2, .5r)\times(0, .25r^2))}r^{n/2 +1}||\nabla h||_{L^2(B(z_2, .5r)\times(0, .25r^2))},
\end{align*}
so (\ref{eq:RDTFXest}) implies
\begin{align*}
A + B &\leq c(n,\varphi)r^{n-1}(r^{-n}||\nabla h||^2_{L^2(B(z_1, .5r)\times (0,.25r))}) 
\\& \qquad \qquad+ c(n, \varphi)r^{n-1}(r^{-n/2}||h||_{C^0(B(z_2, .5r)\times(0, .25r^2))}||\nabla h||_{L^2(B(z_2,.5r)\times(0,.25r^2))})
%%%
\\& \leq c(n, \varphi)||h||_X^2r^{n-1}  \leq c(n,\varphi)\varepsilon^2r^{n-1}.
\end{align*}
This completes the proof.
\end{proof}

\begin{proof}[Proof of Lemma \ref{lemma:massdistortionestimate}]
The first statement is due to Lemma \ref{lemma:cutofffunctiondecay}. Towards the second, let $\bar r$ be as in Corollary \ref{cor:numeratordistortion}. 

Now observe, by Lemma \ref{lemma:cutofffunctiondecay}:
\begin{align*}
&\left|\frac{d}{dt}\int_{.9}^{1.1}\varphi_{r^{-\eta}}(\ell, t)d\ell\right| = \left|\int_{.9}^{1.1}-\varphi_{r^{-\eta}}''(\ell, t) - \frac{n-1}{\ell}\varphi_{r^{-\eta}}'(\ell, t) + \frac{n-1}{\ell^2}\varphi_{r^{-\eta}}(\ell, t)\right|
\\&= \bigg|-\varphi_{r^{-\eta}}'(1.1, t) + \varphi_{r^{-\eta}}'(.9, t) -\left[\frac{n-1}{\ell}\varphi_{r^{-\eta}}(\ell, t)\bigg|_{.9}^{1.1} - \int_{.9}^{1.1}\frac{d}{d\ell}\left[\frac{n-1}{\ell}\right]\varphi_{r^{-\eta}}(\ell, t)d\ell\right] 
\\&+ \int_{.9}^{1.1}\frac{n-1}{\ell^2}\varphi_{r^{-\eta}}(\ell, t)d\ell\bigg|
\\&= \bigg|-\varphi_{r^{-\eta}}'(1.1, t) + \varphi_{r^{-\eta}}'(.9, t) - \frac{n-1}{1.1}\varphi_{r^{-\eta}}(1.1, t) + \frac{n-1}{.9}\varphi_{r^{-\eta}}(.9, t) 
\\& - \int_{.9}^{1.1}\frac{n-1}{\ell^2}\varphi_{r^{-\eta}}(\ell, t)d\ell  + \int_{.9}^{1.1}\frac{n-1}{\ell^2}\varphi_{r^{-\eta}}(\ell, t)d\ell \bigg|
\\& \leq \frac{c(n,\varphi)}{(r^{-\eta})^{n/2}}\exp\left(-\frac{(d_{a,b})^2}{4}r^{\eta}\right).
\end{align*}
In particular, for $t\in (0, r^{2-\eta})$, we have
\begin{equation}\label{eq:denominatordistortion}
\left|\frac{d}{dt}\int_{.9}^{1.1}\varphi_{r^{-\eta}}(\ell, \tfrac{t}{r^2})\right| \leq \frac{c(n,\varphi)}{r^2}r^{n\eta/2}\exp\left(-D(a,b))r^{\eta}\right).
\end{equation}
We also have
\begin{equation}\label{eq:roughC0massestimate}
|M_{C^0}(g_t, r, \varphi_{r^{-\eta}}(\cdot, \tfrac{t}{r^2}))| \leq c(n, \varphi)r^{n-2},
\end{equation}
for all $t>0$, by definition and (\ref{eq:RDTFXest}). 

Combining (\ref{eq:denominatordistortion}) and (\ref{eq:roughC0massestimate}) with Corollary \ref{cor:numeratordistortion} we find
\begin{align*}
 \int_{\theta_0}^{r^{2-\eta}}&\bigg| \frac{d}{dt} M_{C^0}(g_t, r, \varphi(\cdot, \tfrac{t}{r^2}))\bigg|dt 
 = \bigg| \int_{\theta_0}^{r^{2-\eta}} \frac{\frac{d}{dt}\left[M_{C^0}(g_t ,r, \varphi(\cdot, \tfrac{t}
{r^2}))\left(r\int_{.9}^{1.1}\varphi(\ell, \tfrac{t}{r^2})d\ell\right)\right]}{r\int_{.9}^{1.1}\varphi(\ell, \tfrac{t}{r^2})d\ell}dt\bigg|
 \\& + \bigg|\int_{\theta_0}^{r^{2-\eta}} \frac{M_{C^0}(g_t ,r, \varphi(\cdot, \tfrac{t}{r^2}))\left(r\int_{.9}^{1.1}\varphi(\ell, \tfrac{t}{r^2})d\ell\right)\frac{d}{dt}\left[\int_{.9}^{1.1}\varphi(\ell, \tfrac{t}{r^2})d\ell\right]}{r\left[\int_{.9}^{1.1}\varphi(\ell, \tfrac{t}{r^2})d\ell\right]^2} dt\bigg|
%%%%
\\&\leq \frac{c(\varphi)}{r}\bigg|\int_{\theta_0}^{r^{2-\eta}} \frac{d}{dt}\left[M_{C^0}(g_t ,r, \varphi(\cdot, \tfrac{t}{r^2}))\left(r\int_{.9}^{1.1}\varphi(\ell, \tfrac{t}{r^2})d\ell\right)\right]dt\bigg|
\\& +c(n,\varphi)r^{n-2}\int_{\theta_0}^{r^{2-\eta}}\bigg| \frac{d}{dt}\int_{.9}^{1.1}\varphi(\ell, \tfrac{t}{r^2})d\ell\bigg| dt
%%%
\\& \leq c(n, \varphi)\varepsilon^2r^{n-2} + c'(n)r^{n\eta/2 + n -1}\exp\left(-D(a,b)r^{\eta}\right)
\\& \qquad \qquad + c(n, \varphi)r^{n-2 +n\eta/2 - 2 + 2-\eta}\exp\left(-D(a,b)r^{\eta}\right)
%%%
\\& \leq c(n,\varphi)\varepsilon^2r^{n-2} + c(n, \varphi)r^{n-2 +n\eta/2 - 2 + 2-\eta}\exp\left(-D(a,b)r^{\eta}\right).
\end{align*}
This proves the second statement of the lemma. To prove the third, increase $\bar r$ and hence $r$ depending on $n, \eta, a, b, c_0$, and $\tau$ so that $c_0r^{-\tau} < 1$ for all $r> \bar r$, and so that 
\begin{equation*}
c(n, \varphi)r^{n-2 +n\eta/2 - 2 + 2-\eta}\exp\left(-D(a,b)r^{\eta}\right) \leq c(n, \eta, a, b, \tau) r^{n-2-2\tau},
\end{equation*}
where $c(n,\varphi)$ and $D(a,b)$ are as in the previous estimate. Then the result follows from the second statement of the lemma.
\end{proof}

\begin{remark}\label{rmk:massdistortiondiffradii}
Let $\bar r$ be as in the third statement of Lemma \ref{lemma:massdistortionestimate}, and let $g_0$ and $g_t$ be as in the hypotheses of that statement. Arguing in a similar way to the proofs of Corollary \ref{cor:numeratordistortion} and Lemma \ref{lemma:massdistortionestimate} we also have that for all $\alpha>0$,
\begin{equation*}
\int_0^{r^{2-\eta}} \left| \frac{d}{dt} M_{C^0}(g_t, \varphi_{r^{-\eta}}(\cdot, \tfrac{t}{r^2}), \alpha r) \right| \leq c(n, \varphi, c_0, \alpha)r^{n-2-2\tau}.
\end{equation*}
\end{remark}

\section{Monotonicity results for the $C^0$ local mass}\label{sec:monotonicity}

We first record the following lemma, which is essentially a special case of \cite[Lemma $5.9$]{PBGthesis} in the case of Ricci-DeTurck flows with a Euclidean background. 
\begin{lemma}\label{lemma:earlierscalarcurvature}
There exists $\bar \theta = \bar\theta(\kappa_0, n, \beta)$ such that for all $\theta < \bar \theta$, the following is true:

Suppose that $(g_t)_{t>0}$ is some smooth family of Riemannian metrics on $\R^n$ evolving by (\ref{eq:RDTFeucl}) and satisfying $|| g_t - \delta||_{C^0(\R^n)}< 1$ and (\ref{eq:RDTFderivests}). Suppose that for some $\kappa_0 \in \R$, $\beta\in (0, 1/2)$, $x\in \R^n$, and $0< \theta < \bar \theta$, the condition (\ref{eq:betaweakconditionRn}) holds for $g_t$ at $y$ with the lower bound $\kappa_0$, for all $y\in B(x, 2^\beta/(2^\beta - 1)\theta^\beta)$. Then
\begin{equation} 
R^g(x,\theta) \geq \kappa_0 - \sum_{i=1}^{\infty}\frac{c(n)2^{i-1}}{\theta}\exp\left(-\frac{2}{D(\theta/2^i)^{1-2\beta}}\right),
\end{equation}
 where $D = D(n)$, and $R^g(\cdot, \theta)$ denotes the scalar curvature of $g_t$ at $t = \theta$.
\end{lemma}
\begin{proof}
We show the contrapositive. We first prove a claim.

\begin{claim*} Under the hypotheses of Lemma \ref{lemma:earlierscalarcurvature}, if $R^g(x,\theta) < \kappa_0$, then there exists a sequence of points $x^k \in \R^n$ such that for all $k$,:
\begin{enumerate}
\item $x^k \in B_{\delta}(x^{k-1}, (\theta/2^k)^{\beta})$, and
\item $R(x^k, \theta/ 2^k) < \kappa_0 + \sum_{i=1}^{k} \frac{c(n)}{\theta/2^{i-1}}\exp\left(-\frac{(\theta/2^i)^{2\beta}}{D(\theta/2^{i-1})}\right)$
\end{enumerate}
for some constants $c$ and $D'$ that depend only on $n$.
\end{claim*}
\begin{proof}[Proof of Claim]
We apply Lemma \ref{lemma:RDTFscalarHKests} to (\ref{eq:RDTFscalarev}). Set $x^0 = x$, and suppose that 
\begin{equation*}
R^g(\cdot, \theta/2) \geq \kappa_0 + \frac{c(n)}{\theta}\exp\left(-\frac{(\theta/2)^{2\beta}}{D\theta}\right)
\end{equation*} 
on $B(x^0, (\theta/2)^{\beta})$, where $c = c(n)$ and $D = D(n)$ will be determined in the course of the proof.

First let $C = C(n)$ and $D = D(n)$ be the constants from Lemma \ref{lemma:RDTFscalarHKests}. Observe that, by Lemma \ref{lemma:RDTFscalarHKests}, we have, for all $k\in \N$ and $x' \in \R^n$,
\begin{equation*}
\begin{split}
\int_{\R^n\setminus B(x', (\theta/2^k)^\beta)}& \Phi^{RD}(x', \theta/2^{k-1}; y,\theta/2^k)dy 
%%%
\\& \leq \exp\left(-\frac{(\theta/2^k)^{2\beta}}{2D(\theta/2^k)}\right)\int_{\R^n\setminus B(x', (\theta/2^k)^\beta)} \frac{C(n)}{(\theta/2^k)^{n/2}}\exp\left(-\frac{|x'-y|^2}{2D(\theta/2^k)}\right)dy
%%%
\\& \leq C'(n)\exp\left(-\frac{(\theta/2^k)^{2\beta}}{D(\theta/2^{k-1})}\right)
\end{split}
\end{equation*}
where $C'$ is some constant that depends only on $n$.

Then, by (\ref{eq:RDTFscalargenerallowerbound}) we have
\begin{align*}
\kappa_0 &> R^g(x^0,\theta) \geq \int_{B(x^0, (\theta/2)^{\beta})}\Phi^{RD}(x^0, \theta; y, \theta/2)R^{g}(y, \theta/2)dy 
\\& + \int_{\R^n\setminus B(x^0, (\theta/2)^{\beta})}\Phi^{RD}(x^0, \theta; y, \theta/2)R^{g}(y, \theta/2)dy
%%%
\\& \geq \left(\kappa_0 + \frac{c(n)}{\theta}\exp\left(-\frac{(\theta/2)^{2\beta}}{D'\theta}\right)\right)\left( 1 - \int_{\R^n\setminus B(x^0, (\theta/2)^{\beta})}\Phi^{RD}(x^0, \theta; y, \theta/2)dy\right) 
\\& - \frac{c'(n)}{\theta}\int_{\R^n\setminus B(x^0, (\theta/2)^{\beta})}\Phi^{RD}(x^0, \theta; y, \theta/2)dy
%%%
\\&\geq \kappa_0 + \frac{c(n)}{\theta}\exp\left(-\frac{(\theta/2)^{2\beta}}{D\theta}\right) - \left(\kappa_0 + \frac{c(n)}{\theta}\exp\left(-\frac{(\theta/2)^{2\beta}}{D\theta}\right)\right)C'(n)\exp\left(-\frac{(\theta/2)^{2\beta}}{D\theta}\right)
\\& - \frac{c'(n)C'(n)}{\theta}\exp\left(-\frac{(\theta/2)^{2\beta}}{D\theta}\right).
\end{align*}
This is a contradiction if
\begin{equation*}
\begin{split}
\frac{c(n)}{\theta}\exp\left(-\frac{(\theta/2)^{2\beta}}{D\theta}\right) &- \left(\kappa_0 + \frac{c(n)}{\theta}\exp\left(-\frac{(\theta/2)^{2\beta}}{D\theta}\right)\right)C'(n)\exp\left(-\frac{(\theta/2)^{2\beta}}{D\theta}\right) 
%%%
\\& - \frac{c'(n)C'(n)}{\theta}\exp\left(-\frac{(\theta/2)^{2\beta}}{D\theta}\right) \geq 0,
\end{split}
\end{equation*}
i.e. if 
\begin{equation*}
\begin{split}
&\frac{c(n)}{\theta}\exp\left(-\frac{(\theta/2)^{2\beta}}{D\theta}\right)
\\&  \geq \left( 1 - C'(n)\exp\left(-\frac{(\theta/2)^{2\beta}}{D(\theta)}\right) \right)^{-1}\left[\frac{c'(n)C'(n)}{\theta}\exp\left(-\frac{(\theta/2)^{2\beta}}{D\theta}\right) + \kappa_0C'(n)\exp\left(-\frac{(\theta/2)^{2\beta}}{D\theta}\right) \right].
\end{split}
\end{equation*}
This can be achieved by choosing $c(n) \geq 10c'(n)C'(n)$ and $\bar \theta$ sufficiently small depending on $\beta$, $n$, and $\kappa_0$ so that for all $\theta < \bar \theta$,
\begin{equation}\label{eq:thetascalarthreshold}
\begin{split}
C'(n)\exp\left(-\frac{(\theta/2)^{2\beta}}{D\theta}\right) &\leq \frac{1}{2},\\
\kappa_0C'(n)\exp\left(-\frac{(\theta/2)^{2\beta}}{D\theta}\right) &\leq \frac{c'(n)C'(n)}{\theta}\exp\left(-\frac{(\theta/2)^{2\beta}}{D\theta}\right).
\end{split}
\end{equation}
%I believe that in the similar argument of \cite{PBG19} this level of detail could be avoided due the global nature of the hypothesis; see \cite[$(7.2)$]{PBG19}.

Thus by contradiction there exists some $x^1 \in B(x^0, (\theta/2)^{\beta})$ such that 
\begin{equation*}
R^g(x^1, \theta/2) < \kappa_0 + \frac{c(n)}{\theta}\exp\left(-\frac{(\theta/2)^{2\beta}}{D\theta}\right).
\end{equation*}

We now iterate this process as follows: Suppose that there exists $x^k \in B(x^{k-1}, (\theta/2^k)^{\beta})$ such that $R(x^k, \theta/ 2^k) < \kappa_0 + \sum_{i=1}^{k} \frac{c(n)}{\theta/2^{i-1}}\exp\left(-\frac{(\theta/2^i)^{2\beta}}{D(\theta/2^{i-1})}\right)$, but that 
\begin{equation*}
R(\cdot, \theta/2^{k+1}) \geq \kappa_0 + \sum_{i=1}^{k+1} \frac{c(n)}{\theta/2^{i-1}}\exp\left(-\frac{(\theta/2^i)^{2\beta}}{D(\theta/2^{i-1})}\right)
\end{equation*} 
on $B(x^k, (\theta/2^{k+1})^{\beta})$. Arguing as above, we then have
\begin{align*}
\kappa_0 &+ \sum_{i=1}^{k} \frac{c(n)}{\theta/2^{i-1}}\exp\left(-\frac{(\theta/2^i)^{2\beta}}{D(\theta/2^{i-1})}\right)
%%%
\\& > R(x^k, \theta/ 2^k) 
\\& \geq \int_{B(x, (\theta/2^{k+1})^{\beta})}\Phi^{RD}(x, \theta/ 2^k; y, \theta/2^{k+1})R^{g}(y, \theta/2^{k+1})dy 
\\& + \int_{\R^n\setminus B(x, (\theta/2^{k+1})^{\beta})}\Phi^{RD}(x, \theta/ 2^k; y, \theta/2^{k+1})R^{g}(y, \theta/2^{k+1})dy
\\& \geq \left(\kappa_0 + \sum_{i=1}^{k+1} \frac{c(n)}{\theta/2^{i-1}}\exp\left(-\frac{(\theta/2^i)^{2\beta}}{D(\theta/2^{i-1})}\right)\right)\left( 1 - \int_{\R^n\setminus B(x, (\theta/2^{k+1})^{\beta})}\Phi^{RD}(x, \theta/ 2^k; y, \theta/2^{k+1})dy  \right) 
\\& - \frac{c'(n)C'(n)}{\theta/2^k}\exp\left(-\frac{(\theta/2^{k+1})^{2\beta}}{D\theta/2^k}\right)
%%%
\\& \geq \kappa_0 + \sum_{i=1}^{k} \frac{c(n)}{\theta/2^{i-1}}\exp\left(-\frac{(\theta/2^i)^{2\beta}}{D(\theta/2^{i-1})}\right) + \frac{c(n)}{2^k}\exp\left(- \frac{(\theta/2^{k+1})^{2\beta}}{D\theta/2^k}\right)
\\& - \left(\kappa_0 + \sum_{i=1}^{k+1} \frac{c(n)}{\theta/2^{i-1}}\exp\left(-\frac{(\theta/2^i)^{2\beta}}{D(\theta/2^{i-1})}\right)\right)C'(n)\exp\left(- \frac{(\theta/2^{k+1})^{2\beta}}{D(\theta/2^{k})}\right) 
\\& - \frac{c'(n)C'(n)}{\theta/2^k}\exp\left(-\frac{(\theta/2^{k+1})^{2\beta}}{D\theta/2^k}\right),
\end{align*}
Decrease $\bar \theta$ further depending on $n$, $\beta$, and $\kappa_0$ so that in addition to (\ref{eq:thetascalarthreshold}) we also have, for all $0 <\theta < \bar \theta$,
\begin{equation*}
\sum_{i=1}^{\infty}\frac{c(n)}{\theta/2^{i-1}}\exp\left(-\frac{(\theta/2^{i+1})^{2\beta}}{D'\theta/2^i}\right) < c'(n).
\end{equation*}
Arguing as in the base case, we produce a contradiction. This proves the claim.
\end{proof}

Now suppose $R(x, \theta) < \kappa_0$, and let $(x^k)_{k=0}^{\infty}$ be as in the claim. We show that there is some $x^{\infty}\in B(\Phi(x), 2^\beta/(2^\beta - 1)\theta^\beta)$ such that $x^k\to x^\infty$ and (\ref{eq:betaweakcondition}) fails at $x^\infty$. Towards the first assertion, observe that for any $n,m$ with $n<m$ we have 
\begin{equation*}
d(x^n, x^m) \leq \sum_{i=n}^{m}\theta^\beta\left(\frac{1}{2^\beta}\right)^i,
\end{equation*}
which can be made arbitrarily small for $n$ and $m$ sufficiently large, since $\sum_{i=1}^{\infty}\left(\frac{1}{2^{\beta}}\right)^k$ converges. In particular $(x^k)$ is Cauchy, and hence converges to some $x^\infty\in M$. Moreover, we have
\begin{equation*}
d(x^k, x^\infty)\leq \lim_{m\to\infty} \sum_{i=k}^{m-1}d(x^i, x^{i+1})  \leq \sum_{i=k}^{\infty} \theta^\beta\left(\frac{1}{2^\beta}\right)^i = \left(\frac{2^{\beta}}{2^{\beta}-1}\right)\left(\frac{\theta}{2^k}\right)^{\beta}
\end{equation*}
and
\begin{equation*}
d(x, x^{\infty})\leq \sum_{i=0}^{\infty}\theta^{\beta}\left(\frac{1}{2^\beta}\right)^i = \left(\frac{2^{\beta}}{2^{\beta}-1}\right)\theta^{\beta}
\end{equation*}
Therefore, $x^\infty \in B(x, 2^\beta/(2^\beta - 1)\theta^\beta)$ and $x^k \in B(x, 2^\beta/(2^\beta - 1)(\theta/2^k)^\beta)$. Moreover, we have
\begin{equation*}
R(x^k, t/2^k) < \kappa_0 + \sum_{i=1}^{\infty}\frac{c(n)2^{i-1}}{\theta}\exp\left(-\frac{2}{D'(\theta/2^i)^{1-2\beta}}\right),
\end{equation*}
so $g$ does not have scalar curvature bounded below by $\kappa_0 + \sum_{i=1}^{\infty}\frac{c(n)2^{i-1}}{\theta}\exp\left(-\frac{2}{D'(\theta/2^i)^{1-2\beta}}\right)$ in the $\beta$-weak sense. Thus we have shown that if (\ref{eq:betaweakcondition}) holds on $B(x, 2^\beta/(2^\beta  -1)\theta^\beta)$ for the lower bound $\kappa_0 + \sum_{i=1}^{\infty}\frac{c(n)2^{i-1}}{\theta}\exp\left(-\frac{2}{D'(\theta/2^i)^{1-2\beta}}\right)$,
then $R(x, \theta) \geq \kappa_0$. The result now follows from replacing $\kappa_0$ by $\kappa_0 - \sum_{i=1}^{\infty}\frac{c(n)2^{i-1}}{\theta}\exp\left(-\frac{2}{D'(\theta/2^i)^{1-2\beta}}\right)$.
\end{proof}

\begin{lemma}\label{lemma:scalarcurvaturegluing}
For all $\eta>0$, $A>2$, $n\in \N$, and $\beta\in (0, 1/2)$ there exists $c = c(n ,A ,\beta, \eta)$ and $\bar r = \bar r(n, \eta, \beta)$ such that for all $r > \bar r$ the following is true:

 Suppose that $(g_t)_{t>0}$ is some smooth family of Riemannian metrics on $\R^n$ evolving by (\ref{eq:RDTFeucl}) and satisfying $|| g_t - \delta||_{C^0(\R^n)} < 1$ and (\ref{eq:RDTFderivests}), and that for some $\beta \in (0, 1/2)$ and $x\in \R^n$ the condition (\ref{eq:betaweakconditionRn}) holds for $g_t$ at $y$ with the lower bound $0$, for all $y \in B(x, 2^\beta/(2^\beta - 1)r^{1-\eta\beta})$. Then, for all $t\in (0, r^{2-\eta}]$, 
 \begin{equation*}
 R(g_t)\big|_x \geq - c(n,A,\beta,\eta)r^{-A}.
 \end{equation*}
\end{lemma}
\begin{proof}
Choose $\bar r$ sufficiently large such that $\bar r^{-\eta} < \bar \theta(0, n,\beta)$, where $\bar \theta$ is as in Lemma \ref{lemma:earlierscalarcurvature}. Let $r> \bar r$ and let $\theta \leq r^{-\eta}$. Consider $\hat g(t) := \frac{1}{r^2}g(r^2t)$, which, by Remark \ref{rmk:parabolicscaling}, is a Ricci-DeTurck flow with respect to the constant background $\hat\delta := \frac{1}{r^2}\delta$. By Remark \ref{rmk:betaweakscalinginvariant}, $\hat g(0)$ has nonnegative scalar curvature in the $\beta$-weak sense on $B_{\hat\delta}(x, 2^\beta/(2^\beta - 1)\theta^{\beta})$, so by Lemma \ref{lemma:earlierscalarcurvature}, we find
\begin{align*}
R^g(x, \theta r^2) &= r^{-2} R^{\hat g}(x, \theta) \geq -r^{-2}\sum_{i=1}^{\infty}\frac{c(n)2^{i-1}}{\theta}\exp\left(-\frac{2}{D(\theta/2^i)^{1-2\beta}}\right).
\end{align*}

Moreover, if $A>1$ then
\begin{align*}
\sum_{i=1}^{\infty}\frac{c(n)2^{i-1}}{\theta}\exp\left(-\frac{2}{D(\theta/2^i)^{1-2\beta}}\right) &= \sum_{i=1}^{\infty}c(n)\left(\frac{2^{i-1}}{\theta}\right)\exp\left(-c(D,\beta)\left(\frac{2^{i-1}}{\theta}\right)^{1-2\beta}\right)
\\& \leq \sum_{i=1}^{\infty}c(n,\beta,A) \left(\frac{2^{i-1}}{\theta}\right) \left(\frac{2^{i-1}}{\theta}\right)^{-A}
\\& = c(n,\beta, A)\theta^{A - 1}\sum_{i=1}^{\infty} \left(\frac{1}{2^{A-1}}\right)^{i-1} \leq c(n,\beta, A)\theta^{A - 1}.
\end{align*}
Then, for all $t\in (0, r^{2-\eta})$, $t = \theta_t r^2$ for some $\theta_t\in (0, r^{-\eta})$ so the previous analysis implies that for all $A>1$,
\begin{align*}
R^g(x,t) &= R^{g}(x, \theta_tr^2)\geq -c(n, \beta, A)r^{-2}\theta_t^{A-1} \geq -c(n, \beta, A)r^{-2}(r^{-\eta})^{A-1}.
\end{align*}
Replacing $A$ by $2 + \eta A - \eta > 2$ yields the result.
\end{proof}

\begin{proposition}\label{prop:classicalmonotonicity}
For all $c_0>0, \tau > (n-2)/2$, $\beta \in (0, 1/2)$, $0< \eta < 2\tau - (n-2)$ there exists $\bar r = \bar r(n, c_0, \tau, \beta, \eta)$ such that for all $r> \bar r$ the following is true:

Let $g$ be a $C^0$ metric on $A(0, .8r, 12r)$ such that $|| g- \delta||_{C^0(A(0, .8r, 12r))} \leq c_0r^{-\tau}$. Suppose $g$ has nonnegative scalar curvature in the $\beta$-weak sense on $A(0, .8r, 12r)$, and suppose that there exists a $C^0$ metric $g_0$ on $\R^n$ that agrees with $g$ on $A(0, .95r, 11.5r)$, for which there is a smooth Ricci-DeTurck flow $(g_t)_{t>0}$ satisfying (\ref{eq:RDTFinitialcondition}), (\ref{eq:RDTFXest}), and (\ref{eq:RDTFderivests}. If $\varphi^1$ and $\varphi^2$ are any two smooth functions with nonzero integral over $(.9, 1.1)$ that do not change sign on $(.9, 1.1)$, then for all $r' \in [\tfrac{1.1}{.9}r, 10r]$ we have
\begin{equation}\label{eq:C0classicalmonotonicity}
M_{C^0}(g_{r^{2-\eta}}, \varphi^1, r') - M_{C^0}(g_{r^{2-\eta}}, \varphi^2, r) \geq -c(n, c_0, \tau, \beta, \eta)r^{n-2 - 2\tau + \eta}.
\end{equation}
\end{proposition}
\begin{proof}
Take $\bar r$ sufficiently large depending on $c_0$ and $\tau$ so that for all $r> \bar r$, $c_0r^{-\tau} < 1$, and increase $\bar r$ as necessary depending on $\beta,\eta,$ and $n$ so that Lemma \ref{lemma:scalarcurvaturegluing} holds. Increase $\bar r$ further depending on $\eta$ and $\beta$ so that for all $r > \bar r$ and all $x\in A(0, .9r, 11r)$, 
\begin{equation*}
B(x, 2^\beta/(2^\beta - 1)r^{1-\eta\beta}) \subset A(0, .95r, 11.5r).
\end{equation*} 
Fix $r> \bar r$. Now let $r' \in [\tfrac{1.1}{.9}r, 10r]$. Using Remark \ref{rmk:C0agreeswithC1average} we find that, by the Lemma \ref{lemma:integralIVT}, there exist values $\ell_1, \ell_2 \in [.9, 1.1]$ such that, using the notation of (\ref{eq:mC2notation}), we have
\begin{align*}
M_{C^0}(g_{r^{2-\eta}}, r', \varphi^1) &- M_{C^0}(g_{r^{2-\eta}}, r, \varphi^2) 
%%%
\\& = \frac{\int_{.9}^{1.1}m_{C^2}(g_{r^{2-\eta}}, r'\ell)\varphi^1(\ell)d\ell}{\int_{.9}^{1.1}\varphi^1(\ell)d\ell} - \frac{\int_{.9}^{1.1}m_{C^2}(g_{r^{2-\eta}}, r\ell)\varphi^2(\ell)d\ell}{\int_{.9}^{1.1}\varphi^2(\ell)d\ell}
%%%
\\& = \frac{m_{C^2}(g_{r^{2-\eta}}, r'\ell_1)\int_{.9}^{1.1}\varphi^1(\ell)d\ell}{\int_{.9}^{1.1}\varphi^1(\ell)d\ell} - \frac{m_{C^2}(g_{r^{2-\eta}}, r\ell_2)\int_{.9}^{1.1}\varphi^2(\ell)d\ell}{\int_{.9}^{1.1}\varphi^2(\ell)d\ell}
%%%
\\&= m_{C^2}(g_{r^{2-\eta}}, r'\ell_1) - m_{C^2}(g_{r^{2-\eta}}, r\ell_2).
\end{align*}
Recall that $r' \geq \tfrac{1.1}{.9}r$, so $\ell_1 r' \geq .9 r' \geq 1.1r \geq \ell_2 r$ and $A(0, \ell_2 r, \ell_1r')\subset A(0, .9r, 11r)$. Also, by Remark \ref{rmk:globalbetaweak}, (\ref{eq:betaweakconditionRn}) holds for $g_t$ at all $y\in A(0, .95r, 11.5r)$. Therefore we apply Lemma \ref{lemma:Bartnikscalculation} and Lemma \ref{lemma:scalarcurvaturegluing} with $A = 2 + 2\tau - \eta$ to $g_t$ at $t= r^{2-\eta}$ and use (\ref{eq:RDTFXest}) and (\ref{eq:RDTFderivests})  to find
\begin{align*}
M_{C^0}(g_{r^{2-\eta}}, r', \varphi^1)& - M_{C^0}(g_{r^{2-\eta}}, r, \varphi^2) = m_{C^2}(g_{r^{2-\eta}}, r'\ell_1) - m_{C^2}(g_{r^{2-\eta}}, r\ell_2) 
\\& \geq \int_{A(o, r\ell_2, r'\ell_1)} R(g_{r^{2-\eta}}) - c(n) |g - \delta| |\nabla^2 g | - c(n)|\nabla g|^2 dx
\\& \geq \int_{A(o, r\ell_2, r'\ell_1)} -c(n, \tau, \beta, \eta)r^{-2 - 2\tau + \eta} - \frac{c(n, c_0)r^{-2\tau}}{r^{2-\eta}}dx
\\& \geq -c(n, c_0, \tau, \beta, \eta)(r'\ell_1)^nr^{-2 - 2\tau + \eta}
\\& \geq -c(n, c_0, \tau, \beta, \eta)r^{n-2-2\tau + \eta}.
\end{align*}
\end{proof}

\begin{corollary}\label{cor:C0monotonicity}
For all $c_0>0, \tau > (n-2)/2$, $\beta\in (0, 1/2)$, $0< \eta < 2\tau - (n-2)$, and all smooth cutoff functions $\varphi: \R \to \R^{\geq 0}$ with $\supp(\varphi)\subset\subset (.9, 1.1)$ there exists $\bar r = \bar r(n, c_0, \tau, \supp(\varphi), \eta, \beta)$ such that for all $r> \bar r$ the following is true:

Suppose $g$ is a continuous metric on $A(0, .8r, 12r)\subset \R^n$ such that $|| g - \delta||_{C^0(A(0, .8r, 12r))} \leq c_0r^{-\tau}$. Suppose also that $g$ has nonnegative scalar curvature in the $\beta$-weak sense on $A(0, .8r, 12r)$. Let $\varphi_{r^{-\eta}}(\ell, t)$ be a smooth solution to (\ref{eq:radialsolvesbackwardsheat}) corresponding to $\varphi$, given by Lemma \ref{lemma:cutofffunctiondecay}. Then, for all $r' \in [\tfrac{1.1}{.9}r, 10r]$, we have
\begin{equation}
M_{C^0}(g, \varphi_{(r')^{-\eta}}(\cdot, 0), r') -  M_{C^0}(g, \varphi_{r^{-\eta}}(\cdot, 0), r) \geq -c(n,  c_0, \varphi, \tau, \beta, \eta)r^{n-2 -2 \tau + \eta}. 
\end{equation}

Moreover, if $\tilde \varphi: \R \to \R^{\geq 0}$ is another smooth cutoff function with $\supp(\tilde \varphi)\subset\subset (.9, 1.1)$ then there exists $\bar r = \bar r(n, c_0, \tau, \supp(\varphi), \supp(\tilde\varphi), \eta, \beta)$ such that for all $r>\bar r$ the following is true:

Let $g$ and $\varphi_{r^{-\eta}}$ be as in the previous statement. Suppose $\tilde \varphi_{r^{-\eta}}(\ell, t)$ is a smooth solution of (\ref{eq:radialsolvesbackwardsheat}) corresponding to $\tilde \varphi$, given by Lemma \ref{lemma:cutofffunctiondecay}. For all $r' \in [\tfrac{1.1}{.9}r, 10r]$, we have
\begin{equation}
M_{C^0}(g, \tilde \varphi_{(r')^{-\eta}}(\cdot, 0), r') - M_{C^0}(g, \varphi_{r^{-\eta}}(\cdot, 0), r) \geq -c(n, c_0, \varphi, \tilde \varphi, \tau, \beta, \eta)r^{n-2 -2\tau +\eta}.
\end{equation}
\end{corollary}
\begin{proof}
Take $\bar \varepsilon$ be as in Lemma \ref{lemma:RDTFexistenceandests} and take $\bar r$ sufficiently large depending on $c_0, \tau, n$ so that for all $r> \bar r$, $c_0r^{-\tau} < \bar \varepsilon$. Increase $\bar r$ as needed to guarantee that $\bar r$ is sufficiently large so that both Proposition \ref{prop:classicalmonotonicity} and the last statement of Lemma \ref{lemma:massdistortionestimate} hold, and so that Lemma \ref{lemma:cutofffunctiondecay} holds for $\bar \theta = \bar r^{-\eta}$. Now fix $r> \bar r$ and let $g_0$ be an extension of $g$ to $\R^n$ that agrees with $g$ on $A(0, .85r, 11.5 r)$ as given by Lemma \ref{lemma:RDTFexistenceandests}, and let $(g_t)_{t>0}$ be a Ricci-DeTurck flow for $g_0$, whose existence is guaranteed by Lemma \ref{lemma:RDTFexistenceandests} as well.

Note that $\varphi_{(r')^{-\eta}}(\cdot, \tfrac{r^{2-\eta}}{(r')^{2}})$ does not change sign on $(.9, 1.1)$ by Lemma \ref{lemma:cutofffunctiondecay}. Since $g_0$ agrees with $g$ on $A(0, .85r, 11.5r)$, we have, by Lemma \ref{lemma:massdistortionestimate} and Proposition \ref{prop:classicalmonotonicity},
\begin{align*}
M_{C^0}(g, \varphi_{(r')^{-\eta}}(\cdot, 0), r') - &M_{C^0}(g, \varphi_{r^{-\eta}}(\cdot, 0), r) 
%%%
\\& \geq M_{C^0}(g, \varphi_{(r')^{-\eta}}(\cdot, 0), r') - M_{C^0}(g_{r^{2-\eta}}, \varphi_{(r')^{-\eta}}(\cdot, \tfrac{r^{2-\eta}}{(r')^2}), r')
\\& + M_{C^0}(g_{r^{2-\eta}}, \varphi_{(r')^{-\eta}}(\cdot, \tfrac{r^{2-\eta}}{(r')^2}), r') - M_{C^0}(g_{r^{2-\eta}}, \varphi, r)
\\& + M_{C^0}(g_{r^{2-\eta}}, \varphi, r)- M_{C^0}(g, \varphi_{r^{-\eta}}(\cdot, 0), r)
%%%
\\& \geq -\int_{0}^{r^{2-\eta}}\left| \frac{d}{dt}M_{C^0}(g_t, \varphi_{(r')^{-\eta}}(\cdot, \tfrac{t}{(r')^2}), r') \right|dt
\\& + M_{C^0}(g_{r^{2-\eta}}, \varphi_{(r')^{-\eta}}(\cdot, \tfrac{r^{2-\eta}}{(r')^2}), r') - M_{C^0}(g_{r^{2-\eta}}, \varphi, r)
\\& - \int_0^{r^{2-\eta}}\left| \frac{d}{dt} M_{C^0}(g_t, \varphi(\cdot, \tfrac{t}{r^2}), r) \right| dt
%%%
\\&\geq - c(n, c_0, \varphi, \tau, \beta, \eta)r^{n-2 - 2\tau + \eta},
\end{align*}
where in the last step we have applied the last statement of Lemma \ref{lemma:massdistortionestimate} to the first and last terms and made use of the fact that $\varphi_{r^{-\eta}}(\cdot, \tfrac{r^{2-\eta}}{r^2}) = \varphi(\cdot)$, and we have applied Proposition \ref{prop:classicalmonotonicity} with $\varphi^1 = \varphi_{(r')^{-\eta}}(\cdot, \tfrac{r^{2-\eta}}{r'}^2)$ and $\varphi^2= \varphi$ to the second term. This proves the first statement. 

To prove the second statement, increase $\bar r$ so that Lemma \ref{lemma:massdistortionestimate} and Lemma \ref{lemma:cutofffunctiondecay} also hold for $\tilde\varphi$ with $\bar \theta = \bar r^{-\eta}$. Fix $r> \bar r$, and argue in the same way to find that
\begin{align*}
M_{C^0}(g, \tilde \varphi_{(r')^{-\eta}}(\cdot, 0), r') &- M_{C^0}(g, \varphi_{r^{-\eta}}(\cdot, 0), r)
%%%
\\& \geq M_{C^0}(g, \tilde \varphi_{(r')^{-\eta}}(\cdot, 0), r') - M_{C^0}(g_{r^{2-\eta}}, \tilde \varphi_{(r')^{-\eta}}(\cdot, \tfrac{r^{2-\eta}}{(r')^2}), r')
\\& + M_{C^0}(g_{r^{2-\eta}}, \tilde \varphi_{(r')^{-\eta}}(\cdot, \tfrac{r^{2-\eta}}{(r')^2}), r') - M_{C^0}(g_{r^{2-\eta}}, \varphi, r)
\\& + M_{C^0}(g_{r^{2-\eta}}, \varphi, r) - M_{C^0}(g, \varphi_{r^{-\eta}}(\cdot, 0) ,r)
%%%
\\& \geq - c(n, c_0, \varphi, \tilde\varphi, \tau, \beta, \eta)r^{n-2 - 2\tau + \eta}.
\end{align*}
where now we have applied Proposition \ref{prop:classicalmonotonicity} with $\varphi^1 = \tilde \varphi_{(r')^{-\eta}}(\cdot, \tfrac{r^{2-\eta}}{(r')^2})$ and $\varphi^2 = \varphi$.
\end{proof}

\begin{remark}\label{rmk:etamonotonicity}
A similar argument also implies the following statement: For all $c_0>0, \tau > (n-2)/2$, $\beta\in (0, 1/2)$, $\eta_1, \eta_2\in (0, 2\tau - (n-2))$, and all smooth cutoff functions $\varphi: \R \to \R^{\geq 0}$ with $\supp(\varphi)\subset\subset (.9, 1.1)$ there exists $\bar r = \bar r(n, c_0, \tau, \supp(\varphi), \eta_1, \eta_2, \beta)$ such that for all $r> \bar r$ the following is true:

Suppose $g$ is a continuous metric on $A(0, .8r, 12r)\subset \R^n$ such that $|| g - \delta||_{C^0(A(0, .8r, 12r))} \leq c_0r^{-\tau}$. Suppose also that $g$ has nonnegative scalar curvature in the $\beta$-weak sense on $A(0, .8r, 12r)$. Let $\varphi_{r^{-\eta_1}}(\ell, t)$ and $\varphi_{r^{-\eta_2}}(\ell, t)$ be the smooth solutions to (\ref{eq:radialsolvesbackwardsheat}) corresponding to $\varphi$, given by Lemma \ref{lemma:cutofffunctiondecay}. Then, for all $r' \in [\tfrac{1.1}{.9}r, 10r]$, we have
\begin{equation}
M_{C^0}(g, \varphi_{(r')^{-\eta_1}}(\cdot, 0), r') -  M_{C^0}(g, \varphi_{r^{-\eta_2}}(\cdot, 0), r) \geq -c(n,  c_0, \varphi, \tau, \beta, \eta)r^{n-2 -2 \tau + \max\{\eta_1, \eta_2\}}. 
\end{equation}

To see why this is true, let $\eta = \max\{\eta_1, \eta_2\}$ so that $\varphi_{r^{-\eta_1}}(\ell, t)$ and $\varphi_{r^{-\eta_2}}(\ell, t)$ are both defined and do not change sign for $t\in (0, r^{2-\eta})$. Let $\bar r$ be as in Corollary \ref{cor:C0monotonicity} for this value of $\eta$. Then the result is proved by arguing as in the proof of the first statement of Corollary \ref{cor:C0monotonicity} but with Proposition \ref{prop:classicalmonotonicity} applied to $\varphi^1 = \varphi_{(r')^{-\eta_1}}(\cdot, \tfrac{r^{2-\eta}}{(r')^2})$ and $\varphi^2 = \varphi_{r^{-\eta_2}}(\cdot, \tfrac{r^{2-\eta}}{r^2})$.
\end{remark}

\begin{lemma}\label{lemma:affineC0monotonicity}
Suppose $g$ is a continuous metric on $\R^n\setminus B(0, r_0/10)$ such that $g$ has nonnegative scalar curvature in the $\beta$-weak sense and $|| g - \delta||_{C^0(\R^n\setminus B(0, r_0/10))} < \varepsilon \leq \bar \varepsilon (n)$, where $\bar \varepsilon (n)$ is as in Lemma \ref{lemma:RDTFexistenceandests}, and also $| g - \delta|_x \leq c_0|x|^{-\tau}$ for all $|x| > r_0/10$. For all $r>0$ let $L^r: \R^n \to \R^n$ denote the map given by $L^r(x) = x+v$, for some fixed vector $v$ where $|v| \leq br$. Then, for all $b>1$ and all $0 < \eta < 2\tau - n + 2$, there exists $\bar r = \bar r(b, r_0, \beta, \eta)$ such that for all $r> \bar r$, and all $r_1 > (1+b)r/.9$ and $r_2 > (br + 1.1 r_1)/.9$
we have
\begin{equation*}
\begin{split}
M_{C^0}(g, \varphi_{r_2^{-\eta}}(\cdot, 0), r_2) &\geq M_{C^0}((L^r)^*g, \varphi_{r_1^{-\eta}}(\cdot, 0), r_1) 
\\& - c(n, c_0, \varphi, \tau ,\beta, \eta)r^{n- 2 - 2\tau + \eta}.
\end{split}
\end{equation*}
\end{lemma}
\begin{proof}
Let $g_0$ be an extension of $g$ to all of $\R^n$ that agrees with $g$ on $\R^n\setminus B(0, r_0/5)$, such that $||g_0 - \delta ||_{C^0(\R^n)} < \varepsilon$. Let $g_t$ be the Ricci-DeTurck flow starting from $g_0$ as in Lemma \ref{lemma:RDTFexistenceandests}. Let $\bar r > r_0$ and fix $r> \bar r$. We will increase $\bar r$ as needed in the course of the proof. First, a couple of observations:
\begin{enumerate}
\item $B(0, r_0) \subset B(v, .9 r_1)$, since $r> r_0$.
\item $B(v, 1.1 r_1) \subset B(0, .9r_2)$.
\end{enumerate}

 For all $x\in B(0, 1.1 r_2) \setminus B(v, .9 r_1)$, if $y\in B(x, \tfrac{2^\beta}{2^\beta  -1}r^{1-\beta\eta})$ then $(1)$ implies that
 \begin{align*}
 |y| & \geq |x| -  \tfrac{2^\beta}{2^\beta  -1}r^{1-\beta\eta}
 \\& \geq |x-v| - |v| - \tfrac{2^\beta}{2^\beta  -1}r^{1-\beta\eta} 
 \\&\geq (1+b)r - br - \tfrac{2^\beta}{2^\beta  -1}r^{1-\beta\eta}
 \\&> r_0/5,
 \end{align*}
where the last step is true provided that $\bar r$ and hence $r$ is big enough depending on $\eta$, $\beta$, and $r_0$. Therefore $g_0$ has nonnegative scalar curvature in the $\beta$-weak sense on $B(x, \tfrac{2^\beta}{2^\beta  -1}r^{1-\beta\eta})$, since $g_0$ agrees with $g$ on $\R^n\setminus B(0, r_0/5)$. In particular, by Lemma \ref{lemma:scalarcurvaturegluing}, if we take $\bar r \geq r_0$ depending on $n, \tau, \beta, \eta$  and let $r> \bar r$ we have that, for all $x \in B(0, 1.1 r_2) \setminus B(v, .9 r_1)$ and all $t\in (0, r^{2-\eta})$,
\begin{equation}\label{eq:lowerscalarboundwithshift}
R(x,t) \geq -c(n,\tau, \beta, \eta)r^{-2-2\tau + \eta}.
\end{equation}

Now note that if $\tilde g$ is any smooth metric, then we have
\begin{align*}
m_{C^2}(L^*\tilde g, r) &= c(n)\int_{\S(0, r)} (\partial_i(L^*\tilde g)_{ij} - \partial_j(L^*\tilde g)_{ii})\nu^j|_xdS(x)
%%%
\\&= c(n)\int_{\S(0, r)} (\partial_i \tilde g_{ij} - \partial_j\tilde g_{ii})\nu^j|_{x+v}dS(x)
%%%
\\&= c(n)\int_{\S(v, r)} (\partial_i \tilde g_{ij} - \partial_j\tilde g_{ii})\nu^j|_{x}dS(x).
\end{align*}
Moreover, $(2)$ implies that $B(v, \ell r_1)\subset B(0, \ell' r_2)$ for all $\ell, \ell'\in (.9, 1.1)$ so, arguing as in the proof of Lemma \ref{lemma:Bartnikscalculation}, we have
\begin{equation}\label{eq:Bartnikcalculationwithshift}
\begin{split}
& m_{C^2}(\tilde g, \ell' r_2) - m_{C^2}((L^r)^*\tilde g, \ell r_1) 
\\& \geq \int_{B(0, \ell'r_2)\setminus B(v, \ell r_1)} R(\tilde g) - c(n)|\tilde g - \delta||\nabla^2 \tilde g| - c(n)|\nabla \tilde g|^2 dV_\delta.
\end{split}
\end{equation}

Now observe that $(L^r)^*g_t$ is a Ricci-DeTurck flow with respect to $(L^r)^*\delta = \delta$, satisfying (\ref{eq:RDTFinitialcondition}), (\ref{eq:RDTFXest}), and (\ref{eq:RDTFderivests}) for $(L^r)^*g_0$, so arguing as in the proofs of Proposition \ref{prop:classicalmonotonicity} and Corollary \ref{cor:C0monotonicity} and applying (\ref{eq:lowerscalarboundwithshift}) to (\ref{eq:Bartnikcalculationwithshift}) with $\tilde g = g_{r^{2-\eta}}$, we find, for some $\ell_1, \ell_2 \in (.9, 1.1)$,
\begin{align*}
M_{C^0}&(g, \varphi_{r_2^{-\eta}}(\cdot, 0), r_2) - M_{C^0}((L^r)^*g, \varphi_{r_1^{-\eta}}(\cdot, 0) ,r_1)
%%%
\\ &\geq -\int_0^{r_1^{2-\eta}} \left|\frac{d}{dt} M_{C^0}(g_t, \varphi_{r_2^{-\eta}}(\tfrac{t}{r_2^2}),  r_2)\right|  
\\& + \int_{B(0, \ell_2 r_2)\setminus B(v, \ell_1 r_1)} R(g_{(\ell_1 r_1)^{2-\eta}}) 
\\& \qquad \qquad - c(n)|g_{(\ell_1 r_1)^{2-\eta}} - \delta||\nabla^2 g_{(\ell_1 r_1)^{2-\eta}}| - c(n)|\nabla g_{(\ell_1 r_1)^{2-\eta}}|^2 dV_\delta 
\\& - \int_0^{r_1^{2-\eta}}\left|\frac{d}{dt} M_{C^0}((L^r)^*g_t, \varphi_{r_1^{-\eta}}(\tfrac{t}{r_1^2}), r_1)\right|
%%%
\\& \geq -c(n, c_0, \varphi, \tau, \beta, \eta)r^{n-2 - 2\tau +\eta}.
\end{align*}
\end{proof}

\section{Comparing masses in different coordinate charts via gluing}\label{sec:coordinates}

\begin{proposition}\label{prop:gluingwmollifiedmaps}
There exists $\bar \delta = \bar \delta(n)< 1$ and $c= c(n)$ such that for all $\delta < \bar \delta$ and all $r>0$ the following is true:

Let $\phi: D\to C$ be a diffeomorphism between some domains $C,D\subset \R^n$ such that, for some $r>0$, $\R^n\setminus \overline{B(0, r/10)} \subset D$ and, for all $x \in A(0, r, 10r)$, $B(\phi(x), r/4)\subset C$. Suppose that $\phi$ is locally $(1+\delta)$-bilipschitz on $\R^n\setminus \overline{B(0, r/10)}$. Then there exists a Euclidean isometry $L$ and a local diffeomorphism $\tilde \phi$ defined on $D$ such that $L(x_0) = \phi(x_0)$ for some $|x_0| = 9.5r$, and
\begin{equation*}
\tilde \phi(x) = \begin{cases}
				\phi(x) & \text{ for } |x| \leq r\\
				L(x) & \text{ for } |x| \geq 10r,
		        \end{cases}	
\end{equation*}
and
\begin{equation*}
|| \tilde \phi^*\delta - \delta||_{C^0(\R^n\setminus \overline{B(0, r/10)})} \leq c(n)\delta.
\end{equation*}
\end{proposition}
\begin{proof}
Let $c = c(n)$ and $\bar \delta = \bar \delta(n)$ be as in Lemma \ref{lemma:closenesstoisofixedmollscale}. Define $\rho(\ell)$ for $\ell \in [1/10, 10]$ by 
\begin{equation*} 
\rho(\ell) = \begin{cases}
		0 & \text{ for } \ell \leq 1\\
                 \frac{e}{16}\exp\left( - \frac{1}{ 1 - \frac{6-\ell}{5}}\right) & \text{ for } 1< \ell < 6\\
                 \frac{1}{16} & \text{ for } \ell \geq 6,
                 \end{cases}
\end{equation*}
so that $\rho$ is smooth and $\rho'$ is bounded on $[1,10]$.

We first show the case where $r=1$. By construction, $\rho \equiv 1/16$ on $[6,10]$. Fix some $x_0$ with $|x_0| = 9.5$. Then, for all $x\in B(x_0, 2(1+\delta)^2\rho(|x_0|))$, $|x| \geq 19/2 - (1+\delta)^2/2 > 1/10$ so $B(x_0, 2(1+\delta)^2\rho(|x_0|)) \subset D$. Also, $2(1+\delta)\rho(|x_0|) = (1+\delta)/8 < 1/4$ so, by hypothesis, $B(\phi(x_0), 2(1+\delta)\rho(|x_0|)) \subset C$. Using the notation of Appendix \ref{appendix:mollification}, by Lemma \ref{lemma:closenesstoisofixedmollscale} there exists a Euclidean isometry $L$ such that $L(x_0) = \phi(x_0)$ and
\begin{equation*}
16|| \phi_{1/16} - L||_{C^0(B(x_0, 1/16))} + || d\phi_{1/16} - dL||_{C^0(B(x_0, 1/16))} \leq c(n)\delta.
\end{equation*}

Let $x\in A(0, 9, 10)$, and let $(x_i)_{i=1}^{k}$ be a sequence of points in $A(0, 9, 10)$ such that $|x_0 - x_1| < 1/16,$ $|x_k - x| < 1/16$, and for $i = 1, \ldots, k-1$, $|x_i - x_{i+1}| < 1/16$. Observe that, by choosing the $x_i$ far enough apart, $k$ can be bounded above by some constant depending only on $n$. Arguing as above, we also have $B(x_i, 2(1+\delta)^2\rho(|x_0|)) \subset D$ and $B(\phi(x_i), 2(1+\delta)\rho(|x_0|)) \subset C$ so, by Lemma \ref{lemma:closenesstoisofixedmollscale}, there exist isometries $L^i$ such that
\begin{equation*}
16|| \phi_{1/16} - L^i||_{C^0(B(x_i, 1/16))} + || d\phi_{1/16} - dL^i||_{C^0(B(x_i, 1/16))} \leq c(n)\delta.
\end{equation*}
Then we have
\begin{equation}\label{eq:dphiicloseness}
\begin{split}
|d_{x_0}\phi_{1/16} - d_x\phi_{1/16}| &\leq |d_{x_0}\phi_{1/16} - dL| + |dL - d_{x_1}\phi_{1/16}| 
\\& + \left[\sum_{i=1}^{k-1} |d_{x_i}\phi_{1/16} - dL^i| + |dL^i - d_{x_{i+1}}\phi_{1/16}|\right] 
\\& + |d_{x_k}\phi_{1/16} - dL^k| + |dL^k - d_x\phi_{1/16}|
%%%
\\& \leq 2(k+2)c(n)\delta \leq c(n)\delta.
\end{split}
\end{equation}

Since this analysis holds for any such $x$, we also have
\begin{equation}\label{eq:mollisopointwiseest}
\begin{split}
|\phi_{1/16}(x) - L(x)| &\leq |\phi_{1/16}(x_0) - L(x_0)| + \int_0^1 |d_{\gamma(s)}\phi_{1/16}(\dot{\gamma}(s)) - dL(\dot{\gamma}(s))|ds
%%%
\\& \leq  |\phi_{1/16}(x_0) - L(x_0)| + \int_0^1 (|d_{x_0}\phi_{1/16}(\dot{\gamma}(s)) - dL(\dot{\gamma}(s))| 
\\& \qquad \qquad + |d_{\gamma(s)}\phi_{1/16}\dot{\gamma}((s)) - d_{x_0}\phi_{1/16}\dot{\gamma}((s))) |) ds
%%%
\\&\leq c(n)\delta
\end{split}
\end{equation}
where $\gamma$ is a path from $x$ to $x_0$ within $A(0, 9, 10)$, and in the last step we have applied the estimate (\ref{eq:dphiicloseness}) at $\gamma(s)$. 

Let $\chi: \R^n\to \R$ be a smooth cutoff function identically equal to $1$ on $B(0, 9)$, with support contained in $B(0, 10)$, such that
\begin{equation*}
||\nabla \chi||_{C^0(\R^n)} \leq 10.
\end{equation*} 
Now define $\tilde \phi$ on $D$ by
\begin{equation*}
\tilde \phi(x) = \chi(x) \phi_{\rho}(x) + (1- \chi(x))L(x).
\end{equation*}
Adjust $\bar \delta$ and $c(n)$ so that Corollary \ref{cor:annularbilipschitzestimate} holds, and also so that $.5(1-\delta) > 1/10$ for all $\delta < \bar \delta$. By Corollary \ref{cor:annularbilipschitzestimate}, $||\tilde \phi^*\delta - \delta||_{C^0(A(0, 1, 9))} \leq c(n)\delta$, since $\tilde \phi = \phi_{\rho}$ in this region, and this estimate also holds trivially on $\R^n\setminus B(0, 10)$, since $L$ is an isometry. Let $x\in A(0, 9, 10)$. Then we have
\begin{align*}
d_x\tilde \phi &= (\nabla_x\chi)\otimes(\phi_{1/16}(x) - L(x)) + dL + \chi(x)(d_x\phi_{1/16} - dL)
%%%
\\& = dL + E,
\end{align*}
where $|E|\leq c(n)\delta$ by choice of $L$, and (\ref{eq:dphiicloseness}) and (\ref{eq:mollisopointwiseest}). In particular, $||\tilde \phi^*\delta - \delta||_{C^0(A(0, 9, 10))} \leq c(n)\delta$. 

Now observe that, as $\ell \to 1$, $\rho(0)\xrightarrow[]{C^\infty} 0$, so if $|x| \leq 1$ then, by dominated convergence,
\begin{equation*}
\lim_{y\to x}\tilde\phi(y) = \lim_{y\to x}\phi_{\rho}(y) = \lim_{y\to x}\int_{\R^n}\phi(y - \rho(|y|)z)\zeta(z)dz = \phi(x),
\end{equation*}
where $\zeta$ is as in Appendix \ref{appendix:mollification}. Also, for $k\geq 1$,
\begin{align*}
\lim_{y \to x}d^k\tilde \phi(y) &= \lim_{y\to x}d^k\phi_\rho(y)
%%%
\\&  = \lim_{y\to x}\int_{\R^n}\sum_{i=1}^k d^i\phi\big|_{y - \rho(|y|)z}d^{k-i-1}\left(I - \rho'(|y|)\frac{y\otimes z}{|y|}\right)\zeta(z)dz
%%%
\\&= d^k\phi\big|_x.
\end{align*}
Therefore, $\tilde \phi$ agrees with $\phi$ on $D \cap B(0,1)$, and hence $|| \tilde \phi^*\delta - \delta||_{C^0(D\cap B(0,1))} \leq c(n)\delta$ as well. This proves the result for $r=1$. 

We now remove the assumption that $r=1$. Suppose the hypotheses hold for some $r$, and apply the previous analysis to the map $\phi'$ defined by $\phi'(x) = r^{-1}\phi(rx)$ to find some smooth map $\tilde \phi'$, equal to a Euclidean isometry $L'$ for $|x|\geq 10$, such that $L'(x') = \phi'(x')$ for some $|x'| = 9.5$ and
\begin{equation*}
|| (\tilde \phi')^*\delta - \delta||_{C^0(D')} \leq c(n)\delta,
\end{equation*}
where $D' = \{r^{-1}x : x\in D\}$.

Let $\tilde \phi(x) = r\tilde\phi'(r^{-1}x)$. Then we have
\begin{equation*}
|| \tilde \phi^*\delta - \delta||_{C^0(D)} =  || \tilde \phi'^*\delta - \delta||_{C^0(D')} \leq c(n)\delta.
\end{equation*}
Also, if $x_0 = rx'$, we have
\begin{align*}
r^{-1}L'(x_0) &= r^{-1}L'(rx') = L'(x') = \phi'(x')
%%%
\\&= r^{-1}\phi(rx') = r^{-1}\phi(x_0),
\end{align*}
so $L'(x_0) = \phi(x_0)$.
This proves the result.
\end{proof}

\begin{theorem}\label{thm:gluingtransmaptoiso}
Suppose that $M^n$ is a smooth manifold and $g$ is a $C^0$ metric on $M$. Suppose that $E$ is an end of $M$ and that $\Phi_1$ and $\Phi_2$ are two $C^0$-asymptotically flat coordinate charts for $E$, so that for $m=1,2$ there exist $c_m>0$ and $r_m >1$ such that
\begin{equation*}
|(\Phi_m)_* g - \delta| \big|_x \leq c_m |x|^{-\tau_m},
\end{equation*}
for all $|x| \geq r_m$, for some $\tau_m > 0$.

Then there exist $r_0 = r_0(\Phi_1, \Phi_2)$, $r_0'' = r_0''(n, \Phi_1, \Phi_2)> 0$ and $c = c(n, \Phi_1, \Phi_2)$ such that, for all $r> r_0''$, there exists a Euclidean isometry $L^r$ such that $|L^r(0)|\leq 25r$ and there exists a map $\tilde\phi^r: \R^n\setminus \overline{B(0, r_0/10)}\to \R^n$, which is a diffeomorphism onto its image, such that
\begin{equation}
\tilde \phi^r(x) = \begin{cases}
				\Phi_2\circ \Phi_1^{-1}(x) & \text{ for } |x| \leq r\\
				 L^r(x) & \text{ for } |x| \geq 10r
			\end{cases}
\end{equation}
and
\begin{equation}\label{eq:pulledbackmetricclosetoeucl}
|| (\tilde \phi^r)^*(\Phi_2)_*g - \delta||_{C^0(A(0, .9r, 11r))} \leq cr^{-\min\{\tau_1, \tau_2\}}.
\end{equation}
\end{theorem}
\begin{proof}
Let $\phi = \Phi_2 \circ \Phi_1^{-1}$ and $\tau = \min\{\tau_1, \tau_2\}$. We apply Proposition \ref{prop:gluingwmollifiedmaps} to the map $\phi$. We first show that, by Lemma \ref{lemma:transmapdomains}, the hypotheses of Proposition \ref{prop:gluingwmollifiedmaps} hold: Let $r_0$ be as in Lemma \ref{lemma:transmapdomains}. Let $r_0'' \geq r_0' \geq r_0$, where $r_0'$ is as in Lemma \ref{lemma:transmapdomains}. Will we increase $r_0''$ throughout the proof as needed. Lemma \ref{lemma:transmapdomains} \ref{item:transmapdefndoutsideofball} implies that $\phi$ is defined on $\R^n\setminus \overline{B(0, r_0/10)}$, and by Lemma \ref{lemma:transmapdomains} \ref{item:bilipoutsideofball}, for all $r> r_0''$, $\phi$ is locally $(1+cr^{-\tau})$-bilipschitz on $\R^n\setminus \overline{B(0, r/10)}$, where $c$ is as in Lemma \ref{lemma:transmapdomains} \ref{item:bilipoutsideofball}.

Moreover, by Lemma \ref{lemma:transmapdomains} \ref{item:imagecontainsballs}, for all $r> r_0''$, $B(\phi(x), r/4)\subset \phi(\R^n\setminus \overline{B(0, r_0/10)})$. Increase $r_0''$ further, depending on $c$, $\tau$, and $n$ so that for all $r> r_0''$ we have $cr^{-\tau} < \bar \delta(n)$, where $\bar \delta(n)$ is as in Proposition \ref{prop:gluingwmollifiedmaps}. Since $\phi$ is locally $(1+cr^{-\tau})$-bilipschitz on $\R^n\setminus \overline{B(0, r/10)}$, we may apply Proposition \ref{prop:gluingwmollifiedmaps} to see that there exists a Euclidean isometry $L^r$ and a local diffeomorphism $\tilde \phi^r$ defined on $\R^n\setminus \overline{B(0, r_0/10)}$ such that $L^r(x_0) = \phi(x_0)$ for some $|x_0| = 9.5r$ and
\begin{equation*}
\tilde \phi^r(x) = \begin{cases}
				\phi(x) & \text{ for } |x| \leq r\\
				L^r(x) & \text{ for } |x| \geq 10r
			\end{cases}
\end{equation*} 
and
 \begin{equation}\label{eq:rgluedmapbilipschitz}
|| (\tilde \phi^r)^*\delta - \delta||_{C^0(\R^n\setminus \overline{B(0, r/10)})} \leq cr^{-\tau}.
\end{equation}
The latter condition implies that $\tilde \phi^r$ is a local diffeomorphism on all of $\R^n\setminus \overline{B(0, r_0/10)}$, since it also agrees with $\phi$ on $\overline{B(0, r)\setminus B(0,  r_0/10)}$. Therefore, Lemma \ref{lemma:generaldiffeogluing} implies that $\tilde \phi^r$ is a diffeomorphism onto its image.

To see why (\ref{eq:pulledbackmetricclosetoeucl}) is true, observe that
\begin{align*}
|| (\tilde \phi^r)^*(\Phi_2)_*g &- \delta||_{C^0(A(0, .9r, 11r))} 
%%%
\\& \leq || (\tilde \phi^r)^*(\Phi_2)_*g - (\tilde \phi^r)^*\delta||_{C^0(A(0, .9r, 11r))} + || (\tilde \phi^r)^*\delta - \delta||_{C^0(A(0, .9r, 11r))}
%%%
\\& \leq ||d\tilde \phi^r||_{C^0(A(0,.9r, 11r))} || (\Phi_2)_* g - \delta||_{C^0(\tilde \phi^r(A(0,.9r, 11r)))} +  || (\tilde \phi^r)^*\delta - \delta||_{C^0(A(0, .9r, 11r))}
%%%
\\& \leq cr^{-\tau},
\end{align*}
where the last step is due to (\ref{eq:rgluedmapbilipschitz}), and bounding $|| (\Phi_2)_* g - \delta||_{C^0(\tilde \phi^r(A(0,.9r, 11r)))}$ by arguing as in the last step of Lemma \ref{lemma:transmapdomains} \ref{item:bilipoutsideofball}.

It remains to bound $|L^r(0)|$. Write $L^r(x) = O^rx + \phi(x_0) - O^rx_0$, for some orthogonal matrix $O^r$, where we are using the fact that $L^r(x_0) = \phi(x_0)$. Then we have, by Lemma \ref{lemma:transmapdomains},
\begin{align*}
|L^r(0)|& \leq |\phi(x_0)| + |O^rx_0| 
%%%
\\& \leq \inf_{x\in \partial B(0, r_0)}\{|\phi(x_0) - \phi(x)| + |\phi(x)|\} + |x_0|
%%%
\\& \leq || d\phi ||_{C^0(\R^n\setminus B(0, r_0))}\dist(x_0, \partial B(0, r_0)) + ||\phi||_{C^0(\partial B(0, r_0))} + |x_0|
%%%
\\& \leq (1 + \tfrac{1}{2})(|x_0| - r_0) + ||\phi||_{C^0(\partial B(0, r_0))} + |x_0|
%%%
\\& \leq 23.75 r + || \phi||_{C^0(B(0, r_0))}
%%%
\\& \leq 25 r,
\end{align*}
provided $r_0''$ and hence $r$ is increased once more depending on $\phi$ so that $ || \phi||_{C^0(B(0, r_0))} \leq 1.25 r_0'' \leq 1.25 r$. This completes the proof.
\end{proof}

We are now ready to prove the main result of the section:
\begin{corollary}\label{cor:coordinatemonotonicity}
Suppose that $M^n$ is a smooth manifold and that $g$ is a $C^0$ metric on $M$. Suppose that $E$ is an end of $M$ and that $\Phi_1$ and $\Phi_2$ are two $C^0$-asymptotically flat coordinate charts for $E$ with decay rates $\tau_1, \tau_2>(n-2)/2$, decay constants $c_1, c_2>0$ and decay thresholds $r_1, r_2> 1$, respectively. Let $\tau = \min\{\tau_1, \tau_2\}$, $0 < \eta < 2\tau - (n-2)$, $\varphi: \R \to \R^{\geq 0}$ be a smooth cutoff function with $\supp(\varphi)\subset \subset (.9, 1.1)$, and for $r>0$ let $\varphi_{r^{-\eta}}(\ell, t)$ be the smooth solution corresponding to $\varphi$ given by Lemma \ref{lemma:cutofffunctiondecay}. Suppose that for some $\beta \in (0, 1/2)$, $g$ has nonnegative scalar curvature in the $\beta$-weak sense on $\Phi_1^{-1}(\R^n\setminus \overline{B(0,1)}) \cup \Phi_2^{-1}(\R^n\setminus \overline{B(0,1)})$. Then there exist $\bar r(\Phi_1, \Phi_2, \tau, n, \supp(\varphi), \eta, \beta)$ and $c = c(\Phi_1, \Phi_2, n, \varphi, \eta, \tau, \beta)$ such that, for all $r\geq \bar r$,
\begin{equation*}
 M_{C^0}(g, \Phi_2, \varphi_{(150r)^{-\eta}}(\cdot, 0), 150r) - M_{C^0}(g, \Phi_1, \varphi_{r^{-\eta}}(\cdot, 0), r) \geq -cr^{n-2-2\tau + \eta}.
 \end{equation*}
\end{corollary}
\begin{proof}
 Let $c$ and $\bar r$ be as in Corollary \ref{cor:C0monotonicity} where we take $c_0$ to be the constant ``$c$'' from Theorem \ref{thm:gluingtransmaptoiso}, and increase $\bar r$ so that $\bar r \geq r_0''$, where $r_0''$ is as in Theorem \ref{thm:gluingtransmaptoiso}. Increase $\bar r$ further so that $\bar r$ is at least as large as the threshold $\bar r$ given by Lemma \ref{lemma:affineC0monotonicity} with $b=25$. For $r > \bar r$ let $\tilde \phi^r$ be as in Theorem \ref{thm:gluingtransmaptoiso}. Let $c$ be as in Corollary \ref{cor:C0monotonicity} and increase $c$ so that $c$ is at least as large as the constant $c$ from Theorem \ref{thm:gluingtransmaptoiso} and the constant $c$ from Lemma \ref{lemma:affineC0monotonicity}. We apply Corollary \ref{cor:C0monotonicity} to the metric $(\tilde\phi^r)*(\Phi_2)_*g$, which is defined on $\R^n\setminus \overline{B(0, r_0/ 10)}$ for $r_0$ as in Theorem \ref{thm:gluingtransmaptoiso}, to find that for $r \geq \bar r$ and all $r' \in [\tfrac{1.1}{.9}r, 10r]$ we have
\begin{equation}\label{eq:masscomparisonwithgluedmap}
M_{C^0}((\tilde\phi^r)^*(\Phi_2)_*g, \varphi_{(r')^{-\eta}}(\cdot, 0), r') - M_{C^0}((\tilde\phi^r)^*(\Phi_2)_*g, \varphi_{r^{-\eta}}(\cdot, 0), r) \geq -cr^{n-2 -2\tau + \eta}.
\end{equation}
Replace $\bar r$ by $1.1\bar r$ so that (\ref{eq:masscomparisonwithgluedmap}) also holds for $r/1.1$, for any $r\geq \bar r$.

 By definition of $\tilde \phi^r$, we have 
 \begin{equation*}
 (\tilde\phi^r)^*(\Phi_2)_*g |_{x} = \begin{cases}
 						(\Phi_1)_*g|_x & \text{ for }|x|\leq r\\
						 (L^r)^*(\Phi_2)_*g|_x & \text{ for } |x|\geq 10r,
 					      \end{cases}
 \end{equation*}
 where $L^r$ is some Euclidean isometry such that $|L(0)| \leq 25r$. Therefore, we have
\begin{equation*}
 \begin{split}
 M_{C^0}((\tilde\phi^r)^*(\Phi_2)_*g, \varphi_{(10r/.9)^{-\eta}}(\cdot, 0), 10r/.9) &=  M_{C^0}((L^r)^*(\Phi_2)_*g, \varphi_{(10r/.9)^{-\eta}}(\cdot, 0) ,10r/.9),\\
 M_{C^0}((\tilde\phi^r)^*(\Phi_2)_*g, \varphi_{(r/1.1)^{-\eta}}(\cdot, 0),r/1.1) &=  M_{C^0}((\Phi_1)_*g, \varphi_{(r/1.1)^{-\eta}}(\cdot, 0), r/1.1).
 \end{split}
 \end{equation*}
Therefore, (\ref{eq:masscomparisonwithgluedmap}) implies 
 \begin{align*}
 M_{C^0}(&(L^r)^*(\Phi_2)_*g, \varphi_{(10r/.9)^{-\eta}}(\cdot, 0), 10r/.9) - M_{C^0}((\Phi_1)_*g, \varphi_{(r/1.1)^{-\eta}}(\cdot, 0),(r/1.1)) 
 %%%
 \\& =  M_{C^0}((\tilde\phi^r)^*(\Phi_2)_*g, \varphi_{(10r/.9)^{-\eta}}(\cdot, 0), 10r/.9) -  M_{C^0}((\tilde\phi^r)^*(\Phi_2)_*g, \varphi_{(r/1.1)^{-\eta}}(\cdot, 0),r/1.1)
 %%%
 \\& \geq M_{C^0}((\tilde\phi^r)^*(\Phi_2)_*g, \varphi_{(10r/.9)^{-\eta}}(\cdot, 0), 10r/.9) -M_{C^0}((\tilde\phi^r)^*(\Phi_2)_*g, \varphi_{(r/.9)^{-\eta}}(\cdot, 0),r/.9)
 \\& + M_{C^0}((\tilde\phi^r)^*(\Phi_2)_*g, \varphi_{(r/.9)^{-\eta}}(\cdot, 0),r/.9) - M_{C^0}((\tilde\phi^r)^*(\Phi_2)_*g, \varphi_{(r/1.1)^{-\eta}}(\cdot, 0),r/1.1)
 %%%
 \\& \geq -cr^{n-2 - 2\tau + \eta}.
 \end{align*}
 Moreover, by Corollary \ref{cor:C0monotonicity} we have 
 \begin{align*}
 M_{C^0}((L^r)^*(\Phi_2)_*g,  \varphi_{(50r)^{-\eta}}(\cdot, 0), 50r) &-  M_{C^0}((L^r)^*(\Phi_2)_*g,  \varphi_{(10r/.9)^{-\eta}}(\cdot, 0), 10r/.9)
 %%%
 \\& \geq -cr^{n-2 - 2\tau + \eta}.
 \end{align*}
 
Then, applying Lemma \ref{lemma:affineC0monotonicity} with $r_1 = 50r$ and $r_2 = 150r/1.1$, we find
 \begin{equation*}
 \begin{split}
 M_{C^0}((\Phi_2)_*g, \varphi_{(150r/1.1)^{-\eta}}(\cdot, 0),& 150r/1.1)
 %%%
 \\& \geq M_{C^0}((L^r)^*(\Phi_2)_*g, \varphi_{(50r)^{-\eta}}(\cdot, 0), 50r)  - cr^{n-2 - 2\tau +\eta}
 \\& \geq M_{C^0}((L^r)^*(\Phi_2)_*g, \varphi_{(10r/.9)^{-\eta}}(\cdot, 0), 10r/.9)  - cr^{n-2 - 2\tau +\eta}
 \\& \geq M_{C^0}((\Phi_1)_*g, \varphi_{(r/1.1)^{-\eta}}(\cdot, 0) ,r/1.1)  - cr^{n-2 - 2\tau + \eta},
 \end{split}
 \end{equation*}
 where in the first step we have also used Lemma \ref{lemma:rotationinvarianceofmass}. Then replacing $r/1.1$ by $r$ yields the result.
\end{proof}

\section{The $C^0$ mass at infinity}\label{sec:masslimit}

\subsection{Taking limits at infinity}
In this section we prove Theorem \ref{thm:fullC0monotonicity}. We first record the following:
\begin{lemma}\label{lemma:prelimC0existence}
Let $M$ be a smooth manifold and $g$ a continuous Riemanian metric on $M$. Suppose that $E$ is an end of $M$
for which there is $C^0$-asymptotically flat coordinate chart $\Phi$ with decay rate $\tau > (n-2)/2$, decay threshold $r_0$, and decay coefficient $c_0$, and suppose that $g$ has nonnegative scalar curvature in the sense of Ricci flow on $\Phi^{-1}(\R^n\setminus \overline{B(0,1)})$. Let $\varphi: \R \to \R^{\geq 0}$ be a smooth cutoff function with $\supp(\varphi) \subset \subset (.9, 1.1)$. Fix $\eta \in (0, 2\tau - (n-2))$. For all $r>0$ let $\varphi_{r^{-\eta}}(\ell, t)$ denote the smooth time-dependent function corresponding to $\varphi$ given by Lemma \ref{lemma:cutofffunctiondecay}. Then the limit
\begin{equation*}
\lim_{r\to \infty} M_{C^0}(g, \Phi, \varphi_{r^{-\eta}}(\cdot, 0), r)
\end{equation*}
exists, and is either finite or $+\infty$.
\end{lemma}
\begin{proof}
Let $\beta\in (0, 1/2)$ be the parameter for which (\ref{eq:betaweakcondition}) holds everywhere for $g$. Let $\bar r = \bar r(n, c_0, \tau, \varphi, \eta, \beta)$ be as in Corollary \ref{cor:C0monotonicity}. Let $r> \max\{r_0/.8, \bar r\}$. Applying Corollary \ref{cor:C0monotonicity} to $\Phi_* g\big|_{A(0, .8r, 12r)}$ we find that for all $r'\in [\tfrac{1.1}{.9}r, 10r]$ we have
\begin{equation}\label{eq:masslimitexistsstep}
M_{C^0}(g, \Phi, \varphi_{(r')^{-\eta}}(\cdot, 0), r') - M_{C^0}(g, \Phi, \varphi_{r^{-\eta}}(\cdot, 0), r) \geq -c(n, \varphi, c_0, \tau, \beta, \eta)r^{n-2 -2\tau + \eta}
\end{equation}
As a shorthand, write $a(r):= M_{C^0}(g, \Phi, \varphi_{r^{-\eta}}(\cdot, 0), r)$, so that (\ref{eq:masslimitexistsstep}) means that for any sufficiently large $r>0$ and any $r'\in [\tfrac{1.1}{.9}r, 10r]$, $a(r') - a(r) \geq - c(n, \varphi, c_0, \tau, \beta, \eta)r^{-\delta + \eta}$ where $\delta = n-2 - 2\tau$ so $-\delta + \eta < 0$. In particular, Lemma \ref{lemma:monotonicityconverges} implies that for any $r>0$, $\lim_{k\to\infty} a(10^kr)$ exists, and is either finite or equal to $+\infty$. It remains to show that
\begin{equation*}
\lim_{r\to\infty}a(r) = \lim_{k\to\infty}a(10^k)=: a_{\infty},
\end{equation*}
say. Towards this objective, let $\varepsilon>0$. Choose $k_0$ sufficiently large so that for all $k\geq k_0$ we have $|a(10^k) - a_{\infty}| < \varepsilon/4$ and $c(10^k)^{-\delta + \varepsilon} < \varepsilon/ 4$, where $c$ is the constant from (\ref{eq:masslimitexistsstep}). 

Let $r> 10^{k_0 + 1}$, and choose $k \geq k_0 + 1$ such that $r\in [10^k, 10^{k+1}]$. First suppose that $r \in [\tfrac{1.1}{.9}10^k, \tfrac{.9}{1.1}10^{k+1}]$. Then $r\in [\tfrac{1.1}{.9}10^k, 10^{k+1}]$ and $10^{k+1}\in [\tfrac{1.1}{.9}r, 10r]$ so Corollary \ref{cor:C0monotonicity} implies
\begin{equation*}
a_{\infty} -\varepsilon/2 < a(10^k) - c(10^k)^{-\delta + \eta} \leq a(r) \leq a(10^{k+1}) + cr^{-\delta + \eta} < a_{\infty} + \varepsilon/2.
\end{equation*}
Now suppose $r \in [10^k, \tfrac{1.1}{.9}10^k)$ (resp. $r\in (\tfrac{.9}{1.1}10^{k+1}, 10^{k+1}]$). Choose $r' \in (\tfrac{1.1}{.9}r, (\tfrac{.9}{1.1})^2(10r))$ (resp. $r' \in ((\tfrac{1.1}{.9})^2(\tfrac{r}{10}), \tfrac{.9}{1.1}r)$). Then $r' \in [\tfrac{1.1}{.9}10^k, \tfrac{.9}{1.1}10^{k+1}]$, so, arguing as above, we have that $|a_{\infty} - a(r')| < \varepsilon/2$ and $|a_{\infty} - a(r'/10)| < \varepsilon/2$. Similarly, Corollary \ref{cor:C0monotonicity} implies
\begin{equation*}
a_{\infty} - \varepsilon \leq a(r'/10) - c(r'/10)^{-\delta + \eta} \leq a(r) \leq a(r') + c(r)^{-\delta + \eta} \leq a_{\infty} + \varepsilon.
\end{equation*}
This completes the proof.
\end{proof}

\begin{proof}[Proof of Remark \ref{rmk:etaindependence}]
To see why Remark \ref{rmk:etaindependence} is true, apply Lemma \ref{lemma:prelimC0existence} to Remark \ref{rmk:etamonotonicity} to find that
\begin{equation*}
\lim_{r\to \infty} M_{C^0}(g, \varphi_{r^{-\eta_1}}(\cdot, 0), r) \geq \lim_{r\to \infty} M_{C^0}(g, \varphi_{r^{-\eta_2}}(\cdot, 0), r).
\end{equation*}
Exchanging $\eta_1$ and $\eta_2$ yields the result.
\end{proof}

\begin{remark}\label{rmk:alternativelimit}
Assume we are in the setting of Lemma \ref{lemma:prelimC0existence}. Let $g_0$ be an extension of $g|_{A(0,.8r, 1.2r)}$ for some large $r$ and $g_t$ be a Ricci-DeTurck flow for $g_0$ whose existence is given by Lemma \ref{lemma:RDTFexistenceandests}. Observe that by Lemma \ref{lemma:massdistortionestimate} we have, 
\begin{equation*}
\lim_{r\to \infty} M_{C^0}(g, \varphi_{r^{-\eta}}(\cdot, 0), r) = \lim_{r\to \infty} M_{C^0}(g_{r^{2-\eta}}, \varphi, r).
\end{equation*}
\end{remark}

We are now ready to prove Theorem \ref{thm:fullC0monotonicity}.
\begin{proof}[Proof of Theorem \ref{thm:fullC0monotonicity}]
Let $\bar r = \bar r(n, c_0, \tau, \supp(\varphi^1), \supp(\varphi^2),  \eta, \beta)$ be as in the second statement of Corollary \ref{cor:C0monotonicity}, with $c_0 = \max\{c_1, c_2\}$ and $\tau = \min\{\tau_1, \tau_2\}$. Increase $\bar r$ as needed so that $\bar r$ is greater than the threshold $\bar r$ given by Corollary \ref{cor:coordinatemonotonicity}.  Also let $c = c(n, c_0, \varphi^1, \varphi^2, \tau, \beta, \eta)$ be as in Corollary \ref{cor:C0monotonicity}, and increase $c$ as needed so that $c$ is at least as large as the constants ``$c$'' given by Corollary \ref{cor:coordinatemonotonicity} for $\varphi^1$ and $\varphi^2$. Applying Corollary \ref{cor:C0monotonicity} to $(\Phi_2)_*g\big|_{A(0, .8(150)r, 12(150)r)}$ implies that
\begin{equation*}
M_{C^0}(g, \Phi_2, \varphi_{(200r)^{-\eta}}^2(\cdot, 0), 200r) \geq M_{C^0}(g, \Phi_2, \varphi_{(150 r)^{-\eta}}^1(\cdot, 0), 150r) - cr^{n-2 - 2\tau + \eta}.
\end{equation*}
Then, for all $r> \bar r$, Corollary \ref{cor:coordinatemonotonicity} implies that
\begin{equation*}
M_{C^0}(g, \Phi_2, \varphi_{(150 r)^{-\eta}}^1(\cdot, 0), 150r) \geq M_{C^0}(g, \Phi_1, \varphi_{r^{-\eta}}^1(\cdot, 0), r) - cr^{n-2 - 2\tau + \eta}
\end{equation*}
so
\begin{equation}\label{eq:fullC0monotonicity}
M_{C^0}(g, \Phi_2, \varphi_{(200r)^{-\eta}}^2(\cdot, 0), 200r) \geq M_{C^0}(g, \Phi_1, \varphi_{r^{-\eta}}^1(\cdot, 0), r) - 2cr^{n-2 - 2\tau + \eta}.
\end{equation}
This proves the first statement, with $c$ adjusted.

Towards the second, note that for $m= 1,2$, the limits
\begin{equation*}
\lim_{r\to \infty} M_{C^0}(g, \Phi_2, \varphi_{r^{-\eta}}^2(\cdot, 0), r), \lim_{r\to \infty} M_{C^0}(g, \Phi_1, \varphi_{r^{-\eta}}^1(\cdot, 0), r)
\end{equation*}
exist by Lemma \ref{lemma:prelimC0existence}, so letting $r\to \infty$ in (\ref{eq:fullC0monotonicity}) implies that
\begin{equation*}
\lim_{r\to \infty} M_{C^0}(g, \Phi_2, \varphi_{r^{-\eta}}^2(\cdot, 0), r) \geq \lim_{r\to \infty} M_{C^0}(g, \Phi_1, \varphi_{r^{-\eta}}^1(\cdot, 0), r).
\end{equation*}
Exchanging $\Phi_2$ with $\Phi_1$ and $\varphi^1$ with $\varphi^2$ proves the second statement of the theorem.
\end{proof}

\subsection{Finiteness conditions for the $C^0$ mass}
In this section we prove Theorem \ref{thm:fullC0existence}. We first establish the finiteness condition for the $C^0$ mass at infinity.

\begin{theorem}\label{thm:finitemass}
Suppose $g$ is a continuous Riemannian metric on $\R^n\setminus \overline{B(0, r_0)}$ such that, for some $\tau > (n-2)/2$, $c_0>0 , \bar r >r_0$ we have
\begin{equation*}
|g - \delta|\big|_x \leq c_0|x|^{-\tau}
\end{equation*}
for all $|x| > \bar r$.
Suppose $g$ has nonnegative scalar curvature in the sense of Ricci flow. Let $\varphi: \R \to \R^{\geq 0}$ be a smooth cutoff function with $\supp(\varphi)\subset \subset (.9, 1.1)$. Fix $\eta \in (0, 2\tau - (n-2))$. For all $r>0$ let $\varphi_{r^{-\eta}}(\ell, t)$ denote the smooth time-dependent function corresponding to $\varphi$ given by Lemma \ref{lemma:cutofffunctiondecay}. Then the limit $\lim_{r\to\infty}M_{C^0}(g, \varphi_{r^{-\eta}}(\cdot, 0), r)$ is finite if and only if the following condition holds:

There exists a sequence of numbers $r_k\to \infty$ such that $r_{k+1} > 1.1/.9 r_k >0$ for all $k$, and for which for all $k$ there exists an extension $g_0^k$ of $g|_{\R^n\setminus \overline{B(0, .7r_k)}}$ to all of $\R^n$ such that there is a Ricci-DeTurck flow $g_t^k$ for $g_0^k$ satisfying (\ref{eq:RDTFinitialcondition}), (\ref{eq:RDTFXest}), (\ref{eq:RDTFderivests}), and (\ref{eq:L1scalarcondition}).
\end{theorem}

\begin{proof}
First note that if $|x|$ is large then 
\begin{equation*}
|x| - \frac{2^\beta}{2^\beta - 1}\left(\frac{|x|}{.9}\right)^{1-\eta \beta} \geq \frac{|x|}{2} > r_0
\end{equation*}
so $B(x, 2^\beta/(2^\beta - 1)(|x|/.9)^{1-\eta\beta} )\subset \R^n\setminus \overline{B(0, r_0)}$, and hence Lemma \ref{lemma:scalarcurvaturegluing} implies that for all $A>2$ there is a constant $c(n, A, \beta, \eta)$ such that for all $t\in (0, (|x|/ .9)^{2-\eta}$,
\begin{equation*}
R(g_t)\big|_x \geq -c(n,A,\beta,\eta) |x|^{-A}
\end{equation*}
or
\begin{equation}\label{eq:negativescalarcurvatureestimate}
R(g_t)_-\big|_x \leq c(n,A,\beta,\eta) |x|^{-A},
\end{equation}
where $f_-$ denotes the negative part of a function $f$.

For the sake of simplicity, let $a(r):= M_{C^0}(g, \varphi_{r^{-\eta}}(\cdot, 0), r)$. We now show that 
\begin{equation}\label{eq:finitelimit}
\lim_{r\to \infty} a(r) < \infty
\end{equation}
if and only if there exists a sequence $r_k\to \infty$ such that
\begin{equation}\label{eq:suplimitwithscaling}
\lim_{k\to \infty}\sup_{r' > 1.1/.9 r_k} a(1.1/.9r') - a(.9/1.1r_k) = 0.
\end{equation}

To see why these conditions are equivalent, suppose (\ref{eq:suplimitwithscaling}) fails. Let $r_k \to \infty$ be a strictly increasing sequence. By assumption there exists some $\varepsilon_0 > 0$ such that for all $k$ there exists $j>k$ with
\begin{equation}\label{eq:cauchylimitcontradictionassumption}
\sup_{r' > 1.1/.9 r_j} a(1.1/.9 r') - a(.9/1.1 r_j) > \varepsilon_0.
\end{equation}
We construct a new sequence $\tilde r_k\to \infty$ inductively as follows: By assumption there exists $j_1 > 1$ for which
\begin{equation*} 
\sup_{r' > 1.1/.9 r_{j_1}} a(1.1/.9 r') - a(.9/1.1 r_{j_1}) > \varepsilon_0.
\end{equation*}
Let $\tilde r_1 = .9/1.1 r_{j_1}$ and choose $r_1' > 1.1/.9 r_{j_1} = (1.1/.9)^2\tilde r_1$ such that $a(1.1/.9 r_1') - a(.9/1.1 r_{j_1}) > \varepsilon_0/2$. Setting $\tilde r_2 := 1.1/.9 r_1'$, this becomes $a(\tilde r_2) - a(\tilde r_1) > \varepsilon_0/2$. Now, given $\tilde r_{2k}$, choose $m_{2k+1}$ sufficiently large so that $r_{m_{2k+1}}> (1.1/.9)^3 \tilde r_{2k}$. Then there exists $j_{2k+1} > m_{2k+1}$ such that
\begin{equation*} 
\sup_{r' > 1.1/.9 r_{j_{2k+1}}} a(1.1/.9 r') - a(.9/1.1 r_{j_{2k+1}}) > \varepsilon_0.
\end{equation*}
Let $\tilde r_{2k+1} = .9/1.1 r_{j_{2k+1}}> (1.1/.9)^2\tilde r_{2k}$ and set $\tilde r_{2k + 2} := 1.1/.9 r_{2k + 1}'$ where $r_{2k+1}'$ is chosen so that $a(1.1/.9 r_{2k+1}') - a(.9/1.1 r_{j_{2k+1}}) > \varepsilon_0/2$, and hence $a(\tilde r_{2k+2}) - a(\tilde r_{2k+1}) > \varepsilon_0/2$. Then $\tilde r_k\to \infty$ and $(a(\tilde r_k))_{k=1}^{\infty}$ is not Cauchy, so (\ref{eq:finitelimit}) fails. In particular, (\ref{eq:finitelimit}) implies (\ref{eq:suplimitwithscaling}). Conversely, suppose that (\ref{eq:suplimitwithscaling}) holds. Pass to a subsequence $r_{k_j}$ so that for all $j$, $r_{k_{j+1}} > (1.1/.9)^2 r_{k_j}$. Let $\tilde r_j = .9/1.1 r_{k_j}$. Then (\ref{eq:suplimitwithscaling}) implies that
\begin{align*}
\lim_{j\to \infty}\sup_{m> j} |a(\tilde r_m) - a(\tilde r_j)| &\leq \lim_{j\to\infty }\sup_{r' > (1.1/.9)^2\tilde r_j} |a(1.1/.9 r') - a(\tilde r_{j})| 
%%%
\\& = \lim_{j\to \infty}\sup_{r' > 1.1/.9 r_{k_j}} |a(1.1/.9r') - a(.9/1.1 r_{k_j})| 
%%%
\\& =0,
\end{align*}
so $(a(\tilde r_j))_{j=1}^{\infty}$ is Cauchy, and hence it converges to some limit $a_\infty < \infty$. By Lemma \ref{lemma:prelimC0existence} it follows that (\ref{eq:finitelimit}) holds.

We now prove the theorem. First assume that the limit is finite. Let $r_k\to \infty$ be a sequence of numbers such that $r_{k+1} > 1.1/.9 r_k >0$ for all $k$, sufficiently large so that $|| g - \delta||_{C^0(\R^n\setminus B(0, .6 r_k))} < \bar \varepsilon$, where $\bar \varepsilon$ is as in Lemma \ref{lemma:RDTFexistenceandests}. For all $k$ let $g_0^k$ be the extension of $g|_{\R^n\setminus \overline{B(0, .7 r_k)}}$ to all of $\R^n$ given by Lemma \ref{lemma:RDTFexistenceandests}, so that there exists a Ricci-DeTurck flow $g_t^k$ for $g_0^k$ satisfying (\ref{eq:RDTFinitialcondition}), (\ref{eq:RDTFXest}), and (\ref{eq:RDTFderivests}). Now fix $k$, and fix $r' > 1.1/.9 r_k$. As in the proof of Proposition \ref{prop:classicalmonotonicity}, by Lemma \ref{lemma:integralIVT} there exist $\ell, \ell' \in [.9, 1.1]$ such that
\begin{equation*}
\begin{split}
M_{C^0}(g^k_{(.9/1.1 r_k)^{2-\eta}}, \varphi_{(1.1/.9 r')^{-\eta}}(\tfrac{(.9/1.1 r_k)^{2-\eta}}{(1.1/.9 r')^2}), 1.1/.9 r') &= m_{C^2}(g^k_{(.9/1.1 r_k)^{2-\eta}}, 1.1/.9\ell' r')\\
M_{C^0}(g^k_{(.9/1.1 r_k)^{2-\eta}}, \varphi_{(.9/1.1 r_k)^{-\eta}}(\tfrac{(.9/1.1 r_k)^{2-\eta}}{(.9/1.1 r_k)^2}), .9/1.1r_k) &= m_{C^2}(g^k_{(.9/ 1.1 r_k)^{2-\eta}}, .9/1.1\ell r_k).
\end{split}
\end{equation*}

Then, using (\ref{eq:negativescalarcurvatureestimate}), Lemma \ref{lemma:Bartnikscalculation} with (\ref{eq:RDTFderivests}) to bound the quadratic term, Remark \ref{rmk:massdistortiondiffradii}, and Lemma \ref{lemma:massdistortionestimate}  we have
\begin{align*}
& \bigg|\int_{A(0, .9r_k, 1.1r')} R(g^k_{(.9/1.1 r_k)^{2-\eta}})dx\bigg|  \leq \left|\int_{A(0, .9/1.1 \ell r_k, 1.1/.9\ell' r')} R(g^k_{(.9/1.1 r_k)^{2-\eta}})dx\right| + cr_k^{n-2-2\tau + \eta}
%%%
\\& \leq |m_{C^2}(g^k_{(.9/1.1 r_k)^{2-\eta}}, 1.1/.9\ell' r') - m_{C^2}(g^k_{(.9/1.1 r_k)^{2-\eta}}, .9/1.1\ell r_k)| + cr_k^{n-2-2\tau + \eta}
%%%
\\& = \bigg| M_{C^0}(g^k_{(.9/1.1 r_k)^{2-\eta}}, \varphi_{(1.1/.9 r')^{-\eta}}(\tfrac{(.9/1.1 r_k)^{2-\eta}}{(1.1/.9 r')^2}), 1.1/.9 r') 
\\& \qquad \qquad - M_{C^0}(g^k_{(.9/1.1 r_k)^{2-\eta}}, \varphi_{(.9/1.1 r_k)^{-\eta}}(\tfrac{(.9/1.1 r_k)^{2-\eta}}{(.9/1.1 r_k)^2}), .9/1.1r_k) \bigg|+ cr_k^{n-2-2\tau + \eta}
%%%
\\& \leq \left| M_{C^0}(g^k_{0}, \varphi_{(1.1/.9 r')^{-\eta}}(0), 1.1/.9 r') - M_{C^0}(g^k_{0}, \varphi_{(.9/1.1 r_k)^{-\eta}}(0), .9/1.1r_k)\right| + cr_k^{n-2-2\tau + \eta}
%%%
\\&= |a(1.1/.9 r') - a(.9/1.1 r_k)| + cr_k^{n-2-2\tau + \eta}.
\end{align*}
Then we have, by (\ref{eq:suplimitwithscaling}),
\begin{equation*}
\begin{split}
\lim_{k\to \infty} \sup_{r' > 1.1/.9 r_k}&\left| \int_{A(0, .9r_k, 1.1r')} R(g_{(.9/1.1 r_k)^{2-\eta}})dx\right| 
\\& \leq \lim_{k\to \infty}\sup_{r' > 1.1/.9 r_k} |a(1.1/.9 r') - a(.9/1.1 r_k)| = 0
\end{split}
\end{equation*}
so (\ref{eq:L1scalarcondition}) holds.

Conversely, suppose that (\ref{eq:L1scalarcondition}) holds. Let $r_k, g_0^k$, and $g_t^k$ be as given. Fix some $k$ and some $r' > 1.1/.9 r_k$. Arguing as in the previous step, we find that there exist $\ell, \ell' \in [.9, 1.1]$ for which we have
\begin{align*}
& |a(r') - a(r_k)| - cr_k^{n-2 - 2\tau + \eta}
%%%
\\ & \leq \bigg| M_{C^0}(g^k_{(.9/1.1 r_k)^{2-\eta}}, \varphi_{(1.1/.9 r')^{-\eta}}(\tfrac{(.9/1.1 r_k)^{2-\eta}}{(1.1/.9 r')^2}), r') 
\\& \qquad \qquad - M_{C^0}(g_{(.9/1.1 r_k)^{2-\eta}}, \varphi_{(.9/1.1 r_k)^{-\eta}}(\tfrac{(.9/1.1 r_k)^{2-\eta}}{(.9/1.1 r_k)^2}), r_k)\bigg| 
%%%
\\& = \left| m_{C^2}(g^k_{(.9/1.1 r_k)^{2-\eta}}, \ell' r') - m_{C^2}(g^k_{(.9/1.1 r_k)^{2-\eta}}, \ell r_k) \right|
%%%
\\&= \left| \int_{A(0, \ell r_k, \ell' r')} R(g^k_{(.9/1.1 r_k)^{2-\eta}}) - Q^R[g^k_{(.9/1.1 r_k)^{2-\eta}}]dx \right|
%%%
\\& \leq \int_{A(0, .9r_k, 1.1r')} R(g^k_{(.9/1.1 r_k)^{2-\eta}}) dx + cr_k^{n-2-2\tau + \eta},
\end{align*}
where again we have used (\ref{eq:negativescalarcurvatureestimate}) in the last step. Therefore, (\ref{eq:L1scalarcondition}) implies that
\begin{equation*}
\lim_{k \to \infty}\sup_{r' > 1.1/.9 r_k} |a(r') - a(r_k)|  = 0,
\end{equation*}
and hence (\ref{eq:suplimitwithscaling}) holds.
\end{proof}

We are now ready to prove Theorem \ref{thm:fullC0existence}. The proof follows quickly from our previous results.
\begin{proof}[Proof of Theorem \ref{thm:fullC0existence}]
To prove the first statement, observe that if $m_{ADM}(g, \Phi)$ is finite, then by Remark \ref{rmk:C0agreeswithC1average} we have
\begin{align*}
&\left| M_{C^0}(g, \Phi, \varphi_{r^{-\eta}}(0), r) - m_{ADM}(g, \Phi)\right| 
\\& = \left|\frac{\int_{.9 r}^{1.1 r}\tfrac{1}{r}\varphi_{r^{-\eta}}(\tfrac{u}{r}, 0)\left[ \int_{\S(u)} (\partial_j g_{ij} - \partial_i g_{jj})\nu^i dS - m_{ADM}(g, \Phi)\right]du}{\int_{.9}^{1.1}\varphi_{r^{-\eta}}(u, 0)du}\right|
%%%
\\& \leq \max_{u \in [.9, 1.1]} \left| \int_{\S(ur)} (\partial_j g_{ij} - \partial_i g_{jj})\nu^i dS - m_{ADM}(g, \Phi) \right|
%%%
\\& \xrightarrow[r \to \infty]{}0,
\end{align*}
so
\begin{equation*}
\lim_{r\to \infty} M_{C^0}(g, \Phi, \varphi_{r^{-\eta}}(0), r) = m_{ADM}(g, \Phi).
\end{equation*}

If $m_{ADM}(g, \Phi) = \infty$, then for all $N>0$ there exists $r_N$ such that for all $r> r_N$,
\begin{equation*}
 \int_{\S(r)} (\partial_j g_{ij} - \partial_i g_{jj})\nu^i dS > N.
\end{equation*}
Then, for all $r> r_N/.9$, we have
\begin{equation*}
M_{C^0}(g, \Phi, \varphi_{r^{-\eta}}(0), r) \geq \min_{u\in [.9, 1.1]}  \int_{\S(ur)} (\partial_j g_{ij} - \partial_i g_{jj})\nu^i dS > N,
\end{equation*}
so
\begin{equation*}
\lim_{r\to \infty} M_{C^0}(g, \Phi, \varphi_{r^{-\eta}}(0), r) = \infty.
\end{equation*}
The proof the show the result when $m_{ADM}(g, \Phi) = -\infty$ is similar.

We now prove the second statement. By Lemma \ref{lemma:prelimC0existence}, $\lim_{r\to\infty} M_{C^0}(g, \Phi, \varphi_{r^{-\eta}}(\cdot, 0), r)$ exists. By Theorem \ref{thm:fullC0monotonicity}, the limit is independent of choice of $\Phi$ and $\varphi$. By Theorem \ref{thm:finitemass} applied to $\Phi_*g$, the limit is finite if and only if the condition (\ref{eq:L1scalarcondition}) holds.
\end{proof}

\appendix
\section{A bilipschitz map is $C^0$-close to a Euclidean isometry}\label{appendix:almostisometries}

\begin{lemma}\label{lemma:bilipschitztoalmostisometry} 
For all $\delta < 1$, $r>0$ the following is true:

Suppose that $C,D\subset \R^n$ are some domains, and $\phi: D\to C$ is a diffeomorphism. If, for some $x_0 \in D$, $B(x_0, (1+\delta)r)\subset D$, $B(\phi(x_0), r)\subset C$, and $\phi$ is locally $(1+\delta)$-bilipschitz on $B(x_0, (1+\delta )r)$, then $\phi$ is a $4\delta r$-isometry on $B(x_0, r/(1+\delta))$, i.e. if  $|| \phi^*\delta - \delta||_{C^0(B(x_0, (1+\delta)r))} \leq \delta$ then for all $x,y \in B(x_0, r/(1+\delta))$ we have
\begin{equation*}
||\phi(x) - \phi(y)| - |x-y| |\leq 4\delta r.
\end{equation*}
\end{lemma}

\begin{proof}
First assume that $\phi(x_0) = 0$ and $x_0 = 0$. Note that, for all $x\in B(0, r/(1+\delta))$, we have
\begin{equation*}
|\phi(x) - 0| \leq \int_0^1 |d\phi_{sx}||x|ds \leq (1+ \delta)r/(1+\delta) = r.
\end{equation*} 
In particular, for all $x, y \in B(0, r/(1+\delta))$, $\phi(x), \phi(y) \in B(0, r)$, so the shortest path from $\phi(x)$ to $\phi(y)$ is contained in the ball $B(0,r) \subset C$.

Now fix some $x, y \in B(0, r/(1+\delta))$, so that $\phi(x), \phi(y) \in B(0,r)$. By (\ref{eq:localglobalbilipdiffeo}) we have
\begin{equation}\label{eq:bilipschitzpointwisecondition}
(1-2\delta)|x-y|  \leq |\phi(x) - \phi(y)| \leq (1+2\delta)|x-y|.
\end{equation}
 We now address the assumptions that $x_0 = 0$ and $\phi(0) = 0$. For general $x_0$ and $\phi(x_0)$, apply (\ref{eq:bilipschitzpointwisecondition}) to the map $\hat \phi(x):= \phi(x+ x_0) - \phi(x_0)$. Then for all $x,y \in B(x_0, r/(1+\delta))$, $x- x_0 \in B(0, r/(1+\delta))$, $y- x_0 \in B(0, r/(1+\delta))$, and
\begin{equation*}
|\phi(x) - \phi(y)| = |\hat \phi(x- x_0) - \hat \phi(y- x_0)|
\end{equation*}
so by (\ref{eq:bilipschitzpointwisecondition}) we have
\begin{align*}
(1-2\delta)|x-y| & = (1-2\delta)|(x- x_0) - (y- x_0)| \leq |\hat \phi(x - x_0) - \hat \phi(y-x_0)|
%%%
\\& = |\phi(x) - \phi(y)|,
\end{align*}
and similarly,
\begin{equation*}
|\phi(x) - \phi(y)| = |\hat \phi(x- x_0) - \hat \phi(y- x_0)| \leq (1+2\delta)|x-y|.
\end{equation*}
as above. In particular,
\begin{equation*}
||\phi(x) - \phi(y)| - |x-y|| \leq 2\delta|x-y| \leq 4r\frac{\delta}{1+\delta} \leq 4r\delta.
\end{equation*}
\end{proof}

We now record the following result, which is a special case of \cite[Theorem $1$]{Vestfrid05}. See Section \ref{subsec:bilipschitzmaps} for the definition of a $\delta$-isometry.
\begin{lemma}\label{lemma:almostisometryclosetoisowmultloss}
There exists $c= c(n)$ and $\bar \delta = \bar \delta(n)< 1$ such that for all $\delta < \bar \delta$ the following is true:

If $\phi$ is a continuous $\delta$-isometry on some ball $B(x_0, r)$, then there exists a Euclidean isometry $L$ such that $L(x_0) = \phi(x_0)$ and
\begin{equation*}
|| \phi - L||_{C^0(B(x_0, r))} \leq  c\delta.
\end{equation*}

In particular, for all $\delta < \bar \delta$, if $\phi: D\to C$ is a diffeomorphism between some domains $C,D \subset \R^n$ such that $B(x_0, (1+\delta)r)\subset D$,  $B(\phi(x_0), r)\subset C$, and $\phi$ is locally $(1+\delta)$-bilipschitz on $B(x_0, (1+\delta)r)$, then there exists a Euclidean isometry $L$ such that $L(x_0) = \phi(x_0)$ and 
\begin{equation*}
|| \phi - L||_{C^0(B(x_0, r/(1+\delta)))} \leq 4c\delta r.
\end{equation*}
\end{lemma}
\begin{proof}
The second statement follows from the first statement and Lemma \ref{lemma:bilipschitztoalmostisometry}. The first statement is due to \cite[Theorem $1$]{Vestfrid05}.
\end{proof}

\section{Bilipschitz multiplicative loss with varying mollification scale}\label{appendix:mollification}
Here we will fix some conventions involving mollified maps. Let $\zeta: \R^n \to \R $ denote the standard mollifier,
\begin{equation*}
\zeta(z) = c(n)\exp\left(-\frac{1}{1 - |z|^2}\right),
\end{equation*}
where $c(n)$ is a normalization constant chosen so that
\begin{equation*}\label{eq:zetaintegratesto1}
\int_{\R^n}\zeta(z)dz = 1.
\end{equation*} 
Observe that $\zeta \in W^{1,1}(\R^n)$, and that $\zeta \equiv 0$ outside of $B(0, 1)$.

If $D\subset \R^n$ is some domain and $F:D \to \R^n$ is any continuous map, then, for any $\rho_0>0$ we use $F_{\rho_0}$ to denote the map, defined on $\{x\in D: B(x, \rho_0)\subset D\}$, which is given by
\begin{equation*}
F_{\rho_0}(x) = \int_{\R^n}F(x - \rho_0z)\zeta(z)dz = \int_{\R^n}F(x-z)\zeta\left(\frac{z}{\rho_0}\right)\rho_0^{-n}dz = \int_{\R^n}F(z)\zeta\left(\frac{x-z}{\rho_0}\right)\rho_0^{-n}dz.
\end{equation*}

If $\rho: [0,\infty)\to [0, \infty)$ is any continuous function, then we use $F_{\rho}$ to denote the map given by $F_{\rho}(x) = (F_{\rho(|x|)})(x)$, i.e. 
\begin{equation}\label{eq:defmollwvaryingscale}
F_{\rho}(x) = \int_{\R^n}F(x - \rho(|x|)z)\zeta(z)dz,
\end{equation}
so $F_{\rho}$ is defined on $\{x\in D : B(x, \rho(|x|))\subset D\}$. For $x$ such that $\rho(|x|)\neq 0$, we have
\begin{equation}\label{eq:varyingmollscalecov}
F_{\rho}(x) = \int_{\R^n}F(z)\zeta\left(\frac{x-z}{\rho(|x|)}\right)\rho(|x|)^{-n}dz.
\end{equation}

We first show the following result, for a constant mollification scale.
\begin{lemma}\label{lemma:closenesstoisofixedmollscale}
There exists $c = c(n)$ and $\bar \delta = \bar \delta(n)$ such that for all $\delta < \bar \delta$, the following is true for all $\rho_0 >0$:

Suppose $\phi: D\to C$ is a diffeomorphism between some domains $C,D \subset \R^n$ such that $B(x_0, 2(1+\delta)^2\rho_0) \subset D$, $B(\phi(x_0), 2(1+\delta)\rho_0)\subset C$. If $\phi$ is locally $(1+\delta)$-bilipschitz on $B(x_0, 2(1+\delta)^2\rho_0)$, then there exists a Euclidean isometry $L^{x_0,\rho_0}$ such that $L^{x_0,\rho_0}(x_0) = \phi(x_0)$ and 
\begin{equation*}
\rho_0^{-1}|| \phi_{\rho_0} - L^{x_0,\rho_0} ||_{C^0(B(x_0, \rho_0))} + || d\phi_{\rho_0} - dL^{x_0, \rho_0}||_{C^0(B(x_0, \rho_0))} \leq c\delta
\end{equation*}
\end{lemma}
\begin{proof}
Let $c$ and $\bar \delta$ be as in Lemma \ref{lemma:almostisometryclosetoisowmultloss}. We will adjust $c$ throughout the proof. By Lemma \ref{lemma:almostisometryclosetoisowmultloss} there exists a Euclidean isometry $L$ such that $L(x_0) = \phi(x_0)$ and $|| \phi - L||_{C^0(B(x_0,2\rho_0))} \leq c\delta\rho_0$, so we have
\begin{align*}
|| \phi_{\rho_0} - L||_{C^0(B(x_0,\rho_0))} &=  \sup_{x\in B(x_0,\rho_0)}\left| \int_{B(x,\rho_0)}(\phi(z) - L(z))\zeta\left(\frac{x-z}{\rho_0}\right)\rho_0^{-n}dz \right|
%%%
\\& \leq || \phi - L||_{C^0(B(x_0,2\rho_0))} \leq  c\delta\rho_0,
\end{align*}
with $c$ adjusted. Similarly,
\begin{align*}
||d\phi_{\rho_0} - dL||_{C^0(B(x_0,\rho_0))} &= \sup_{x\in B(x_0,\rho_0)}\left| \int_{B(x,\rho_0)}(\phi(z) - L(z))\otimes (\nabla\zeta(|\cdot|))\big|_{\frac{x-z}{\rho_0}}\rho_0^{-n-1}dz \right|
%%%
\\& \leq \frac{c(n)}{\rho_0}|| \phi - L||_{C^0(B(x_0, 2\rho_0))} \leq c\delta,
\end{align*}
with $c$ adjusted. 
\end{proof}

\begin{lemma}\label{lemma:isometriesundermollification}
Let $L$ be any Euclidean isometry and $\rho: [0, \infty) \to [0, \infty)$ be any $C^1$ function. Then
\begin{equation*}
L_{\rho}\equiv L.
\end{equation*}
\end{lemma}
\begin{proof}
First note that, for any orthogonal matrix $O$, we have
\begin{equation}\label{eq:orthogonalmatrixmollification}
\int_{\R^n} Oz\zeta(z)dz = 0.
\end{equation}
This follows from the fact that
\begin{equation*}
\int_{\R^n}y\exp\left(-\frac{1}{1 - |y|^2}\right)dy = 0,
\end{equation*}
after performing the change of variables $y = Oz$. Now write $L(x) = Ox + v$ for some orthogonal matrix $O$ and some fixed vector $v$. Then we have
\begin{align*}
L_{\rho}(x) &= \int_{\R^n}L(x - \rho(|x|)z)\zeta(z)dz
%%%
\\& = \int_{\R^n}(Ox - \rho(|x|)Oz + v)\zeta(z)dz
%%%
\\&= Ox -\rho(|x|)\int_{\R^n}Oz\zeta(z)dz + v
%%%
\\&= Ox + v = L(x).
\end{align*}
\end{proof}

\begin{lemma}\label{lemma:mollfiedptwisederivest}
If $\phi: D \to \R^n$ is any $C^0$ map defined on a subset $D\subset \R^n$, $L$ is any isometry, and $\rho: [0, \infty) \to [0,\infty)$ is a $C^1$ function, then, for any point $x$ such that $\rho(|x|)\neq 0$ and $B(x, \rho(|x|))\subset D$, we have
\begin{equation*}
|d\phi_{\rho} - dL|\big|_x \leq c(n)|| \phi - L||_{C^0(B(x, \rho(|x|)))}\left| \frac{1}{\rho(|x|)} + \frac{\rho'(|x|)}{\rho(|x|)} \right|
\end{equation*}
\end{lemma}
\begin{proof}
Since $L$ is invariant under mollification by Lemma \ref{lemma:isometriesundermollification}, $dL|_x = dL_{\rho}|_x$. Therefore, we have
\begin{align*}
|d\phi_{\rho} - dL|\big|_x &\leq  \left| \int_{\R^n} (\phi(z) - L(z))\otimes (\nabla \zeta)(\tfrac{x- z}{\rho(|x|)})\left[\frac{I}{\rho(|x|)} - (x - z)\otimes \frac{\rho'(|x|)x}{\rho(|x|)^2|x|}\right]   \rho(|x|)^{-n}dz \right|
\\& + \left| \int_{\R^n} (\phi(z) - L(z))\otimes \zeta(\tfrac{x- z}{\rho(|x|)})\left(-n\rho(|x|)^{-n-1}\frac{\rho'(|x|)x}{|x|}\right)dz \right|
%%%
\\& \leq c(n)||\phi - L||_{C^0(B(x,\rho(|x|)))}\left| \frac{1}{\rho(|x|)} + \frac{\rho'(|x|)}{\rho(|x|)} \right| \int_{\R^n} |\nabla \zeta|(\tfrac{x- z}{\rho(|x|)})\left[ 1+ \frac{|x - z|}{\rho(|x|)}\right]\rho(|x|)^{-n}dz 
\\& +  c(n)||\phi - L||_{C^0(B(x,\rho(|x|)))}\left|\frac{\rho'(|x|)}{\rho(|x|)}\right|\left|\int |\zeta(\tfrac{x- z}{\rho(|x|)})|\rho(|x|)^{-n}dz\right|
%%%
\\& \leq c(|| \zeta||_{W^{1,1}(\R^n)}, n)||\phi - L||_{C^0(B(x,\rho(|x|)))}\left| \frac{1}{\rho(|x|)} + \frac{\rho'(|x|)}{\rho(|x|)} \right|.
\end{align*}
\end{proof}

\begin{corollary}\label{cor:annularbilipschitzestimate}
There exist $c = c(n, b)$ and $\bar \delta = \bar\delta(n) < 1$ such that for all $\delta < \bar \delta$ and all $r>0$ the following is true:

Suppose that $\phi: D\to C$ is a diffeomorphism between some domains $C,D\subset \R^n$ such that  $A(0, .5(1-\delta)r,  (10.5+ .5\delta)r)\subset D$ and, for all $x\in A(0, r, 10r)$, $B(\phi(x), r/4)\subset C$. Suppose that $\phi$ is locally $(1+\delta)$-bilipschitz on $A(0, .5(1-\delta)r,  (10.5+ .5\delta)r)$. Let $\rho: [1, \infty) \to (0, \tfrac{1}{4}]$ be a $C^1$ function with $|\rho'| \leq b$. If $\rho_r(\ell) = \rho(\ell/r)r$ then
 \begin{equation*}
 || \phi_{\rho_r}^*\delta - \delta ||_{C^0(A(0, r, 10r))} \leq c\delta.
 \end{equation*}
\end{corollary}
\begin{proof}
Let $\bar \delta$ and $c$ be as in Lemma \ref{lemma:almostisometryclosetoisowmultloss}. We will adjust $c$ throughout the proof.

We first record a pointwise estimate. Fix some $x\in A(0, r, 10 r)$ so that 
\begin{equation*}
B(x, (1+\delta)\rho_r(|x|))\subset A(0, .5(1-\delta)r,  (10.5+ .5\delta)r) \subset D.
\end{equation*} 
Then Lemma \ref{lemma:almostisometryclosetoisowmultloss} implies that there is a Euclidean isometry $L^{x, r}$ such that 
\begin{equation*}
|| \phi - L^{x, r}||_{C^0(B(x, \rho_r(|x|)/(1+\delta)))} \leq c\delta \rho_r(|x|).
\end{equation*} 

In particular, Lemma \ref{lemma:mollfiedptwisederivest} implies that
\begin{equation*}
| d\phi_{\rho_r}  - dL^{x, r}|\big|_x \leq c\delta \rho_r(|x|)\left(\frac{1}{\rho_r(|x|)} + \frac{\rho_r'(|x|)}{\rho_r(|x|)}\right) = c\delta\left( 1 + \rho'(|x|/r) \right) \leq c\delta(1+b),
\end{equation*}
with $c$ adjusted.

Therefore, for all $x\in A(0, r, 10 r)$ we have
\begin{equation*}
|\phi_{\rho_r}^*\delta - \delta|\big|_x \leq (|d\phi_{\rho_r}|\big|_x + |dL^{x, r}|)| d\phi_{\rho_r}  - dL^{x, r}|\big|_x \leq c\delta,
\end{equation*}
with $c$ adjusted again.
\end{proof}

\bibliographystyle{plain}
\bibliography{C0ADMbib}

\end{document}